\documentclass[fleqn,reqno,11pt,a4paper,final]{amsart}

\usepackage[a4paper,left=20mm,right=20mm,top=20mm,bottom=20mm,marginpar=20mm]{geometry}
\usepackage[english]{babel}
\usepackage{amsmath}
\usepackage{amssymb}
\usepackage{amsthm}
\usepackage{amscd}
\usepackage[utf8]{inputenc}
\usepackage{color}
\usepackage[english=american]{csquotes}
\usepackage[final]{graphicx}
\usepackage{hyperref}
\usepackage{calc}
\usepackage{mathptmx}
\usepackage{mathtools}
\usepackage{textcmds} 
\usepackage{tikz}
\usepackage{graphicx}
\usepackage{float}
\usepackage{stix}
\usepackage{enumitem}
\usepackage{dsfont}
\usepackage{oldgerm}
\usepackage{makecell}

\usepackage{caption}
\usepackage[backend=biber,bibencoding=utf8,firstinits,maxnames=4,style=alphabetic,isbn=false, doi=false, eprint= true, url= false,date=year]{biblatex}
\usepackage[english=american]{csquotes}
\bibliography{literature.bib}

\DeclareRobustCommand{\gobblefour}[5]{}
\newcommand*{\SkipTocEntry}{\addtocontents{toc}{\gobblefour}}

\definecolor{luh-dark-blue}{rgb}{0.0, 0.313, 0.608}

\graphicspath{{../Pictures/}}

\makeatletter
\renewcommand\paragraph{\@startsection{paragraph}{4}{\z@}%
  {1ex \@plus1ex \@minus.2ex}%
  {-1em}%
  {\normalfont\normalsize\bfseries}}
\makeatother

\numberwithin{equation}{section}

\newtheoremstyle{thmlemcorr}{10pt}{10pt}{\itshape}{}{\bfseries}{.}{10pt}{{\thmname{#1}\thmnumber{ #2}\thmnote{ (#3)}}}
\newtheoremstyle{thmlemcorr*}{10pt}{10pt}{\itshape}{}{\bfseries}{.}\newline{{\thmname{#1}\thmnumber{ #2}\thmnote{ (#3)}}}
\newtheoremstyle{remexample}{10pt}{10pt}{}{}{\bfseries}{.}{10pt}{{\thmname{#1}\thmnumber{ #2}\thmnote{ (#3)}}}
\newtheoremstyle{ass}{10pt}{10pt}{}{}{\bfseries}{.}{10pt}{{\thmname{#1}\thmnumber{ A#2}\thmnote{ (#3)}}}

\theoremstyle{thmlemcorr}
\newtheorem{theorem}{Theorem}
\numberwithin{theorem}{section}
\newtheorem{lemma}[theorem]{Lemma}
\newtheorem{corollary}[theorem]{Corollary}
\newtheorem{proposition}[theorem]{Proposition}

\theoremstyle{thmlemcorr*}
\newtheorem*{theorem*}{Theorem}
\newtheorem{lemma*}[theorem]{Lemma}
\newtheorem{corollary*}[theorem]{Corollary}
\newtheorem{proposition*}[theorem]{Proposition}
\newtheorem{problem*}[theorem]{Problem}
\newtheorem{conjecture*}[theorem]{Conjecture}
\newtheorem{definition*}[theorem]{Definition}
\newtheorem{assumption*}[theorem]{Assumption}

\theoremstyle{remexample}
\newtheorem{remark}[theorem]{Remark}

\theoremstyle{ass}

\newcommand{\Ccal}{\mathcal{C}}

\newcommand{\Ecal}{\mathcal{E}}
\newcommand{\Fcal}{\mathcal{F}}
\newcommand{\Gcal}{\mathcal{G}}

\newcommand{\Lcal}{\mathcal{L}}
\newcommand{\Mcal}{\mathcal{M}}
\newcommand{\Ncal}{\mathcal{N}}
\newcommand{\Ocal}{\mathcal{O}}
\newcommand{\Pcal}{\mathcal{P}}
\newcommand{\Qcal}{\mathcal{Q}}
\newcommand{\Rcal}{\mathcal{R}}

\newcommand{\Tcal}{\mathcal{T}}
\newcommand{\Ucal}{\mathcal{U}}
\newcommand{\Vcal}{\mathcal{V}}
\newcommand{\Wcal}{\mathcal{W}}
\newcommand{\Xcal}{\mathcal{X}}
\newcommand{\Ycal}{\mathcal{Y}}
\newcommand{\Zcal}{\mathcal{Z}}

\newcommand{\Lfrak}{\mathfrak{L}}

\newcommand{\Nfrak}{\mathfrak{N}}
\newcommand{\Ofrak}{\mathfrak{O}}
\newcommand{\Pfrak}{\mathfrak{P}}
\newcommand{\Qfrak}{\mathfrak{Q}}
\newcommand{\Rfrak}{\mathfrak{R}}

\newcommand{\Kbf}{\mathbf{K}}

\newcommand{\Ubf}{\mathbf{U}}

\newcommand{\Xbf}{\mathbf{X}}

\renewcommand{\k}{\mathbf{k}}
\newcommand{\ub}{\mathbf{u}}

\newcommand{\e}{\mathbf{e}}
\newcommand{\p}{\mathbf{p}}
\newcommand{\x}{\mathbf{x}}
\newcommand{\X}{\mathbf{X}}
\renewcommand{\a}{\mathbf{a}}
\newcommand{\gammab}{\boldsymbol{\gamma}}
\newcommand{\psib}{\boldsymbol{\psi}}
\newcommand{\Psib}{\boldsymbol{\Psi}}
\newcommand{\nub}{\boldsymbol{\nu}}
\newcommand{\mub}{\boldsymbol{\mu}}
\newcommand{\phib}{\boldsymbol{\varphi}}
\newcommand{\nb}{\mathbf{n}}

\newcommand{\Zgoth}{\mathord{\text{\textgoth{Z}}}}

\renewcommand{\Re}{\operatorname{Re}}
\renewcommand{\Im}{\operatorname{Im}}

\newcommand{\norm}[1]{\|#1\|}

\newcommand{\abs}[1]{|#1|}

\newcommand{\N}{\mathbb{N}}
\newcommand{\R}{\mathbb{R}}
\newcommand{\C}{\mathbb{C}}

\newcommand{\Z}{\mathbb{Z}}

\newcommand{\eps}{\varepsilon}
\newcommand{\lrangle}[1]{\langle #1 \rangle}

\def\div{\nabla\cdot}

\def\XXint#1#2#3{{\setbox0=\hbox{$#1{#2#3}{\int}$}
\vcenter{\hbox{$#2#3$}}\kern-.5\wd0}}

\renewcommand{\eps}{\varepsilon}
\renewcommand{\phi}{\varphi}

\begin{document}


\title[]{Fast-moving pattern interfaces close to a Turing instability in an asymptotic model for the three-dimensional Bénard--Marangoni problem}
\author{Bastian Hilder}
\address{\textit{Bastian Hilder:}  Department of Mathematics, Technische Universität München, Boltzmannstraße 3, 85748 Garching b.\ München, Germany}
\email{bastian.hilder@tum.de}

\author{Jonas Jansen}
\address{\textit{Jonas Jansen:}   Fraunhofer Institute for Algorithms and Scientific Computing SCAI, Schloss Birlinghoven, 53757 Sankt Augustin, Germany}
\email{jonas.jansen@scai.fraunhofer.de}

\begin{abstract}
    We study the bifurcation of planar patterns and fast-moving pattern interfaces in an asymptotic long-wave model for the three-dimensional Bénard–Marangoni problem, which is close to a Turing instability. We derive the model from the full free-boundary Bénard–Marangoni problem for a thin liquid film on a heated substrate of low thermal conductivity via a lubrication approximation. This yields a quasilinear, fully coupled, mixed-order degenerate-parabolic system for the film height and temperature. As the Marangoni number $M$ increases beyond a critical value $M^*$, the pure conduction state destabilises via a Turing(–Hopf) instability. Close to this critical value, we formally derive a system of amplitude equations which govern the slow modulation dynamics of square or hexagonal patterns. Using center manifold theory, we then study the bifurcation of square and hexagonal planar patterns. Finally, we construct planar fast-moving modulating travelling front solutions that model the transition between two planar patterns. The proof uses a spatial dynamics formulation and a center manifold reduction to a finite-dimensional invariant manifold, where modulating fronts appear as heteroclinic orbits. These modulating fronts facilitate a possible mechanism for pattern formation, as previously observed in experiments.
\end{abstract}
\vspace{4pt}

\maketitle


\setcounter{tocdepth}{1}
\tableofcontents

\section{Introduction}

In two seminal papers in 1900 and 1901 \cite{bénard1900,benard1901}, Henri Bénard studied thin fluid layers resting on heated metal plates. He noticed the emergence of a regular polygonal pattern of typically hexagonal shape, raised in the center and a convection current within each cell. Block \cite{block1956} and Pearson \cite{pearson1958} showed that these patterns form due to temperature variations of the surface tension, even for ultrathin films. This thermocapillary instability is physically explained by the surface tension being, for most fluids, a decreasing function of temperature. Due to the similarity of the surface-tension deficit to the Marangoni effect, the problem of a thin viscous fluid resting on a heated plane is known as the \emph{Bénard--Marangoni problem}.

Since the initial experiments by Bénard, the problem has been revisited, and two main features have been found through experimental studies. First, the planar patterns already found by Bénard have been rediscovered by several other authors. In addition to the hexagons found in \cite{koschmieder1974,schatz1995}, the formation of planar square patterns has been observed in \cite{nitschke1995,eckert1998}. Second, it has been observed that the fluid film can exhibit spontaneous film rupture, see \cite{vanhook1995,vanhook1997}. This paper deals with the former case.

Besides the observation of stationary planar patterns, also their formation from the spatially homogeneous purely conductive state with a flat surface has received much attention. Already, Bénard found that planar patterns do not arise everywhere at once but typically form in the wake of an invading front, see \cite[Fig.~12]{bénard1900}. Later, also the transition between two planar patterns through a planar front has been observed experimentally, see \cite{pantaloni1979} for an invasion of roll waves by polygonal cells and \cite{nitschke1995,eckert1998} for an invasion of hexagonal cells by square cells. For an overview of the experimental literature, we refer the reader to \cite{colinet2001}, while high-resolution images of experimental results can be found in \cite[pg.~83]{vandyke2008}.

Mathematically, these planar patterns are strongly related to similar patterns in other applications, such as biology. A first explanation of how these structures arise from a spatially homogeneous steady state has been given by Turing \cite{turing1952}. He observed that a constant steady state in a reaction-diffusion system with different diffusion coefficients can become unstable at a fixed finite Fourier wave number when varying the diffusivity of one component. This instability is known as a diffusion-driven or \emph{Turing instability} and typically leads to the bifurcation of spatially periodic solutions. While this instability can also be found in the full mathematical model of the Bénard-Marangoni problem, see \cite{samoilova2014}, recently, a thermocapillary thin-film model exhibiting the same instability has been formally derived from the full free-boundary problem using a long-wave (or lubrication) approximation. In this model, the solid-fluid interface is assumed to have a poorly conducting thermal contact, and thus, only the heat flux through the interface is fixed, see \cite{shklyaev2012} and Section \ref{sec:modelling} for more details. In the parameter regime of small Biot number and large capillary number with constant product, their formal expansion leads to a coupled system of partial differential equations for the film height $h = h(t,x,y)$ and the mean temperature $\theta = \theta(t,x,y)$, which reads as
\begin{equation}\label{eq:thin-film-equation}
    \hspace{-0.5cm}\left\{
    \begin{array}{rcl}
        \partial_t h + \div\bigl(\frac{h^3}{3}(\nabla\Delta h-g\nabla h) + M\frac{h^2}{2} \nabla(h-\theta)\bigr) & = & 0  \\
        h\partial_t \theta - \div (h\nabla\theta) + \frac{1}{2}|\nabla h|^2 + \beta(\theta-h) - j\cdot \nabla(\theta-h) + \div\bigl(\frac{h^4}{8}(\nabla \Delta h-g\nabla h) + M\frac{h^3}{6} \nabla(h-\theta)\bigr) & = & 0 
    \end{array}
    \right.
\end{equation}
with nonlinear flux \(j=\frac{h^3}{3}(\nabla \Delta h-g\nabla h) + M\frac{h^2}{2} \nabla(h-\theta)\). Here, $g$ is a gravitational constant, $\beta$ is the rescaled Biot number, and $M$ is the Marangoni number, which is proportional to the prescribed heat flux across the bottom interface. Throughout this paper, $g$ and $\beta$ are treated as fixed parameters, while $M$ is the main bifurcation parameter. This is in line with experimental treatments since the temperature flux can be easily controlled.

Beyond the physical importance, the thin-film system \eqref{eq:thin-film-equation} is also interesting from a mathematical perspective since it is a quasilinear, fully-coupled mixed-order system and also degenerate-parabolic in the film height. Moreover, the analysis of pattern-forming systems has hitherto often been restricted to semilinear toy models. Thus, the rigorous analysis of the asymptotic model \eqref{eq:thin-film-equation} demonstrates the applicability of dynamical-systems methods to more complex models arising in free-boundary flows in order to study pattern formation and pattern interfaces. In particular, this is a first step towards the treatment of the original physical free-boundary problem.

The aim of this paper is a rigorous analysis of the existence of stationary patterns and fast-moving pattern interfaces in equation \eqref{eq:thin-film-equation}. For Marangoni numbers close to a critical value, we establish the bifurcation of periodic patterns with both $(D_4)$ and $(D_6)$ symmetry from the purely conductive state. This includes roll waves and square and hexagonal patterns. Then, we establish the existence of fast-moving planar fronts, which asymptotically connect different stationary patterns. In particular, this provides a mathematical description of the invasion phenomena already observed experimentally by Bénard and others.

\subsection{Main results of the paper}
We summarise the main results and techniques of this paper. We show that

\begin{itemize}
    \item there is an open domain $\Omega_m$ in the $(\beta,g)$-parameter space such that for each $(\beta,g) \in \Omega_m$ there exists a critical Marangoni number $M = M_m^*$ and a Fourier wave number $k = k_m^*$ such that the pure conduction state $(\bar{h},\bar{\theta}) = (1,1)$ undergoes a Turing instability at $M_m^*$ with wave numbers $|\k| = k_m^*$, $\k\in \R^2$;
    \item at $M_m^*$, spatially periodic solutions with either square ($D_4$) symmetry or hexagonal ($D_6$) symmetry bifurcate from the pure conduction state, see Figure \ref{fig:planar-patterns} and Theorems \ref{thm:square-patterns} and \ref{thm:hex-patterns};
    \item formally, slow amplitude modulations of square and hexagonal patterns of the thin-film model \eqref{eq:thin-film-equation} close to the pure conduction state and the critical Marangoni number $M_m^*$ are governed by a degenerate Ginzburg–Landau system with an additional conservation law, see equations \eqref{eq:amplitude-equations-square} and \eqref{eq:amplitude-equations-hex};
    \item in appropriate subsets of $\Omega_m$, there exist fast-moving pattern interfaces describing the invasion of certain patterns by other patterns, see Figures \ref{fig:Modfronts-square}, \ref{fig:Modfronts-hex} and \ref{fig:Modfronts-hex-two} and Theorems \ref{thm:modulating-fronts-square} and \ref{thm:modulating-fronts-hex}. To leading order, these solutions take the form of slow amplitude modulations of the critical Fourier modes.
\end{itemize}

\begin{figure}[H]
    \centering
    \includegraphics[width=0.3\linewidth]{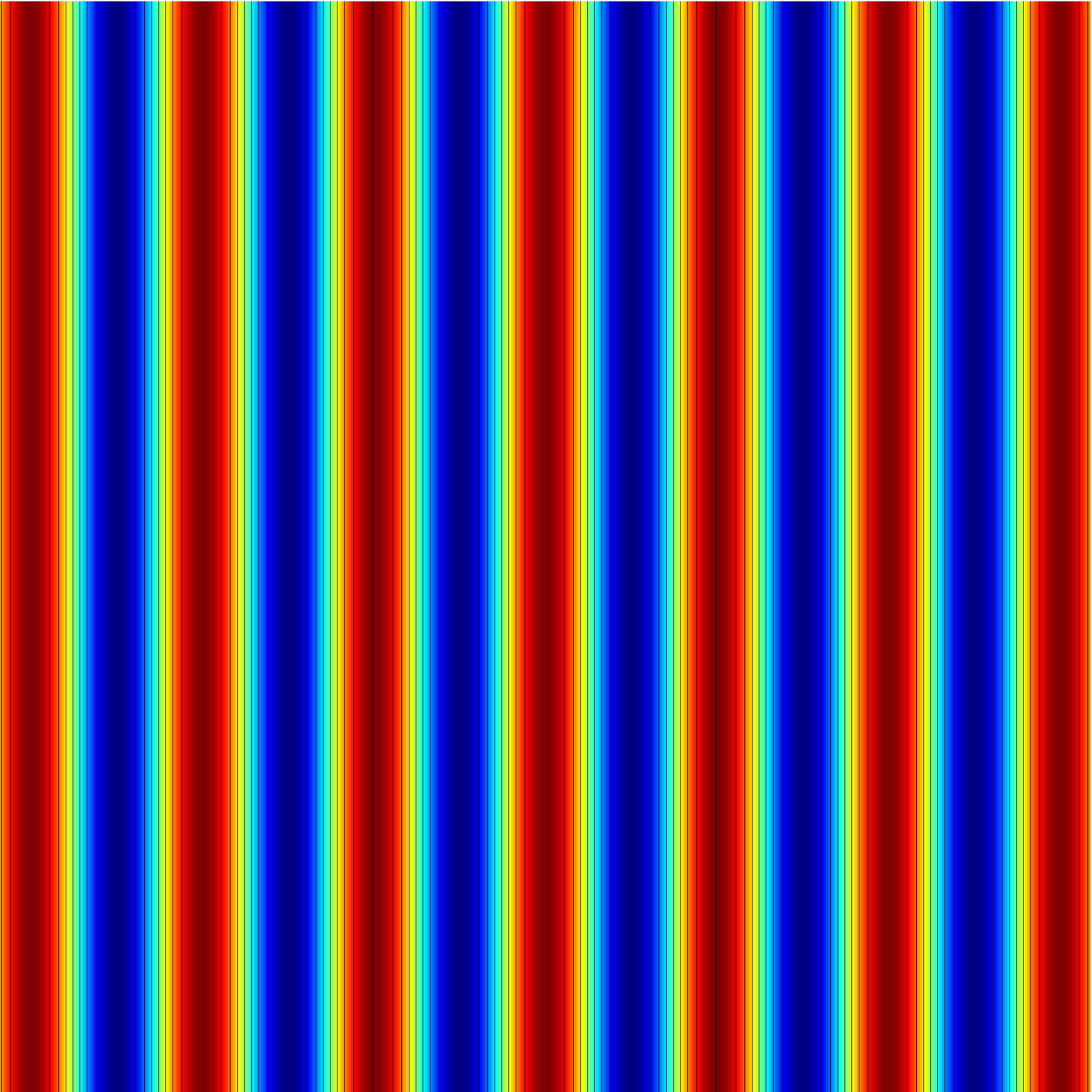}
    \hfill
    \includegraphics[width=0.3\linewidth]{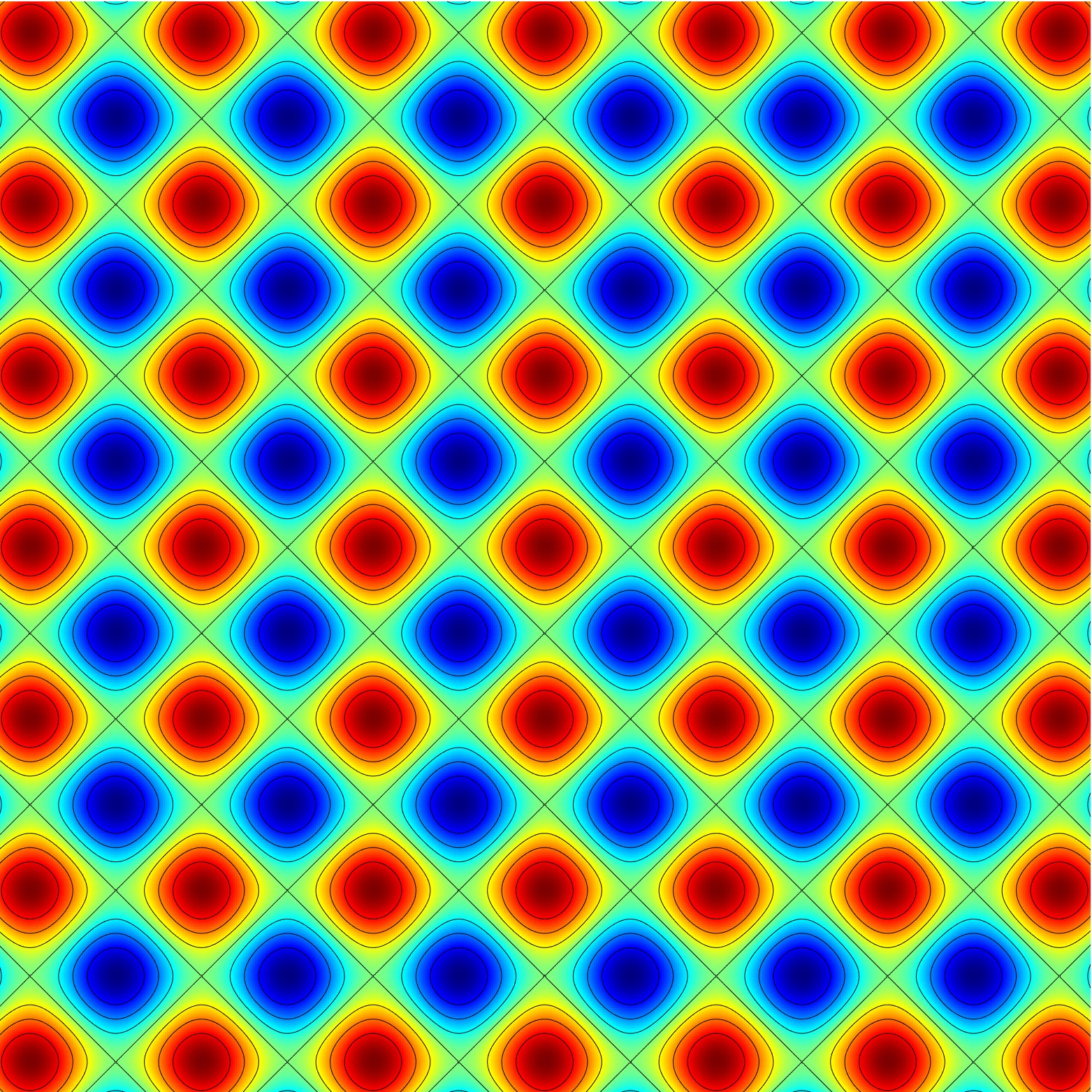}
    \hfill
    \includegraphics[width=0.3\linewidth]{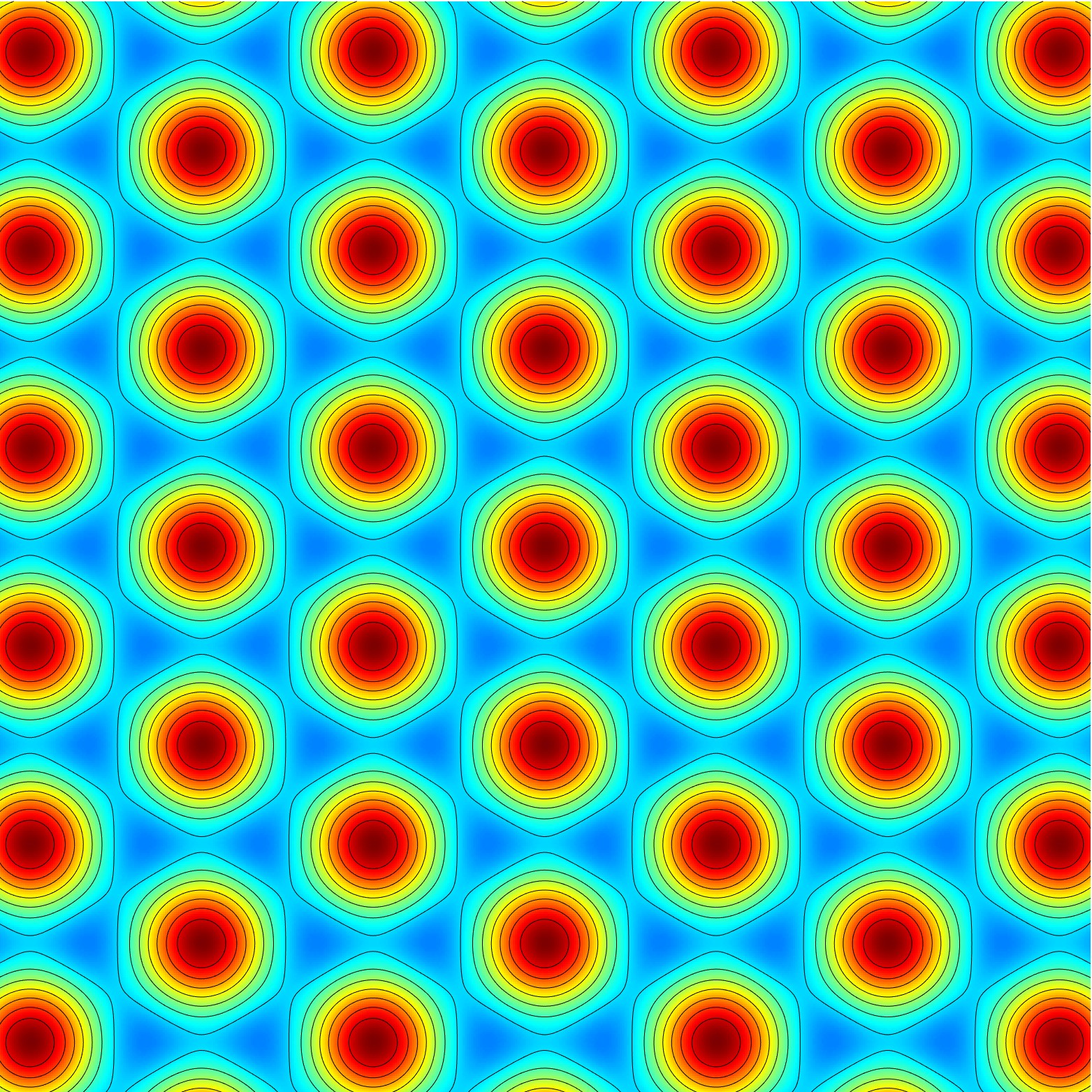}
    \caption{Planar roll waves, square and hexagonal patterns in \eqref{eq:thin-film-equation}. The different colors indicate the different temperatures $\theta$ in the domain, where red indicates a higher temperature and blue indicates a lower one. Additionally, the level sets of the surface height $h$ are plotted as black lines.}
    \label{fig:planar-patterns}
\end{figure}

\begin{remark}
    Throughout the paper and without loss of generality, we study only solutions, which are close to the constant stationary solution $(\bar{h},\bar{\theta}) = (1,1)$.
\end{remark}

\noindent\textbf{Bifurcation of planar patterns. }
To study the existence of planar patterns, we consider the linear instability of the pure conduction state $(\bar{h},\bar{\theta}) = (1,1)$. For this, we derive the dispersion relation of the pure conduction state by studying the Fourier symbol of the linearisation $\hat{\Lcal}_M(\k)$ of \eqref{eq:thin-film-equation} with $\k \in \C^2$. Using the rotation invariance of \eqref{eq:thin-film-equation}, there are two curves of eigenvalues $\lambda_\pm(k;M)$, $k = |\k|$. It turns out that there exists an open domain of the $(\beta,g)$ parameter space such that the pure conduction steady state undergoes a Turing instability. This means that there exists a critical Marangoni number $M_m^* > 0$ and a circle of critical Fourier wave numbers with $|\k| = k_m^*>0$ such that $\lambda_+(\cdot;M_m^*)$ has a quadratic root at $k = k_m^*$. Additionally, for $M > M_m^*$, there is a neighborhood of the circle $|\k| = k_m^*$ where $\lambda_+(\cdot;M)$ has positive real part. Finally, we point out the fact that the $h$-equation in \eqref{eq:thin-film-equation} is a conservation law yields that $\lambda_+(0;M) = 0$ for all $M > 0$ with a quadratic touching.

Close to this Turing instability, we then analyse the existence of planar patterns with either square ($D_4$) symmetry or hexagonal ($D_6$) symmetry using center manifold reduction. Therefore, we study \eqref{eq:thin-film-equation} on function spaces that are periodic with respect to a square and hexagonal lattice, that is functions of the form
\begin{equation*}
    \Ucal = \sum_{k\in \Gamma} \a_\k e^{i\k\cdot \x}, \quad \a_k = \overline{\a_{-\k}} \in \C^2,
\end{equation*}
where $\Gamma$ is the corresponding dual lattice, which is generated by $\k_1 = k_m^*(1,0)$ and $\k_2 = k_m^*(0,1)$ for the square lattice and by $\k_1 = k_m^* (1,0)$, $\k_2 = \tfrac{k_m^*}{2}(-1,\sqrt{3})$ and $\k_3 = -\tfrac{k_m^*}{2}(1,\sqrt{3})$ for the hexagonal lattice, see Figure \ref{fig:Fourier-lattice}. We show that on these lattices, we can apply standard center manifold theory \cite[Thm.~2.20]{haragus2011a} to obtain stationary solutions of the form
\begin{equation*}
    \begin{pmatrix}
        h \\ \theta
    \end{pmatrix} = \begin{pmatrix}
        1 \\ 1
    \end{pmatrix} + \sum_{j = 1}^N A_j(T) e^{i \k_j \cdot \x} \phib_+(\k_j) + c.c. + h.o.t.,
\end{equation*}
where $c.c.$ and $h.o.t.$ denote complex conjugated and higher-order terms, respectively, and $\phib_+(\k_j)$ is the eigenvector associated to the eigenvalue $\lambda_+(\k_j;M_m^*) = 0$ of $\hat{\Lcal}_{M^*}(\k_j)$ and $N\in \{2,3\}$ is the number of waves vectors generating the lattice. Here, the $A_j$ are amplitude modulations of the critical Fourier modes, which can evolve on the slow timescale $T = \mu t$ with $M - M_m^* = \mu M_0$ and $\mu$ sufficiently small. We then show that the amplitude modulations on the square lattice satisfy

\begin{equation*}
    \begin{split}
        \mu \partial_T A_1 &= \mu M_0 \kappa A_1 + (K_0 \abs{A_1}^2 + K_1 \abs{A_2}^2) A_1+ \Ocal\bigl(|(A_1,\bar{A}_1,A_2,\bar{A_2},\mu)|^4\bigr), \\
        \mu \partial_T A_2 &= \mu M_0 \kappa A_2 + (K_1 \abs{A_1}^2 + K_0 \abs{A_2}^2) A_2+ \Ocal\bigl(|(A_1,\bar{A}_1,A_2,\bar{A_2},\mu)|^4\bigr),
    \end{split}
\end{equation*}
whereas the evolution of the hexagonal lattice is given by
\begin{equation*}
    \begin{split}
        \mu \partial_T A_1 &= \mu M_0 \kappa A_1 + N \bar{A}_2 \bar{A}_3 + (K_0 \abs{A_1}^2 + K_2 (\abs{A_2}^2 + \abs{A_3}^2) A_1 + \Ocal\bigl(|(A_1,\bar{A}_1,A_2,\bar{A_2},A_3,\bar{A}_3,\mu)|^4\bigr), \\
        \mu \partial_T A_2 &= \mu M_0 \kappa A_2 + N \bar{A}_1 \bar{A}_3 + (K_0 \abs{A_2}^2 + K_2 (\abs{A_1}^2 + \abs{A_3}^2) A_2 + \Ocal\bigl(|(A_1,\bar{A}_1,A_2,\bar{A_2},A_3,\bar{A}_3,\mu)|^4\bigr), \\
        \mu \partial_T A_3 &= \mu M_0 \kappa A_3 + N \bar{A}_1 \bar{A}_2 + (K_0 \abs{A_3}^2 + K_2 (\abs{A_1}^1 + \abs{A_2}^2) A_3 + \Ocal\bigl(|(A_1,\bar{A}_1,A_2,\bar{A_2},A_3,\bar{A}_3,\mu)|^4\bigr).
    \end{split}
\end{equation*}
The coefficients $\kappa$, $N$ and $K_j$ for $j = 0,1,2$ depend on $(\beta,g) \in \Omega_m$ and can be computed explicity, see \eqref{eq:expansion-linear-coeff}, \eqref{eq:quadratic-coeff}, \eqref{eq:self-interaction-coeff}, \eqref{eq:cross-interaction-coeff-square} and \eqref{eq:cross-interaction-coeff-hex} and also the Supplementary Material. We then solve these systems for fixed points in the invariant subsets $\{A_1 \in \R, A_2 = 0\}$ and $\{A_1 \in \R, A_2 = A_3 = 0\}$ corresponding to roll waves in \eqref{eq:thin-film-equation}, and $\{A_1 = A_2 \in \R\}$ and $\{A_1 = A_2=A_3 \in \R\}$ corresponding to square and hexagonal patterns in \eqref{eq:thin-film-equation}, respectively. We summarise the result for the square lattice in the following theorem, see also Theorem \ref{thm:square-patterns} and Figure \ref{fig:planar-patterns}.

\begin{theorem}[Stationary patterns on square lattice]
    There exists a $\mu_0 > 0$ such that for all $0 < \mu < \mu_0$, the thin-film system \eqref{eq:thin-film-equation} has roll waves for $M_0K_0 < 0$ and square patterns for $M_0(K_0+K_1) < 0$. 
\end{theorem}

For the hexagonal lattice, the bifurcation picture is more complicated due to the quadratic terms, which arise from the resonance of the hexagonal lattice, that is $\k_1 + \k_2 + \k_3 = 0$. Therefore, in order to obtain persistent solutions to the reduced system, we need to distinguish two different cases for the quadratic coefficient $N$. In the first case, $N$ is non-zero but of finite size, whereas in the second case, $N$ is small with respect to $\mu$ and satisfies $N = \sqrt{\mu} N_0$ with $N_0 = \Ocal(1)$ in $\mu$. We point out that there is an explicit curve $\beta(g)$ for $g \in (0,18)$ such that $(\beta(g),g) \subset \Omega_m$ and $N(\beta(g),g) = 0$, see \eqref{eq:beta-von-g} and Figure \ref{fig:coeffs-hex}. We summarize our findings in the following theorem, see also Theorem \ref{thm:hex-patterns} and Figures \ref{fig:planar-patterns} and \ref{fig:bifurcation-diag-hex}.

\begin{theorem}[Stationary patterns on hexagonal lattice]
    There exists a $\mu_0 > 0$ such that for all $0 < \mu < \mu_0$, the thin-film system \eqref{eq:thin-film-equation} has roll waves for $M_0K_0 < 0$. If $N \neq 0$, the system \eqref{eq:thin-film-equation} has hexagonal patterns with amplitude $A = \tfrac{\mu M_0 \kappa}{N}$. If additionally $N = \sqrt{\mu}N_0$, $K_0 + 2K_2 \neq 0$ and $N_0^2 - 4M_0\kappa (K_0+2K_2) > 0$, the system \eqref{eq:thin-film-equation} has hexagonal patterns with amplitudes $A_\pm = \sqrt{\mu}\tfrac{-N_0 \pm \sqrt{N_0^2 - 4M_0 \kappa (K_0 + 2K_2)}}{2(K_0 + 2K_2)}$.
\end{theorem}

\noindent\textbf{Formal amplitude equations. }
The dynamics of pattern-forming systems close to the onset of instability is typically governed by a set of amplitude or modulation equations. These describe the slow modulation of characteristic features of the base pattern such as amplitude, wave number or mean value. The theory of amplitude equations extends the invariant manifold theory used to construct the planar patterns above in the following sense. On the center manifold, the dynamics is fully described by the amplitudes of the (finitely many) critical Fourier modes. This idea is also reflected in modulation theory, where one typically finds that the amplitude dynamics close to the onset of instability can be reduced to a closed system of finitely many PDEs for the amplitude modulations of the critical Fourier modes. We point out that after formally deriving the amplitude equations, one still needs to justify rigorously that the dynamics of the full system is actually determined by the amplitude equations on a sufficiently long time interval, which is a highly nontrivial task, and we refer to \cite{schneider2017} for details. 

Here, we restrict to a formal derivation of the amplitude equations for the thin-film system \eqref{eq:thin-film-equation}. For this, we make a multiscale ansatz to derive evolution equations for slow amplitude modulations of the planar base patterns on the square and hexagonal lattice, respectively. Specifically, defining $\eps$ via $M-M_m^* = \eps^2 M_0$, the ansatz reads as
\begin{equation*}
    \begin{pmatrix}
        h \\ \theta
    \end{pmatrix} = \begin{pmatrix}
        1 \\ 1
    \end{pmatrix} + \eps\sum_{j = 1}^N A_j(T,\Xbf) e^{i \k_j \cdot \x} \phib_+(\k_j)  + c.c. +  \eps^2 A_0(T,\Xbf) \phib_+(0) + h.o.t.,
\end{equation*}
where $T=\eps^2t$ and $\Xbf = \eps \x$. We point out that the ansatz is slightly different compared to the expansion for the planar patterns since the scaling of the amplitude modulations is already explicitly included. Inserting this ansatz into \eqref{eq:thin-film-equation} and equating different powers of $\eps$ to zero then yields the following equations. In the case of the square lattice, the amplitude system is given by
\begin{equation*}
    \begin{split}
        \partial_T A_1 &= -\dfrac{\lambda''_+(k_m^*)}{2 (k_m^*)^2} \partial_X^2 A_1 + M_0 \kappa A_1 + K_c A_0 A_1 + K_0 A_1 \abs{A_1}^2 + K_1 A_1 \abs{A_2}^2, \\
        \partial_T A_2 &= -\dfrac{\lambda''_+(k_m^*)}{2 (k_m^*)^2} \partial_Y^2 A_2 + M_0 \kappa A_2 + K_c A_0 A_2 + K_0 A_2 \abs{A_2}^2 +  K_1 A_2 \abs{A_1}^2, \\
        \partial_T A_0 &= -\dfrac{1}{2} \lambda_+^{\prime\prime}(0) \Delta A_0 + \div \bigl(p_c(\k_1\k_1^T) \nabla \abs{A_1}^2 + p_c(\k_2 \k_2^T) \nabla\abs{A_2}^2\bigr).
    \end{split}
\end{equation*}
For the hexagonal lattice, the amplitude equations are given by
\begin{equation*}
    \begin{split}
        \partial_T A_1 &= -\dfrac{\lambda''_+(k_m^*)}{2 (k_m^*)^2} \partial_X^2 A_1 + M_0 \kappa A_1 + K_c A_0 A_1 + \dfrac{N}{\eps}\bar{A}_2\bar{A}_3 + K_0 A_1 \abs{A_1}^2 + K_2 A_1 (\abs{A_2}^2 + \abs{A_3}^2), \\
        \partial_T A_2 &= -\dfrac{\lambda''_+(k_m^*)}{2 (k_m^*)^2} \left(-\dfrac{1}{2} \partial_X + \dfrac{\sqrt{3}}{2}\partial_Y\right)^2 A_2 + M_0 \kappa A_2 + K_c A_0 A_2 + \dfrac{N}{\eps}\bar{A}_1\bar{A}_3  + K_0 A_2 \abs{A_2}^2 \\
        & \quad +  K_2 A_2 (\abs{A_1}^2 + \abs{A_3}^2), \\
        \partial_T A_3 &= \dfrac{\lambda''_+(k_m^*)}{2 (k_m^*)^2} \left(\dfrac{1}{2} \partial_X + \dfrac{\sqrt{3}}{2}\partial_Y\right)^2 A_3 + M_0 \kappa A_3 + K_c A_0 A_3 + \dfrac{N}{\eps}\bar{A}_1\bar{A}_2  + K_0 A_3 \abs{A_3}^2 \\
        & \quad +  K_2 A_3 (\abs{A_1}^2 + \abs{A_2}^2), \\
        \partial_T A_0 &= -\dfrac{1}{2} \lambda_+^{\prime\prime}(0) \Delta A_0 + \div \bigl(p_c(\k_1\k_1^T) \nabla \abs{A_1}^2 + p_c(\k_2 \k_2^T) \nabla\abs{A_2}^2 + p_c(\k_3 \k_3^T) \nabla \abs{A_3}^2\bigr).
    \end{split}
\end{equation*}
In particular, we point out that the coefficients are the same as in the reduced equations on the center manifold for the construction of planar patterns. Moreover, we recover the same reduced equations appearing in the construction of planar patterns above by making the ansatz $A_j(T,\Xbf) = A_j(T)$ and $A_0 = 0$. Furthermore, we also formally recover the reduced equations in the construction of the modulating fronts below by searching for fast-travelling fronts with the ansatz $A_j(T,\Xbf) = A_j\bigl(\eps(X - \eps^{-1}cT)\bigr)$ with $\Xbf = (X,Y)$.

Finally, we highlight that since the conservation law is of fourth order, we find a new structure of the nonlinear terms appearing in the conservation law. That is, instead of a pure Laplace operator, we obtain a potentially non-elliptic second-order operator in the $A_0$-equations.

\vspace{6pt}

\noindent\textbf{Pattern interfaces. }
As highlighted above, experimental observations indicate that planar patterns arise in the wake of a (radial) invading front. To keep the problem mathematically tractable, we study the existence of planar fronts in this paper, which describe the transition from one planar pattern (including the trivial pure conduction state) to a different pattern. These fronts are also referred to as moving pattern interfaces and can be modelled via modulating fronts, which are solutions of the form
\begin{equation}\label{eq:pattern-interface-ansatz}
    \Ucal(t,\x) = \Vcal(x-ct,\x) = \Vcal(\xi,\p), \quad \x=(x,y).
\end{equation}
Here, $\xi = x-ct$ is a co-moving frame with fixed velocity $c > 0$, which reflects the moving interface. We point out that, for simplicity, we choose this co-moving frame to be in $x$-direction only; therefore, the front is a planar front purely in $x$-direction. It turns out that any other choice of direction leads to the same results. Additionally, we assume that $\Vcal$ is periodic with respect to its second variable $\p = \x$, where periodic here either refers to square ($D_4$) symmetry, that is,
\begin{equation*}
    \Vcal(\xi,x,y) = \Vcal(\xi,x+\tfrac{2\pi}{k_m^*},y) = \Vcal(\xi,x,y+\tfrac{2\pi}{k_m^*}),
\end{equation*}
or hexagonal ($D_6$) symmetry, that is,
\begin{equation*}
    \Vcal(\xi,x,y) = \Vcal(\xi,x+\tfrac{2\pi}{k_m^*},y) = \Vcal(\xi,x-\tfrac{\pi}{k_m^*},y+\tfrac{\sqrt{3}\pi}{k_m^*}) = \Vcal(\xi,x-\tfrac{\pi}{k_m^*},y-\tfrac{\sqrt{3}\pi}{k_m^*}).
\end{equation*}
We call $\Vcal$ a pattern interface or modulating front describing an invasion of the stationary pattern $\Ucal_2$ by the stationary pattern $\Ucal_1$ if $\Vcal$ has the asymptotic boundary conditions
\begin{equation*}
    \lim_{\xi \rightarrow -\infty} \Vcal(\xi,\p) = \Ucal_1(\p), \text{ and } \lim_{\xi \rightarrow +\infty} \Vcal(\xi,\p) = \Ucal_2(\p).
\end{equation*}

To construct these solutions, we follow the strategy of \cite{doelman2003}. However, we face new difficulties due to the quasilinearity of \eqref{eq:thin-film-equation} and its conservation law structure as outlined below. The main idea is to insert the ansatz \eqref{eq:pattern-interface-ansatz} into the thin-film system \eqref{eq:thin-film-equation} and to reformulate the resulting systems as a spatial dynamics system by treating the spatial variable $\xi$ as time. Compared to \cite{doelman2003} this step is more complicated for $\eqref{eq:thin-film-equation}$ since the problem is quasilinear. This requires an extra step to extract the highest-order derivatives and results in an equation of the form
\begin{equation*}
    \partial_\xi \Wcal = \Lfrak_{M_m^*} \Wcal + \Rfrak(\Wcal;M)
\end{equation*}
with $\Wcal = (\Vcal_1,\partial_{\xi}\Vcal_1,\partial_{\xi}^2\Vcal_1,\partial_{\xi}^3\Vcal_1,\Vcal_2,\partial_{\xi}\Vcal_2)$. This approach leads to a, typically ill-posed, infinite-dimensional dynamical system. However, it turns out that the problem can be reduced to a finite-dimensional problem via a center manifold reduction. This requires a spectral analysis of the linear operator $\Lfrak_{M_m^*}$ of the infinite-dimensional spatial dynamics system to show that the spectrum splits into a central part $\sigma_0 \subset i\R$, containing finitely many purely imaginary eigenvalues with finite multiplicity, and a hyperbolic part $\sigma_h$, which is separated from the imaginary axis by a spectral gap. The linear operator can be written as the direct sum of complex-valued $6\times 6$ matrices. However, in contrast to \cite{doelman2003}, the two equations in \eqref{eq:thin-film-equation} are fully coupled on the linear level, and therefore a precise spectral analysis of these $6\times6$ matrices is infeasible. Instead, we obtain the necessary spectral properties of the linear operator by relating the spatial dynamics system back to the physical system \eqref{eq:thin-film-equation} after inserting the modulating front ansatz \eqref{eq:pattern-interface-ansatz}. Using this method, which has been used for example by \cite{afendikov1995}, gives the desired spectral properties from the fact that the pure conduction state destabilises via a Turing instability at a critical Marangoni number $M_m^*$. In particular, we use here that the pattern interfaces are fast-moving, which means that the front velocity $c$ does not decay to zero as $M$ goes to $M_m^*$. It turns out that the construction of slow-moving modulating fronts with speed $c = \Ocal(\eps)$ is significantly more involved since there is no gap between the central part and the hyperbolic part of the spectrum at $M = M_m^*$. We refer to e.g.~\cite{hilder2020b} for the construction in the case of a phenomenological model, and we leave the construction of slow modulating fronts for the thin-film system \eqref{eq:thin-film-equation} for future research.

After establishing the spectral properties, we obtain the existence of a finite-dimensional invariant manifold $\Mcal_0 = \{\Wcal_0 + \Psib(\Wcal_0;M) : (\Wcal_0,M) \in \Ofrak\}$ through center manifold theory, see Theorem \ref{thm:center-manifold-fronts}. While \cite{doelman2003} can make this reduction directly for the spatial dynamics system, we perform the reduction for the physical system by converting the coordinates on the center manifold from the spatial dynamics system to the physical system, see \eqref{eq:Tcal-operator}. This has two advantages; first, by performing the calculations on the level of the physical system, we are able to identify the coefficients of the reduced equations in a direct manner. Second, by working directly with the physical system, we can use the conservation-law structure of the first equation in \eqref{eq:thin-film-equation}. This structure is hidden in the spatial dynamics system since the resolution of the quasilinearity requires giving up the divergence structure. However, while the conversion of the central part $\Wcal_0$ to a central part $\Vcal_0$ for the physical system is straightforward, it turns out that converting the reduction function $\Psib$ defining the center manifold to a reduction function for the physical system $\Tcal \Psib$, see \eqref{eq:Tcal-operator}, is subtle since we need to ensure that the contributions of $\Tcal\Psib$ in the relevant eigenspaces spanned by $\phib_+(\k_j)$ and $\phib_+(0)$ are of higher order, see Lemma \ref{lem:reduction-expansion}. While this is clear for the $\phib_+(\k_j)$-mode, establishing this for the $\phib_+(0)$-mode requires the use of the conservation-law structure. Collecting these results and defining the scaling parameter $\eps$ by $\eps^2 M_0 = M - M_m^*$, we find that the solutions on the center manifold have the form
\begin{equation*}
    \begin{pmatrix}
        h \\ \theta
    \end{pmatrix} = \begin{pmatrix}
        1 \\ 1
    \end{pmatrix} + \eps\sum_{j = 1}^N A_j(\Xi) e^{i \k_j \cdot \p} \phib_+(\k_j)  + c.c. +  \eps^2 A_0(\Xi) \phib_+(0) + h.o.t.,
\end{equation*}
where the amplitude modulations $A_j(\Xi) \in \C$, $j = 1,\dots,N$ and $A_0(\Xi) \in \R$ depend on the slow spatial variable $\Xi = \eps^2 (x-ct)$ and are, to leading order, solutions of the evolution system 
\begin{equation*}
    \begin{split}
        0 &= c\partial_\Xi A_1 + M_0 \kappa A_1 + K_c A_0 A_1 + K_0 A_1 \abs{A_1}^2 + K_1 A_1 \abs{A_2}^2, \\
        0 & = c \partial_\Xi A_2 + M_0 \kappa A_2 + K_c A_0 A_2 + K_0 A_2 \abs{A_2}^2 +  K_1 A_2 \abs{A_1}^2, \\
        0 & = c \partial_{\Xi} A_0
    \end{split}
\end{equation*}
for the square lattice and
\begin{equation*}
    \begin{split}
        0 &= c\partial_\Xi A_1 + M_0 \kappa A_1 + K_c A_0 A_1 + \dfrac{N}{\eps}\bar{A}_2\bar{A}_3 + K_0 A_1 \abs{A_1}^2 + K_2 A_1 (\abs{A_2}^2 + \abs{A_3}^2), \\
        0 &= c\partial_\Xi A_2 + M_0 \kappa A_2 + K_c A_0 A_2 + \dfrac{N}{\eps}\bar{A}_1\bar{A}_3  + K_0 A_2 \abs{A_2}^2 +  K_2 A_2 (\abs{A_1}^2 + \abs{A_3}^2), \\
        0 &= c\partial_\Xi A_3 + M_0 \kappa A_3 + K_c A_0 A_3 + \dfrac{N}{\eps}\bar{A}_1\bar{A}_2  + K_0 A_3 \abs{A_3}^2 + K_2 A_3 (\abs{A_1}^2 + \abs{A_2}^2), \\
        0 &= c\partial_{\Xi} A_0
    \end{split}
\end{equation*}
on the hexagonal lattice. We point out that the coefficients are the same as in the reduced equation in the construction of the stationary patterns. As mentioned above, this system also rigorously recovers the ODE system, which can be derived from the formal amplitude equations for fast-travelling fronts, see also Section \ref{sec:dynamics-amplitude-equations}.

After deriving the reduced systems and setting $A_0$ to zero, it remains to find persistent heteroclinic orbits in the reduced system, which correspond to modulating front solutions in the full system \eqref{eq:thin-film-equation}. Here, we restrict to appropriate invariant subspaces and analyse the dynamics of the leading-order system using phase-plane analysis. From this phase-plane analysis and after taking care of higher-order terms in the $A_0$-equation we also obtain persistence of the heteroclinic orbits. We summarise the result for the square lattice in the following theorem, see also Theorem \ref{thm:modulating-fronts-square} and Figures \ref{fig:phase-diags-square1} and \ref{fig:phase-diags-square2} for the phase-plane diagrams and Figure \ref{fig:Modfronts-square} for visualisations of some of the interfaces.

\begin{theorem}[Modulating fronts on square lattice]
    For $\eps > 0$ sufficiently small, the thin-film system \eqref{eq:thin-film-equation} has modulating front solutions on the square lattice, which describe, for $M_0 < 0$,
    \begin{itemize}
        \item an invasion of the roll waves by the pure conduction state if $K_0 > 0$;
        \item an invasion of the square patterns by the pure conduction state if $K_0 + K_1 >0$;
        \item an invasion of the roll waves by the square patterns if $K_0 >0$, $K_0 + K_1 >0$ and $K_1 - K_0 > 0$;
        \item an invasion of the square patterns by the roll waves if $K_0 > 0$, $K_0 + K_1 >0$ and $K_1 - K_0 < 0$;
    \end{itemize}
    and, for $M_0 > 0$,
    \begin{itemize}
        \item an invasion of the pure conduction state by roll waves if $K_0 < 0$;
        \item an invasion of the pure conduction state by square patterns if $K_0 + K_1  < 0$;
        \item an invasion of the roll waves by the square patterns if $K_0 < 0$, $K_0 + K_1 < 0$ and $K_1 - K_0 < 0$;
        \item an invasion of the square patterns by the roll waves if $K_0 < 0$, $K_0 + K_1 < 0$ and $K_1 - K_0 > 0$.
    \end{itemize}
\end{theorem}

As in the existence of stationary, hexagonal patterns, we also restrict to a neighborhood of the parameter curve $(\beta,g)$ such that $N(\beta,g) = 0$. We summarise the result for the hexagonal lattice in the following theorem, see also Theorem \ref{thm:modulating-fronts-hex} and Figures \ref{fig:phase-diags-hex1} and \ref{fig:phase-diags-hex2} for the phase-plane diagrams and Figures \ref{fig:Modfronts-hex} and \ref{fig:Modfronts-hex-two} for visualisations of some of the interfaces.

\begin{theorem}[Modulating fronts on hexagonal lattice]
    For $\eps > 0$ sufficiently small and in a neighborhood of the curve $(\beta(g),g)$, $g\in (10,18)$, the thin-film system \eqref{eq:thin-film-equation} has modulating front solutions on the hexagonal lattice, which describe, for $M_0 < 0$,
    \begin{itemize}
        \item an invasion of up-hexagons by the pure conduction state;
        \item an invasion of up-hexagons by up-hexagons with a larger amplitude;
    \end{itemize}
    and, for $M_0 > 0$,
    \begin{itemize}
        \item an invasion of the pure conduction state by the roll waves;
        \item an invasion of the pure conduction state by up-hexagons;
        \item an invasion of roll waves by up-hexagons, if $M_0\kappa < -\tfrac{K_0 N_0^2}{(K_0 - K_2)^2}$.
    \end{itemize}
\end{theorem}

Note that this is not a comprehensive list of the existing pattern interfaces, see Section \ref{sec:modulating-fronts} for more details. We also point out that since the persistent argument gives the persistence of the whole phase plane, we also find two-stage invasion fronts using these arguments. These take the form of a leading front connecting a primary state to an intermediate state, which is then followed by a second front connecting this intermediate state to a third stationary state, see Remarks \ref{rem:modulating-pattern-selection-square} and \ref{rem:modulating-pattern-selection-hex}. In particular, we point out that both fronts move with the same speed, see also Figures \ref{fig:Modfronts-square}, \ref{fig:Modfronts-hex} and \ref{fig:Modfronts-hex-two}.

\subsection{Physical background and modelling}\label{sec:modelling}

We give a brief overview of the derivation of the thin-film system \eqref{eq:thin-film-equation} from the Boussinesq--Navier--Stokes system without buoyancy. For additional details, we refer the reader to \cite{shklyaev2012}.

The dynamics of a thin liquid film on a heated plane is described by the Boussinesq--Navier--Stokes system on the domain $\{(t,x,y,z) : t>0, 0<z<h(t,x,y)\}$. Buoyancy effects are ignored since they are irrelevant for sufficiently thin films \cite{pearson1958}. The bulk dynamics is given by a coupling of the Navier--Stokes equations for velocity field $\ub = (\ub',u_3)$ and pressure $p$ with a transport-diffusion equation for the temperature $T$ and, in non-dimensionalised form, reads as
\begin{equation}\label{eq:Navier-Stokes}
    \arraycolsep=2pt
    \def\arraystretch{1.2}
    \left\{
    \begin{array}{rcl}
        \tfrac{1}{\mathrm{Pr}} \bigl(\partial_t \ub + (\ub\cdot\nabla)\ub\bigr) & = & - \nabla p + \Delta \ub - \mathrm{Ga} \e_z,  \\
        \div \ub & = & 0,   \\
        \partial_t T + (\ub\cdot \nabla) T & = & \Delta T.
    \end{array}
    \right.
\end{equation}
The non-dimensionalised boundary conditions at the free boundary $z=h(t,x,y)$ are given by a kinematic boundary condition, a stress-balance condition and a heat transfer condition
\begin{equation}\label{eq:z=h}
    \arraycolsep=2pt
    \def\arraystretch{1.2}
    \left\{
    \begin{array}{rcl}
        \partial_t h + \ub'\cdot \nabla h & = & u_3, \\
        \Sigma(p,\ub) \nb & = & \mathrm{Ca}\kappa \nb - \mathrm{Ma} \nabla T, \\
        \nb \cdot \nabla T & = & - \mathrm{Bi} T,
    \end{array}
    \right.
\end{equation}
where $\Sigma(p,\ub) = \mathrm{Pr} \tfrac{\nabla \ub + \nabla \ub^T}{2}-pI$ is the Cauchy stress tensor, $\nb$ is the unit outer normal and $\kappa = -\tfrac{1}{2}\div \nb$ is the mean curvature of the free surface. Finally, the boundary condition at the solid-fluid interface $z=0$ is given as a no-slip condition for the fluid and a condition for the heat flux through the interface
\begin{equation}\label{eq:z=0}
    \arraycolsep=2pt
    \def\arraystretch{1.2}
    \left\{
    \begin{array}{rcl}
        \ub & = & 0, \\
        \partial_z T & = & -1. \\
    \end{array}
    \right.
\end{equation}
The dimensionless parameters are given by
\begin{equation*}
\begin{split} 
    \text{Prandtl number: } & \mathrm{Pr} = \frac{\mu}{\chi}, &
    \text{Capillary number: } & \mathrm{Ca} = \frac{\sigma_0 H}{\mu \chi}, & 
    \text{Marangoni number: } & \mathrm{Ma} = \frac{\alpha T_{\mathrm{flux}} H^2}{\mu\chi},\\
    \text{Biot number: } & \mathrm{Bi} = \frac{KH}{k}, &
    \text{Galileo number: }& \mathrm{Ga} = \frac{G H^3}{\mu \chi},
\end{split}
\end{equation*}
where $\mu$ is the kinematic viscosity, $\chi$ is the thermal diffusivity, $\sigma = \sigma_0 - \alpha T$ is the temperature-dependent surface tension, $H$ is the characteristic height, $-T_{\mathrm{flux}}$ is the heat flux through the solid-fluid interface, $K$ is the heat transfer coefficient at the fluid-gas interface, $k$ is the thermal conductivity of the fluid and $G$ is a gravitational constant. We note that we assume a linear dependence of the surface tension on the temperature. A version of the thin-film system \eqref{eq:thin-film-equation} assuming a nonlinear dependence of the surface tension on the temperature has been derived in \cite{mikishev2021}.

The pure conduction state
\begin{equation*}
    \bar{h} = 1, \quad \bar{T} = -z + \dfrac{1 + \mathrm{Bi}}{\mathrm{Bi}}, \quad \bar{p} = \mathrm{Ga}(1-z), \quad \bar{\ub} = 0
\end{equation*}
is a stationary solution to the system \eqref{eq:Navier-Stokes}-\eqref{eq:z=0}. 

We now make a long-wave ansatz by rescaling the equations with respect to the aspect ratio $\eps$ of the film
\begin{equation*}
    X = \eps x, \quad Y = \eps y, \quad Z = z, \quad \tau = \eps^2 t, \quad \Ubf' = \dfrac{\ub'}{\eps},\quad  U_3 = \dfrac{u_3}{\eps^2},
\end{equation*}
where $0 < \eps \ll 1$ is a small parameter given by the ratio of $H$ to a typical horizontal lengthscale. Then, we expand the unknowns in $\eps^2$ and write
\begin{equation*}
    \Ubf' = \Ubf_0' + \Ocal(\eps^2), \quad T = \bar{T} + \tilde{\theta} + \eps^2 T_1 + \Ocal(\eps^4), \quad U_3 = U_{3,0} + \Ocal(\eps^2).
\end{equation*}
Similarly for $p$. Additionally, we assume that we are in the regime of large Capillary numbers and small Biot numbers. Specifically, we write
\begin{equation*}
    \mathrm{Ca} = \eps^{-2} C, \quad \mathrm{Bi} = \eps^2 \beta,
\end{equation*}
where $C$ and $\beta$ are $\Ocal(1)$.

Inserting the expansions, equating different orders of $\eps$ to zero and following the standard arguments of a lubrication approximation, see \cite{bruell2024}, yields that $\tilde{\theta}$ is independent of $Z$ and an evolution equation for the fluid height $h$
\begin{equation*}
    \partial_\tau h + \div\bigl(\frac{h^3}{3}(C \nabla\Delta h-\mathrm{Ga}\nabla h) + \mathrm{Ma}\frac{h^2}{2} \nabla(h-\tilde{\theta})\bigr) = 0.
\end{equation*}

In order to obtain a closed system, we still need an equation for the unknown $\tilde{\theta} = \tilde{\theta}(\tau,X,Y)$. For this, we need to consider the expansion for $T$ up to order $\eps^2$. Inserting the expansions into the transport-diffusion equation and considering only the terms of order $\eps^2$, we obtain
\begin{equation*}
    \partial_{\tau}\tilde{\theta} + \Ubf'_0\cdot \nabla \tilde{\theta} - U_{3,0} = \partial_Z^2 T_1 + \Delta \tilde{\theta}
\end{equation*}
in the fluid domain. On the fluid-solid interface, we obtain the boundary condition
\begin{equation*}
    \partial_z T_{1} = 0.
\end{equation*}
For the boundary condition on the fluid-gas interface, we recall the rescaling of the Biot number $B = \eps^2 \beta$ and that $\nb = (1+\eps^2|\nabla h|^2)^{-1/2}(-\eps^2\nabla h,1)$ in the rescaled variables and obtain
\begin{equation*}
    -\eps^2\nabla h \cdot \nabla \tilde{\theta} + \eps^2 \partial_Z T_1 - 1 = \bigl(-\eps^2\beta(\tilde{\theta}-h) - 1 - \eps^2 \beta\bigr) \sqrt{1+\eps^2|\nabla h|^2} + \Ocal(\eps^4).
\end{equation*}
Taylor expanding $\sqrt{1+\eps^2|\nabla h|^2}$ and considering only terms of order $\eps^2$, we obtain the boundary condition at $Z=h(\tau,X,Y)$ given by
\begin{equation*}
    \partial_Z T_1 = \nabla h \cdot \nabla \tilde{\theta} - \frac{1}{2}|\nabla h|^2 - \beta(\tilde{\theta}-h) - \beta.
\end{equation*}
To remove the $-\beta$, we introduce $\theta = \tilde{\theta} - 1$. We then arrive at the following system
\begin{equation}\label{eq:Temperature-epseps}
    \arraycolsep=2pt
    \def\arraystretch{1.2}
    \left\{
    \begin{array}{rcll}
        \partial_Z^2 T_1 & = & \partial_{\tau}\theta + \Ubf'_0\cdot \nabla \theta - U_{3,0} - \Delta \theta , & \text{for } 0 < Z < h,   \\
        \partial_Z T_1 & = & 0, & \text{at } Z=0, \\
        \partial_Z T_1 & = & \nabla h\cdot \nabla\theta - \tfrac{1}{2}|\nabla h|^2 - \beta(\theta-h), &\text{at } Z=h. \\
    \end{array}
    \right.
\end{equation}
Integrating the system \eqref{eq:Temperature-epseps} with respect to $Z$ and using the boundary conditions, we obtain the equation
\begin{equation*}
    h\partial_{\tau} \theta - \div (h\nabla\theta) + \frac{1}{2}|\nabla h|^2 + \beta(\theta-h) - j\cdot \nabla(\theta-h) + \div\bigl(\frac{h^4}{8}(C\nabla \Delta h-\mathrm{Ga}\nabla h) + \mathrm{Ma}\frac{h^3}{6} \nabla(h-\theta)\bigr) = 0
\end{equation*}
with nonlinear flux \(j=\frac{h^3}{3}(C\nabla \Delta h-\mathrm{Ga}\nabla h) + \mathrm{Ma}\frac{h^2}{2} \nabla(h-\theta)\). We note that by rescaling, we can without loss of generality set $C = 1$. Thus, using the notation $X = x$, $Y = y$, $\mathrm{Ma} = M$ and $\mathrm{Ga} = g$, we obtain \eqref{eq:thin-film-equation}.

\subsection{Related results}

While there are only a few rigorous mathematical results on the dynamics of thermocapillary thin-film models, see e.g.~\cite{bruell2024} for a recent analysis of stationary periodic solutions and film-rupture in a purely deformational model, there is a vast literature on physical and numerical results for the thin-film model \eqref{eq:thin-film-equation}, as well as rigorous mathematical results for toy models and fluid models without a free surface.

The model \eqref{eq:thin-film-equation} was initially derived in \cite{shklyaev2012}. For a recent overview of other related long-wave models, we refer to the book \cite{shklyaev2017}. The model was then subsequently analysed numerically  \cite{samoilova2021,samoilova2021a}. Additionally, the dynamics in the gas phase was incorporated in \cite{samoilova2015}, which leads to a different critical Marangoni number. Finally, the model was also adapted to model thermocapillary thin films with a nonlinear feedback control \cite{samoilova2019,samoilova2020,samoilova2023}  as well as a nonlinear dependence of the surface tension on the temperature \cite{mikishev2021}. For additional references on the modelling of long-wave phenomena, we refer to \cite{davis1987,oron1997}.

In addition to the experimental results cited above, the dynamics of thermocapillary thin films close to a short-wave instability was also studied numerically. We specifically point out the results in \cite{bestehorn1996} on square patterns in the full Bénard–Marangoni problem, and the fronts between hexagons and squares in a generalised Swift–Hohenberg equation found in \cite{kubstrup1996}.

The mathematical literature for the dynamics of pattern-forming systems close to a short-wave instability, which is related to the present manuscript, can be split into three parts: amplitude equations, planar patterns, and modulating fronts and pattern interfaces. 

There is extensive literature on amplitude equations for pattern-forming systems, although almost exclusively for spatially one-dimensional settings, and we refer to the book \cite{schneider2017} for a recent overview. For the specific case of a (conserved) Turing instability appearing in the thin-film model \eqref{eq:thin-film-equation}, the amplitude equations were formally derived in \cite{matthews2000}. A rigorous justification of these amplitude equations was then carried out for toy models in \cite{häcker2011,schneider2013a}, and in \cite{zimmermann2014} for the full two-dimensional Bénard–Marangoni problem. Related results pertaining to corresponding attractivity and global existence properties were obtained in \cite{dull2016,schneider2017a}.

The analysis of planar patterns arising from a short-wave instability has also received much attention. The bifurcation of rolls and hexagons close to a Turing instability has been treated in \cite{buzano1983} and the pattern selection has been discussed in \cite{golubitsky1984a}. Furthermore, the bifurcations appearing close to a Turing–Hopf instability, where the critical spectral curves have a nontrivial imaginary part at the critical bifurcation parameter and critical Fourier wave numbers, were treated in \cite{silber1991} for the case of a square lattice. Since it is impossible to provide an extensive list of all obtained results, we refer to the overview articles and books \cite{golubitsky1988,cross1993,hoyle2007} and the references therein.

Finally, we discuss results on the existence of modulating fronts, pattern interfaces and related solutions. Slow-moving modulating fronts in problems with one unbounded spatial dimension were first obtained for the cubic Swift–Hohenberg equation in \cite{collet1986,eckmann1991}. This result was then extended to the Taylor–Couette problem in \cite{hărăguş–courcelle1999} and to nonlocal reaction-diffusion equations in \cite{faye2015}. Furthermore, the case of an additional conservation law was studied in \cite{hilder2020b} close to a Turing instability and in \cite{hilder2022} close to a Turing–Hopf instability. We point out that modulating fronts are related to defect solutions, which describe solutions which are time-periodic in a co-moving frame, see e.g.~\cite{sandstede2004,siemer2020}. 

For models with two unbounded spatial directions, which is treated in this manuscript, modulating fronts have only been constructed so far for a toy model, see \cite{doelman2003}. However, there are also results on stationary pattern interfaces close to a Turing instability, which are used to describe domain walls in convective fluids such as the Rayleigh–Bénard problem, see \cite{haragus2007,haragus2012,haragus2021,haragus2022,iooss2023,iooss2024,iooss2024a}.

While existing results focus on planar fronts, experimental results suggest that the formation of polygonal patterns is initiated at point defects, where radial fronts form. Therefore, their study is related to multi-dimensional localised patterns, which have attracted much attention recently, see the overview \cite{bramburger2024}. In particular, we point to \cite{hill2024} for a formal treatment of radial amplitude equations and \cite{groves2024} for a functional analytic setup, which could be used as a starting point to lift techniques such as center manifold reduction to the radial case.

\vspace{-12pt}
\subsection{Outline of the paper}

The structure of the present paper is as follows: in Section \ref{sec:linear-analysis}, we perform a linear stability analysis of the pure conduction state $(\bar{h},\bar{\theta}) = (1,1)$ and establish regimes where the Turing or Turing–Hopf instability occurs.

In Section \ref{sec:amplitude-equation}, we formally derive the amplitude equations for the square and hexagonal lattices using formal multiscale expansions. We also briefly discuss the dynamics of the obtained amplitude equations and formally obtain the reduced equations that will later reappear in the rigorous analysis of the planar patterns and the modulating fronts.

In Section \ref{sec:stationary}, we study the planar patterns that bifurcate from the purely conduction state close to a Turing instability. Therefore, using center manifold reduction, we first derive reduced equations on the square and hexagonal lattice for the central modes. We then solve these equations and prove the persistence to establish the existence of bifurcation curves for square and hexagonal patterns.

Finally, in Section \ref{sec:modulating-fronts}, we construct modulating travelling front solutions to \eqref{eq:thin-film-equation}. For this, we insert the modulating front ansatz into the thin-film system \eqref{eq:thin-film-equation} and reformulate the resulting system as a spatial dynamics problem in the spatial variable $\xi = x-ct$ with $\x = (x,y)$. We then study the spectral properties of the linearisation and perform a center manifold reduction. Next, we derive the reduced equations on the center manifold and analyse them using phase-plane analysis to obtain heteroclinic orbits. Finally, we show their persistence which establishes the existence of modulating travelling planar front solutions in the thin-film system \eqref{eq:thin-film-equation}. 

Appendix \ref{app:heteroclinic} contains a proof of the existence of heteroclinic orbits using Poincaré–Bendixson. Additionally, we provide the explicit expressions for the coefficients in the reduced equations in the Supplementary Material.

\section{Linear stability analysis}\label{sec:linear-analysis}

To understand the dynamics of \eqref{eq:thin-film-equation} close to the pure conduction state, the first step is to study its linear stability. Therefore, we linearise \eqref{eq:thin-film-equation} about \((\bar{h},\bar{\theta}) = (1,1)\) and obtain the linear operator
\begin{equation*}
    \Lcal_M  = \begin{pmatrix} -\frac{1}{3}\Delta^2 + \bigl(\frac{g}{3} - \frac{M}{2}\bigr)\Delta & \frac{M}{2} \Delta \\
    -\frac{1}{8}\Delta^2  + \bigl(\frac{g}{8} - \frac{M}{6}\bigr)\Delta + \beta & \bigl(1+\frac{M}{6}\bigr)\Delta - \beta
    \end{pmatrix}.
\end{equation*}
We call the pure conduction state \emph{spectrally stable} if the spectrum of $\Lcal_M$ lies in the left complex half-plane and \emph{spectrally unstable} otherwise. To identify the $L^2$-spectrum of $\Lcal_M$, we apply the Fourier transform and obtain
\begin{equation*}
    \hat{\Lcal}_M(k) = \begin{pmatrix} -\frac{1}{3}k^4 - \bigl(\frac{g}{3} - \frac{M}{2}\bigr)k^2 & -\frac{M}{2} k^2 \\
    -\frac{1}{8}k^4 - \bigl(\frac{g}{8} - \frac{M}{6}\bigr)k^2 + \beta & - \bigl(1+\frac{M}{6}\bigr)k^2 - \beta
    \end{pmatrix}.
\end{equation*}
Here, we used that due to the rotational symmetry of \eqref{eq:thin-film-equation}, $\Lcal_M$ only depends on $|\k|$, where $\k\in \R^2$ is the wave vector, and hence, we write it as a function of the wave number $k = |\k| \geq 0$.
Then, $\lambda \in \C$ is in the $L^2$-spectrum of $\Lcal_M$ if it satisfies
\begin{equation*}
    \det(\lambda I - \hat{\Lcal}_M(k)) = 0
\end{equation*}
for some $k \geq 0$, or equivalently if $U=\exp(i(\k\cdot\x - \lambda t))\phib(k)$ is a solution to the linear equation $\partial U = \Lcal_M U$, where $\phib(k)$ is an eigenvector of $\hat{\Lcal}_M(k)$. Note that we drop the explicit dependence of $\phib$ on $M$. Therefore, $\lambda$ is a root of the second-order polynomial
\begin{equation*}
    P_{k,M}(\lambda) = \lambda^2 + a_1(k,M) \lambda + a_0(k,M)
\end{equation*}
with coefficients given by
\begin{equation*}
\begin{split}
    a_1(k,M) & = \beta +\frac{1}{3} k^2
   \left(g+k^2-M+3\right),\\
    a_0(k,M) & = \frac{1}{144} \Bigl(48 \beta  k^2 (g+k^2)+k^4 \bigl(48
   (g+k^2)-M (g+k^2+72)\bigr)\Bigr).
\end{split}
\end{equation*}

The polynomial $\lambda \mapsto P_{k,M}(\lambda)$ has two complex roots (counted with multiplicity). We call the root with larger real part $\lambda_+(k;M)$ and the other $\lambda_-(k;M)$. If the real parts are the same, the roots are complex conjugated since $P_{k,M}$ has real coefficients, which is due to the reflection symmetry of \eqref{eq:thin-film-equation}. In this case, we call the root with positive imaginary part $\lambda_+(k;M)$ and the other one $\lambda_-(k;M)$.

Since we are interested in the critical values for the Marangoni number $M$ (depending on $g$ and $\beta$), where the stability of the trivial state changes, we look for roots $\lambda(k;M)$ with vanishing real part. Here, we need to distinguish between $\Im(\lambda(k;M)) = 0$ and $\Im(\lambda(k;M)) \neq 0$. If $\Im(\lambda(k;M)) = 0$, the corresponding solution to the linearised equation $\exp(i(\k\cdot\x - \lambda(k;M) t)\phib(k)$ is stationary, spatially periodic, while for $\Im(\lambda(k;M)) \neq 0$ it is a spatially periodic travelling wave. We start with the first case \(\Im(\lambda(k;M))=0\), which will result in a Turing (or monotonic) instability. We find that $\lambda(k;M) = 0$ is a root of $P_{k,M}$ if and only if $a_0(k,M) = 0$. Solving this for the Marangoni number $M$ gives
\begin{equation*}
    M_m(k) = \frac{48 (\beta+k^2)(g+k^2)}{k^2 \left(g+k^2+72\right)}.
\end{equation*}
In the second case, which corresponds to a Turing--Hopf (or oscillatory, or wave) instability, we find a root $\lambda(k;M)$ of $P_{k,M}$ with \(\Re(\lambda(k;M))=0\) and \(\Im(\lambda(k;M))\neq 0\) by setting $a_1(k,M) = 0$. This gives
\begin{equation*}
    M_o(k) = g+\frac{3\beta}{k^2}+k^2+3.
\end{equation*}
The imaginary part of the root $\lambda(k;M)$ is then given by
\begin{equation*}
    \Im(\lambda(k;M)) = \pm \sqrt{a_0(M_o(k),k)} = \pm\dfrac{k^2}{12}\sqrt{(72+g+k^2)(M_m(k)-M_o(k))}.
\end{equation*}
Note that this requires that $M_m(k) > M_o(k)$.

Next, we want to understand when the instabilities occur. For this, we introduce the notation
\begin{equation*}
    M_m^* := \inf_{k \in \R} M_m(k), \quad M_o^* := \inf_{k \in \R} M_o(k).
\end{equation*}
We now show that if $M \leq \min(M_m^*,M_o^*)$, the pure conduction state is spectrally stable. Since the $L^2$-spectrum of $\Lcal_M$ is the union of all roots $\lambda(k;M)$ of $P_{k,M}$ for all $k \geq 0$, it is sufficient to show that for each $k \geq 0$ the real part of $\lambda_\pm(k;M)$ is non-positive. First, we show that for fixed $k \neq 0$ if $M < \min(M_m(k),M_o(k))$, then both roots $\lambda_\pm(k;M)$ have strictly negative real part. To do this, we use the Routh--Hurwitz criterion, which states that the roots of $P_{k,M}$ have negative real part if and only if $a_0(k,M) > 0$ and $a_1(k,M) > 0$. We note that for fixed $k\neq 0$ both $a_0$ and $a_1$ are strictly decreasing in $M$. Since $a_0(M_m(k),k) = 0$ and $a_1(M_o(k),k) = 0$ by construction of $M_m(k)$ and $M_o(k)$, respectively, this shows that $a_0(k,M) > 0$ and $a_1(k,M) > 0$ as $M < \min(M_m(k),M_o(k))$. Therefore, both roots have negative real part for $k \neq 0$. For $k = 0$ the roots are explicitly given by $\lambda_1(0) = 0$ and $\lambda_2(0) = -\beta < 0$. Second, if $M = \min(M_m^*,M_o^*)$ it follows $\Re \lambda_{\pm}(k)\leq 0$ for all $k\geq 0$ as $\lambda_{\pm}(k)$ continuously depends on $M$.

Next, we will see that for $M>\min(M_m^*,M_o^*)$ the pure conduction state becomes spectrally unstable. This also raises the question of which instability mechanism, i.e. monotonic or oscillatory instability, occurs first and if this instability occurs at a finite wave number. To answer the latter question, note that $k\mapsto M_m(k)$ has a minimiser if and only if $0<\beta < 72$ and $g>0$.  Indeed, it holds
\begin{equation*}
    \partial_k M_m(k) = -\frac{96 \left(\beta  g (g+72)+2 \beta  g k^2+(\beta -72) k^4\right)}{k^3
\left(g+k^2+72\right)^2}
\end{equation*}
and hence, $\partial_k M_m(k) < 0$ for all $k\geq 0$ if $\beta \geq 72$. Therefore, $M_m(k)$ is monotonically decreasing for $\beta \geq 72$ and has no minimiser. In particular, in this case $M_m^* = \lim_{k \rightarrow \infty} M_m(k) = 48$, which is well-known in the physics literature, see e.g.~\cite{pearson1958,shklyaev2012,samoilova2014}. We note, however, that for $\beta > 72$ the long-wave approximation does not reflect the behavior of the full model anymore and is thus not applicable \cite{samoilova2014}.
In contrast, if $0<\beta < 72$, $M_m(k)$ attains its minimal value at a finite wave number $k_m^*>0$ given by
\begin{equation}\label{eq:km}
    (k_m^*)^2 = \frac{\beta  g+6 \sqrt{2} \sqrt{\beta  g (-\beta +g+72)}}{72 - \beta}.
\end{equation} 
Meanwhile, for every \(\beta>0\) and \(g>0\), \(M_o(k)\) has a minimiser \(k_o^*>0\) given by
\begin{equation}\label{eq:ko}
    (k_o^*)^2 = \sqrt{3\beta}.
\end{equation}

We now discuss the spectral situation at the critical Marangoni number $M^* := \min(M_m^*,M_o^*)$, where the pure conduction state destabilises. We distinguish two cases:
\begin{enumerate}[label  = (\arabic*)]
    \item the monotonic case:  $M_m^* < M_o^*$;
    \item the oscillatory case:  $M_o^* < M_m^*$.
\end{enumerate}

\begin{remark}
    The case $M_m^* = M_o^*$, in which both instabilities occur at the same critical Marangoni number, is realised for a one-dimensional sub-manifold of the parameters $\beta$ and $g$. If additionally the instabilities occur at resonant wave numbers $k_m^*$ and $k_o^*$, this scenario can lead to the bifurcation of complex waves. In the case without an additional conservation law, this occurs, for example, in the Taylor-Couette problem, and we refer to \cite{chossat1994} for more details.
\end{remark}

To simplify notation, we then define parameter regimes
\begin{equation*}
    \begin{split}
        \Omega_m &:= \{(\beta,g) \in \R^2_+ \,:\, M_m^* < M_o^*, \ \beta < 72\}, \\
        \Omega_o &:= \{(\beta,g) \in \R^2_+ \,:\, M_o^* < M_m^*\},
    \end{split}
\end{equation*}
see Figure \ref{fig:spectral-regions}.

We now show that if $(\beta,g) \in \Omega_m$, then the pure conduction state destabilises via a monotonic instability, while for $(\beta, g) \in \Omega_o$ an oscillatory instability occurs. We start by discussing the monotonic case.

\begin{proposition}\label{prop:linear-monotonic}
    Let $(\beta,g) \in \Omega_m$. Then, at
    \begin{equation*}
        M = M_m^* = \frac{48 \left(g (72-\beta +g)+12 \left(6 \beta +\sqrt{2} \sqrt{\beta  g (72-\beta+g)}\right)\right)}{(g+72)^2},
    \end{equation*} there exists a wave number $k_m^* > 0$ given by \eqref{eq:km} such that the roots $\lambda_\pm(k;M_m^*)$ of $P_{k,M_m^*}$ satisfy the following conditions, see Figure \ref{fig:dispersion-monotonic}.
    \begin{enumerate}[label=(\arabic*)]
        \item\label{it:linear-monotonic-1} \(\Re(\lambda_+(k;M_m^*)) \leq 0\) and \(\Re(\lambda_-(k;M_m^*)) < 0\) for all \(k \geq 0\);
        \item\label{it:linear-monotonic-2} $\lambda_+(k;M_m^*) = 0$ if and only if $k \in \{0, k_m^*\}$;
        \item\label{it:linear-monotonic-3} $\partial_k \lambda_+(k;M_m^*) = 0$ at $k \in \{0,k_m^*\}$;
        \item\label{it:linear-monotonic-4} $\partial_k^2 \lambda_+(k;M_m^*) < 0$ at $k \in \{0, k_m^*\}$;
        \item\label{it:linear-monotonic-5} $\partial_M\lambda_+(k_m^*;M_m^*) =: \kappa >0$.
    \end{enumerate}
    That is, the pure conduction state destabilises via a \emph{Turing instability} at wave vectors \(\k\) with \(|\k| = k_m^*\). Additionally, the curves $\lambda_\pm(k;M_m^*)$ are eigenvalues of $\hat{\Lcal}_{M_m^*}(k)$ with eigenvectors $\phib_\pm(k)$.
\end{proposition}

\begin{figure}[H]
    \centering
    \includegraphics[width=0.8\textwidth]{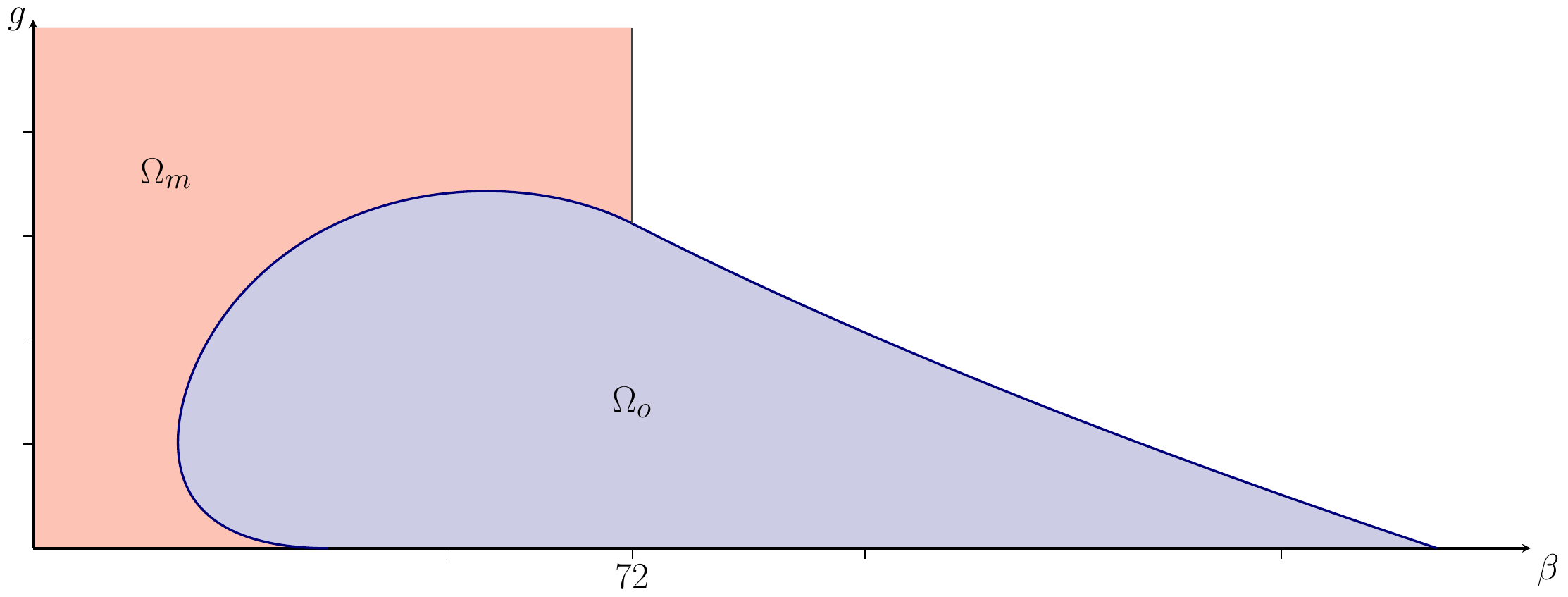}
    \caption{Depiction of the two regions $\Omega_m$ and $\Omega_o$ in parameter space.}
    \label{fig:spectral-regions}
\end{figure}

\begin{remark}\label{rem:Msmaller48}
    It holds $M_m^* < 48$ for all $(\beta,g) \in \Omega_m$, which can be obtained by explicitly checking that $\partial_{\beta}M_m^*(g,\beta)>0$ for $0<\beta<72$ and $g\geq 0$ and recalling that $M_m^*(g,72) = 48$.
\end{remark}

\begin{remark}\label{rem:imaginary-roots}
    Observe that there is \(C>0\) such that $|\Im(\lambda_{\pm}(k;M))| \leq C$ for all $k>0$. Indeed, this follows from $\lambda_\pm(k;M) = \frac{-a_1(k,M) \pm \sqrt{a_1(k,M)^2 - 4a_0(k,M)}}{2}$ and the asymptotics $a_1(k,M)\sim k^4$ and $a_0(k,M)\sim k^6$ as $k\to \infty$. We will use this observation later in the construction of a center manifold for stationary planar patterns.
\end{remark}

\begin{figure}[H]
    \centering
    \includegraphics[width=0.8\linewidth]{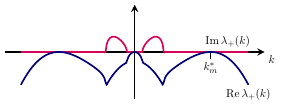}
    \caption{Plot of $\lambda_+(k;M^*_m)$ for $g=12$, $\beta= 0.1865184573$ and $M = M^*_m = 8.5144749311$. The pure conduction state destabilises via a Turing instability at the critical wave number $k_m^* = 1.2843299054$.}
    \label{fig:dispersion-monotonic}
\end{figure}

\begin{proof}
    Let $k_m^*$ be the unique, positive minimiser of $k \mapsto M_m(k)$. Then, by the choice of $M = M_m^* \leq M_m(k)$ and due to the strict monotonicity of $a_0(k,M)$ in $M$, we know that $a_0(M_m^*,k) > a_0(M_m(k),k) = 0$ for all $k \notin \{0, k_m^*\}$. Since $M < M_o^* \leq M_o(k)$ for all $k > 0$ by assumption, it holds that $a_1(M_m^*,k) > a_1(M_o(k),k) =  0$ for all $k > 0$ due to the strict monotonicity of $a_1(k,M)$ in $M$. Therefore, by the Routh--Hurwitz criterion, both roots \(\lambda_\pm(k;M_m^*)\) of $P_{k,M_m^*}$ have strictly negative real part for all $k \notin \{0, k_m^*\}$.
    
    For $k = 0$, we have that $a_0(M_m^*,0) = 0$ and $a_1(M_m^*,0) = \beta > 0$ and hence $\lambda_+(0;M_m^*) = 0$ is a root of $P_{k,M_m^*}$, while the other root \(\lambda_-(0;M_m^*) = -\beta\) has negative real part.

    For \(k= k_m^*>0\), we have that $a_0(M_m^*,k_m^*) = 0$ and thus, $\lambda_+(k_m^*;M_m^*) = 0$ is a root of $P_{M_m^*,k_m^*}$. Since $a_1(M_m^*, k_m^*) > a_1(M_o(k_m^*),k_m^*) = 0$, the second root $\lambda_-(k_m^*;M_m^*)$ is negative. This proves \ref{it:linear-monotonic-1} and \ref{it:linear-monotonic-2}.

    Since $a_0(M_m^*, k) = 0$ and $a_1(M_m^*,k)>0$ for \(k\in \{0, k_m^*\}\) and by the continuous dependence of the roots of \(P_{k,M_m^*}\) on its coefficients, we know that there are neighborhoods of \(0\) and \(k_m^*\) such that $\Re\lambda_-(k;M_m^*) < \Re\lambda_+(k;M_m^*)$ and thus \(\Im\lambda_+(k;M_m^*) = 0\). By the implicit function theorem, we hence obtain that \(k\mapsto \lambda_+(k;M_m^*)\) is a smooth function in the respective neighborhoods of \(0\) and \(k_m^*\). Further, since \(\lambda_+(k;M_m^*) \leq 0\), \(0\) and \(k_m^*\) are maximum points of \(\lambda_+(k;M_m^*)\) proving \ref{it:linear-monotonic-3}.

    To prove \ref{it:linear-monotonic-4}, we notice that by the implicit function theorem, it holds
    \begin{equation*}
        \partial^2_k \lambda_+(k;M_m^*) = - \frac{\partial_k^2 a_0(M_m^*,k)}{a_1(M_m^*,k)} < 0
    \end{equation*}
    for $k \in \{0,k_m^*\}$. Indeed, it holds $\partial_k^2 a_0(M_m^*,k) > 0$ for $k \in \{0,k_m^*\}$ since $k\mapsto a_0(M_m^*,k) \geq 0$ is a polynomial of degree 6 and has three minima in $\R$ at $k = 0, \pm k_m^*$.

    Finally, to prove \ref{it:linear-monotonic-5}, we may use the implicit function theorem again to find that
    \begin{equation}\label{eq:partial-M-implicit-function-thm}
        \partial_M \lambda_+(k_m^*;M_m^*) = - \frac{\partial_M a_1(M_m^*, k_m^*) \lambda(k_m^*) + \partial_M a_0(M^*_m,k_m^*)}{2\lambda(k_m^*) + a_1(M_m^*,k_m^*)} = - \frac{\partial_M a_0(M^*_m,k_m^*)}{a_1(M_m^*,k_m^*)} > 0,
    \end{equation}
    since $\partial_M a_0(M_m^*,k_m^*) = -\tfrac{1}{144}(k_m^*)^4\bigl(g+(k_m^*)^2 + 72\bigr) < 0$ and $a_1(M_m^*,k_m^*) > 0$ as shown above.
\end{proof}

Next, we discuss the spectral picture in the case of an oscillatory instability.

\begin{proposition}\label{prop:linear-oscillatory}
    Let $(\beta, g) \in \Omega_o$. Then, at
    \begin{equation*}
        M = M_o^* = 3 + g + \sqrt{3\beta},
    \end{equation*} 
    there exists a wave number $k_o^* > 0$ given by \eqref{eq:ko} such that the roots $\lambda_\pm(k;M_o^*)$ of $P_{k,M_o^*}$ satisfy the following conditions, see Figure \ref{fig:dispersion-oscillatory}.
    \begin{enumerate}[label=(\arabic*)]
        \item\label{it:linear-oscillatory-1} \(\Re(\lambda_\pm(k;M_o^*)) \leq 0\) for all \(k \geq 0\);
        \item\label{it:linear-oscillatory-2} $\Re(\lambda_+(k;M_o^*)) = 0$ if and only if $k \in \{0, k_o^*\}$ and $\Re(\lambda_-(k;M_o^*)) = 0$ if and only if $k = k_o^*$;
        \item\label{it:linear-oscillatory-2-Im} $\Im \lambda_{\pm}(k_o^*;M_o^*) \neq 0$;
        \item\label{it:linear-oscillatory-3} $\partial_k \Re(\lambda_\pm(k_o^*;M_o^*)) = 0$ and $\partial_k \lambda_+(0;M_o^*) = 0$;
        \item\label{it:linear-oscillatory-4} $\partial_k^2 \Re(\lambda_\pm(k_o^*;M_o^*)) < 0$ and $\partial_k^2 \lambda_+(0;M_o^*) < 0$;
        \item\label{it:linear-oscillatory-5} $\partial_M \Re(\lambda_{\pm}(k_o^*;M_o^*)) >0$.
    \end{enumerate}
    That is, the pure conduction state destabilises via an oscillatory instability at wave vectors \(\k\) with \(|\k| = k_o^*\).
\end{proposition}

\begin{figure}[H]
    \centering
    \includegraphics[width=0.48\linewidth]{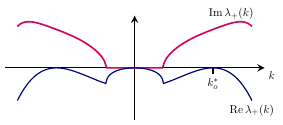}\hfill
    \includegraphics[width=0.48\linewidth]{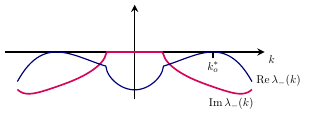}
    \caption{Plot of $\lambda_+(k;M_o^*)$ (left) and $\lambda_-(k;M_o^*)$ (right) for $g=12$, $\beta= 40$ and $M = M^*_o = 36.9089023002$. The pure conduction state destabilises via an oscillatory instability at the critical wave number $k_o^* = 3.3097509196$.}
    \label{fig:dispersion-oscillatory}
\end{figure}

The proof of \ref{it:linear-oscillatory-1}--\ref{it:linear-oscillatory-3} in Proposition \ref{prop:linear-oscillatory} works similar to the proof of Proposition \ref{prop:linear-monotonic}.
However, since \(a_1(M_o^*,k_o^*) = 0\), we cannot rely on the implicit function theorem to obtain the quadratic behaviour of the spectral curves. Hence, the proof takes a slightly different path.

\begin{proof}
    Let $k_o^*$ be the unique, positive minimiser of $k \mapsto M_o(k)$. Since $M_o^* < \min(M_m(k),M_o(k))$ for $k \notin \{0, k_o^*\}$, we have that $a_0(M_o^*,k), a_1(M_o^*,k) > 0$ for $k \notin \{0, k_o^*\}$. Therefore, the real part of $\lambda_\pm(k;M_o^*)$ is strictly negative for $k \notin \{0, k_o^*\}$.

    For $k = 0$, the roots are again given by $\lambda_+(0;M_o^*) = 0$ and $\lambda_-(0;M_o^*) = -\beta < 0$. For $k = k_o^*$, we have that $a_1(M_o^*, k_o^*) = 0$ and $a_0(M_o^*, k_o^*) > a_0(M_m( k_o^*), k_o^*) = 0$ since $M_o^* < M_m(k_o^*)$. Therefore, $\lambda_\pm(k_o^*;M_o^*) = \pm i\sqrt{a_0(M_o^*,k_o^*)}$ and the real parts vanish. This proves \ref{it:linear-oscillatory-1} and \ref{it:linear-oscillatory-2}.

    Since, in a neighborhood of \(0\) and \(k_o^*\), the roots depend smoothly on the coefficients of the polynomial \(P_{k,M_o^*}\) and the real parts of \(\lambda_\pm\) have maxima in \(k_o^*\) and \(\lambda_+\) has another maximum in \(0\), we find that \(\partial_k \Re(\lambda_\pm(k_o^*;M_o^*)) = 0\) and \(\partial_k\lambda_+(0;M_o^*) = 0\). Hence, we obtain \ref{it:linear-oscillatory-3}.

    To prove \ref{it:linear-oscillatory-4}, at \(k=0\) we may argue as in the proof of Proposition \ref{prop:linear-monotonic} via the implicit function theorem. For \(k=k_o^*\), we cannot apply the implicit function theorem since $a_1(M_o^*,k_o^*) = 0$. Instead, we notice that in a neighborhood of \(k_o^*\) the imaginary parts of \(\lambda_\pm\) are strictly bounded away from zero. Hence, we find that in a neighborhood of \(k_o^*\), it holds \(\Re(\lambda_\pm(k;M_o^*)) = - \frac{a_1(k,M_o^*)}{2}\) since \(P_{k,M_o^*}\) has real coefficients. Hence, we may conclude
    \begin{equation*}
        \partial_k^2 \Re(\lambda_\pm(k_o^*;M_o^*)) = - \frac{1}{2} \partial_k^2 a_1(M_o^*,k_o^*) =  - \frac{4}{3} \sqrt{3\beta} < 0.
    \end{equation*}

    Finally, to prove \ref{it:linear-oscillatory-5}, we may use \eqref{eq:partial-M-implicit-function-thm} again to obtain
    \begin{equation*}
        \partial_M \Re \lambda_{\pm} (k_o^*;M_o^*) = - \frac{\partial_M a_1(M_o^*,k_o^*)}{2} > 0,
    \end{equation*}
    where we use that \(\partial_M a_1(M_o^*,k_o^*) = - \frac{1}{3}(k_o^*)^2\). This concludes the proof.
\end{proof}

\section{Derivation of the amplitude equations}\label{sec:amplitude-equation}

After completing the linear stability analysis in Section \ref{sec:linear-analysis}, we now analyse the dynamics of solutions to \eqref{eq:thin-film-equation}, which are close to the pure conduction state $(\bar{h},\bar{\theta}) = (1,1)$. From now on, we also restrict to the case $(\beta, g) \in \Omega_m$, i.e.~that the system is close to a Turing instability. Therefore, we also write $M^* := M_m^*$. In this section, we give an intuition for the dynamics using formal multiscale expansions leading to the formal amplitude equations. Since the calculations only rely on some general properties of the system, we rewrite \eqref{eq:thin-film-equation} as
\begin{equation}\label{eq:evolution-equation-U}
    \begin{pmatrix}
        1 & 0 \\ 0 & h
    \end{pmatrix}\partial_t \Ucal = \Fcal_M(\Ucal) = \Lcal_M \Ucal + \Ncal(\Ucal;M),  \quad \Ucal = \begin{pmatrix}
        h - 1 \\ \theta - 1
    \end{pmatrix},
\end{equation}
where $\Lcal_M$ denotes the linear terms and $\Ncal$ denotes the nonlinearities. Additionally, we write $\Ncal$ as
\begin{equation*}
    \Ncal(\Ucal;M) = \Ncal_2(\Ucal;M) + \Ncal_3(\Ucal;M) + R(\Ucal;M),
\end{equation*}
where $\Ncal_2$ and $\Ncal_3$ are the quadratic and cubic nonlinearities, which are given by
\begin{equation}\label{eq:def-nonlinearities}
\begin{split}
    \Ncal_2(\Ucal,\Vcal;M) & := \dfrac{1}{2} D^2\Fcal_M(0,0)[\Ucal,\Vcal], &&\quad \Ncal_2(\Ucal;M) := \Ncal_2(\Ucal,\Ucal;M), \\
    \Ncal_3(\Ucal,\Vcal,\Wcal;M) & := \dfrac{1}{6} D^3\Fcal_M(0,0)[\Ucal,\Vcal,\Wcal], && \quad \Ncal_3(\Ucal;M) := \Ncal_2(\Ucal,\Ucal,\Ucal;M).
\end{split}
\end{equation}
We note that $\Ncal_2(\cdot,\cdot;M)$ is a symmetric bilinear form and $\Ncal_3(\cdot,\cdot,\cdot;M)$ is a symmetric trilinear form. Additionally, $R = \Ocal(\|\Ucal\|^4)$ denotes the higher-order nonlinearities.

For the purpose of this paper, we restrict to solutions to \eqref{eq:thin-film-equation} with certain spatial periodicities. In particular, we are interested in roll waves, square patterns and hexagonal patterns. To represent these periodic patterns, we introduce the following notation: for a fixed \(N\in \N\), let
\begin{equation*}
    \k_1,\ldots,\k_N \in S_{k_m^*} := \{\k \in \R^2 : |\k| = k_m^*\}
\end{equation*}
be a set of wave vectors and define the corresponding lattice in Fourier space
\begin{equation}\label{eq:fourier-lattice}
    \Gamma := \Bigl\{\sum_{j=1}^{N} n_j \k_j : n_j\in \Z\Bigr\}.
\end{equation}
For future use, we also define
\begin{equation*}
    \Gamma_0 := \Bigl\{0, \pm\k_j, \,:\, j = 1,\dots,N\Bigl\}, \text{ and }  \Gamma_h := \Gamma \setminus \Gamma_0.
\end{equation*}
We endow \(\Gamma\) with the natural distance function \(d(\,\cdot\,\,\,\cdot\,)\) given by the \(\ell^1\)-norm and we normalise the distance such that $d(0,\k_j) = 1$. For simplicity of notation, we also define \(\k_{-j} = - \k_j\) and $\k_0=0$. Square patterns are given by the choice
\begin{equation*}
    \k_1 = k_m^*\begin{pmatrix}
        1 \\ 0
    \end{pmatrix}, \quad \k_2 = k_m^*\begin{pmatrix}
        0 \\ 1
    \end{pmatrix}
\end{equation*}
and hexagonal patterns by 
\begin{equation*}
    \k_1 = k_m^*\begin{pmatrix}
        1 \\ 0
    \end{pmatrix}, \quad \k_2 = \frac{k_m^*}{2}\begin{pmatrix}
        -1 \\ \sqrt{3}
    \end{pmatrix}, \quad \k_3 = -\frac{k_m^*}{2} \begin{pmatrix}
        1 \\ \sqrt{3}
    \end{pmatrix}
\end{equation*}
and roll waves are subsumed under both of these cases. We point out that it is not possible to handle the case of squares and hexagons simultaneously with the methods presented here since a lattice containing both patterns is not discrete and has non-isolated points on the circle $S_{k_m^*}$. Typically, these lattices are referred to as quasilattices. These support quasipatterns, which are not periodic but have a long-range order, see \cite{braaksma2017,iooss2019,iooss2022} for details.

\begin{figure}[H]
    \centering
    \includegraphics[width=0.45\linewidth]{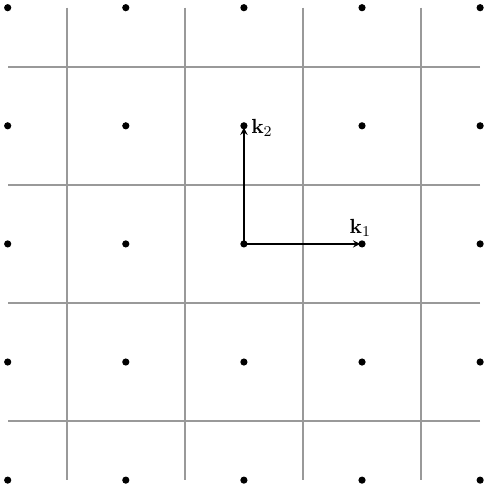}\hfill
    \includegraphics[width=0.45\linewidth]{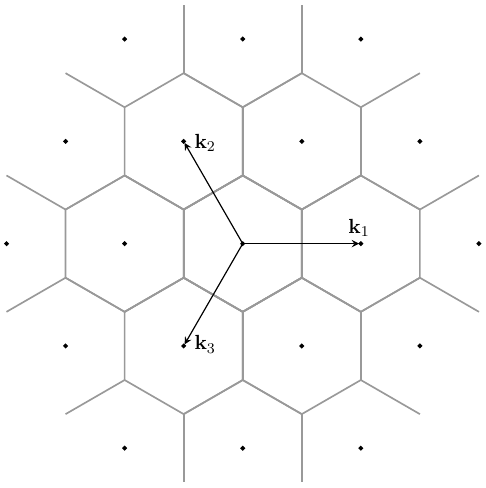}
    \caption{Fourier lattices for square patterns (left) and hexagonal patterns (right). The grey lines form the dual lattice of $\Gamma$, which consists of the fundamental domains of periodicity in physical space.}
    \label{fig:Fourier-lattice}
\end{figure}
Functions that are periodic with respect to the dual lattice of \(\Gamma\) can then be written as
\begin{equation*}
    \Ucal(\x) = \sum_{\gammab \in \Gamma} \a_{\gammab} e^{i\gammab\cdot \x}, 
\end{equation*}
where $\x\in \R^2$, $\a_0 \in \R^2$ and $\a_{\gammab}\in \C^2$ for $\gammab \neq 0$.

The goal of this section now is to formally derive the amplitude equations, which govern the dynamics of \eqref{eq:thin-film-equation} on the periodic lattice $\Gamma$ close to the onset of the monotonic instability. We follow the standard approach of deriving these equations via a formal multiscale ansatz, see e.g.\ \cite{schneider2017}. The idea is that the solutions close to the onset of instability are given to leading order by an amplitude modulation of the unstable Fourier modes since the other modes are exponentially damped. To extract these modes, we define the projections onto the eigenspaces corresponding to $\lambda_\pm(k;M^*)$ by
\begin{equation*}
\begin{split}
    P_\pm(k) & := \frac{\lrangle{\phib_\pm^*(k),\, \cdot\,}}{\lrangle{\phib_\pm^*(k),\phib_\pm(k)}}, \quad k = 0,k_m^*,  
\end{split}
\end{equation*}
where $\phib_\pm^*(\k)$ is the eigenvector of the adjoint eigenvalue problem $(\bar{\lambda}_\pm(\k;M^*) I - \hat{\Lcal}_{M^*}(\k)^*)\phib_\pm^*(\k) = 0$. These projections are well-defined since $\lambda_+(k;M^*) \neq \lambda_-(k;M^*)$ for $k = 0,k_m^*$. Using this projection, we additionally define
\begin{equation*}
        P_0 := P_+(0),
\end{equation*}
as well as, on the function spaces, the projections
\begin{equation*}
\begin{split}
    \Pcal(\gammab)& :  \Ucal = \sum_{\gammab \in \Gamma} \a_{\gammab} e^{i\gammab\cdot \x} \longmapsto \a_{\gammab}, \\
    \Pcal_{\pm}(\k)& : \Ucal = \sum_{\gammab \in \Gamma} \a_{\gammab} e^{i\gammab\cdot \x} \longmapsto P_{\pm}(|\k|)\a_{\k}, \quad \k \in \{0,\pm \k_j : j = 1,\ldots,N\}.
\end{split}
\end{equation*}
Note that the projection $P_\pm(k)$ is only defined on a single coefficient of the Fourier series, whereas $\Pcal_\pm(\k)$ and $\Pcal(\gamma)$ apply to the full Fourier series.

To derive the system of amplitude equations close to the onset of monotonic instability $M^*:=M_m^*$, we write $M = M^* + \eps^2 M_0$ and make the following ansatz
\begin{equation}\label{eq:ansatz-physical}
    \begin{split}
        \Ucal(t,\x) = &\varepsilon\left(\sum_{j=1}^{N} A_j(T,\X) \phib_+(\k_j) E_j + c.c.\right) + \varepsilon^2 A_0(T,\X) \phib_+(0) E_0 + \varepsilon^2 \Psi_0(T,\X) \phib_-(0) E_0 \\
        &+ \varepsilon^2\left(\sum_{j=1}^{N} \Psi_j(T,\X) \phib_-(\k_j) E_j + c.c.\right) + \varepsilon^2 \sum_{\gammab\in \Gamma : d(\gammab,0)=2} \Psib_{\gammab}(T,\X) E_{\gammab},
    \end{split}
\end{equation}
where $A_j(T,\X), \Psi_j(T,\X) \in \C$ with $A_{-j}(T,\X) = \bar{A}_j(T,\X)$ and $\Psi_{-j}(T,\X) = \bar{\Psi}_j(T,\X)$, $j=1,\dots,N$. Additionally, $A_0(T,\X), \Psi_0(T,\X) \in \R$,  $\Psib_{\gammab}(T,\X) \in \C^2$, $T = \varepsilon^2 t$, and $\X = \varepsilon \x$. Recall that $\phib_{\pm}(k_m^*)$ are the eigenvectors of $\Lcal_{M^*}(k_m^*)$, c.f.~Proposition \ref{prop:linear-monotonic}, and we denote by $\phib_{\pm}(\k_j) = \phib_{\pm}(|\k_j|)$ in order to clarify which mode is affected. Furthermore, we use the notation $E_j = e^{i\k_j\cdot\mathbf{x}}$ and $E_{\gammab} = e^{i\gammab\cdot\mathbf{x}}$. By transforming $\Ucal$ into Fourier space, we obtain the following ansatz
\begin{equation}\label{eq:ansatz-fourier}
    \begin{split}
        \hat{\Ucal}(t,\k) = &\left(\sum_{j=1}^{N} \hat{A}_j\Big(T,\frac{\k-\k_j}{\varepsilon}\Big) \phib_+(\k) + c.c.\right) + \varepsilon \hat{A}_0\Big(T,\frac{\k}{\varepsilon}\Big) \phib_+(\k) + \varepsilon \hat{\Psi}_0\Big(T,\frac{\k}{\varepsilon}\Big) \phib_-(\k) \\
        &+ \varepsilon^2\left(\sum_{j=1}^{N} \hat{\Psi}_j\Big(T,\frac{\k-\k_j}{\varepsilon}\Big) \phib_-(\k) + c.c.\right) + \varepsilon^2 \sum_{\gammab\in \Gamma : d(\gammab,0)=2} \hat{\Psib}_{\gammab}\Big(T,\frac{\k-\gammab}{\varepsilon}\Big).
    \end{split}
\end{equation}

\begin{remark}
    We point out that the ansatz in Fourier space \eqref{eq:ansatz-fourier} is not the Fourier transform of the ansatz in physical space \eqref{eq:ansatz-physical}, since the eigenvectors are not fixed at the particular wave number. This is necessary to obtain the correct expansion on the linear level. However, since it makes the nonlinear analysis much more involved, we use the ansatz in physical space for the nonlinear expansions. In fact, the formal derivation can be carried out fully in Fourier space, and we refer to e.g.~\cite{schneider2017} for the derivation in case of a reaction-diffusion system.
\end{remark}

\begin{remark}
    The $\varepsilon^2$-terms at $\phib_-(\k_j) E_j$, $j = 0,\dots,N$ and $E_{\gammab}$, where $\gammab\in \Gamma$ satisfies $d(\gammab,0)=2$, are correction terms, which are necessary to obtain the correct coefficients in the formal amplitude equation since they are used to remove lower-order nonlinear terms appearing in the expansion. Since we need to balance terms on $\varepsilon^4$, it would also be necessary to introduce terms of order $\varepsilon^3$ in the ansatz. However, it turns out that they do not play a role in the amplitude equation and therefore, we neglect them.
\end{remark}

\subsection{The linear expansion}\label{sec:amplitude-equation-lin}
On the linear level, we make the calculations in Fourier space to obtain the necessary expansion. Since these are similar for the different linear terms after inserting the ansatz \eqref{eq:ansatz-fourier}, we perform the expansion in detail for one term,
\begin{equation*}
    \begin{split}
        \hat{\Lcal}_M(k) \hat{A}_j\Big(T,\frac{\k-\k_j}{\varepsilon}\Big) \phib_+(\k;M) &= \lambda_+(\k;M) \hat{A}_j\Big(T,\frac{\k-\k_j}{\varepsilon}\Big) \phib_+(\k;M) \\
        &= \lambda_+(\k_j + \varepsilon \Kbf;M^*+\varepsilon^2 M_0) \hat{A}_j(T,\Kbf) \phib_+(\k_j + \varepsilon \Kbf;M^*+\varepsilon^2 M_0) \\
        &= \dfrac{1}{2} \varepsilon^2 \lrangle{\Kbf, D^2_\k\lambda_+(\k_j;M^*)\Kbf} \hat{A}(T,\Kbf) \phib_+(\k_j;M^\ast) \\
        &+ \varepsilon^2 M_0 \partial_M\lambda_+(\k_j;M^*) \hat{A}(T,\Kbf) \phib_+(\k_j;M^\ast) + \Ocal(\varepsilon^3),
    \end{split}
\end{equation*}
where $D^2_\k\lambda_+$ denotes the Hessian of $\lambda_+$ with respect to the wave vector $\k$. In the first equality, we use that $\phib_+(\k;M)$ is an eigenvector with corresponding eigenvalue $\lambda_+(\k;M)$. For the second equality, we set $\k = \k_j + \eps \Kbf$; that is, we localise in Fourier space around $\k_j$ and then perform a Taylor expansion. The lower-order terms in $\eps$ all vanish using that $\lambda_+(\k_j;M^*) = 0$ and $\nabla \lambda_+(\k_j;M^*) = 0$ since $\lambda_+$ has a maximum at $\k = \k_j$, see Proposition \ref{prop:linear-monotonic}.

Performing a similar expansion for the other terms in the ansatz and then Fourier transforming back to physical space, we obtain that
\begin{equation*}
    \begin{split}
        \Lcal_M \Ucal &= \varepsilon^3 \sum_{j = 1}^N \Big(-\dfrac{1}{2}\dfrac{\lambda_+''(k_m^*;M^*)}{(k_m^*)^2} \lrangle{\k_j,\nabla}^2 A_j(T,\X) + \partial_M \lambda_+(\k_j;M^*)A_j(T,\X)\Big) E_j \phib_+(\k_j;M^*)  + c.c. \\
        &\quad - \dfrac{1}{2}\varepsilon^4 \lambda_+^{\prime\prime}(0) \Delta A_0(T,\X) \phib_+(0) + \varepsilon^2 \lambda_-(0;M^*) \Psi_0(T,\X) \phib_-(0) \\
        &\quad + \varepsilon^2\left(\sum_{j=1}^{N} \lambda_-(\k_j; M^*) \Psi_j(T,\X) \phib_-(\k_j) E_j + c.c.\right) + \varepsilon^2 \sum_{\gammab \in \Gamma : d(\gammab,0) = 2} \hat{\Lcal}_{M^*}(\gammab) \Psib_{\gammab}(T,\X) E_{\gammab}+ \Rcal,
    \end{split}
\end{equation*}
where $\Rcal$ denotes higher-order or irrelevant terms. Examples of irrelevant terms are $\varepsilon^3$-contributions arising from the $\Psib_{\gammab}$, which do not influence the coefficients of the amplitude equations. Here, we also used that $\lrangle{\nabla, D_\k^2\lambda_+(\k_j ;M^*)\nabla} = \tfrac{\lambda_+''(k_m^*;M^*)}{(k_m^*)^2} \lrangle{\k_j,\nabla}^2$ and $\lrangle{\nabla, D^2_\k\lambda_+(0;M^*)\nabla} = \lambda_+''(0;M^*)\Delta$ due to rotation symmetry of $\lambda_+(\k;M^*)$ and the fact that $\lambda_+(\k;M^*)$ has a quadratic root at $\k = 0,\k_j$.

\subsection{The nonlinear expansion}\label{sec:amplitude-equation-nonlin}
For simplicity, we refrain from calculating the full nonlinear expansion and focus on the relevant terms instead. Note that we also drop the dependence of all objects on $M$ since all relevant terms are evaluated at $M^*$. These are the following.
\begin{enumerate}
    \item For $\varepsilon^2$, we are interested in quadratic combinations of the $A_j$, which determine the $\varepsilon^2$-balance terms in the ansatz. Additionally, on the hexagonal lattice, this also generates a quadratic term in the amplitude system due to the resonance $\k_1 + \k_2 + \k_3 = 0$.
    \item For $\varepsilon^3$, we are only interested in the balance at $E_j \phib_+(\k_j)$ for $j = 1,\dots,N$. These terms are created by quadratic interactions of $E_0$ and $E_j$ terms, as well as $E_{\k_j + \k_\ell}$ and $E_{-\ell}$ terms. Additionally, we have to consider cubic interactions of $E_j$, $E_{\ell}$ and $E_{-\ell}$. Through this balance, we obtain the amplitude equation for the $A_j$.
    \item For $\varepsilon^4$, we calculate the balance at $E_0\phib_+(0)$. As we will see, the relevant nonlinear terms can only be created by a quadratic interaction of $E_j$ and $E_{-j}$. This is due to the fact that $\phib_+(0)$ corresponds to the conservation law in the system and the system is reflection symmetric. Therefore, we show that two derivatives must fall on the $A_j$, which adds two powers of $\varepsilon$, see Lemma \ref{lem:conservation-law}. All other nonlinearities at $\k = 0$ are at least of order $\eps^5$ due to the divergence structure.
\end{enumerate}

We split the discussion into three parts. First, we discuss (1) and (2) in the case of a square lattice. Second, we outline the changes to (1) and (2) for the hexagonal lattice. Third, we discuss (3), which is the same for both lattices.

\subsection{The nonlinear expansion: square lattice}\label{sec:amplitude-equation-nonlin-square}

To obtain the nonlinear expansion, we first calculate the correction terms. Using the linear expansion and balancing terms on $\varepsilon^2 E_0 \phib_-(0)$ we find that
\begin{equation}\label{eq:amplitude-Psi0}
    \begin{split}
        &0 = \lambda_-(0) \Psi_0 + 2\sum_{j = 1}^N \Pcal_-(0) \Ncal_2(E_j \phib_+(\k_j), E_{-j} \phib_+(\k_{-j})) \abs{A_j}^2 \\
        \Longleftrightarrow &\Psi_0 = - 2\lambda_-(0)^{-1} \sum_{j = 1}^N \Pcal_-(0) \Ncal_2(E_j \phib_+(\k_j), E_{-j} \phib_+(\k_{-j})) \abs{A_j}^2 =: \nu_0 \sum_{j = 1}^N \abs{A_j}^2.
    \end{split}
\end{equation}
Here, we use that $\lambda_-(0) < 0$ and we note that the coefficients $\nu_0 \in \R$ are independent of $j$ due to rotation symmetry.
Next, we consider the terms at $\varepsilon^2 E_{\gammab}$ for $\gammab=\k_j+\k_{\ell} \in \Gamma$ with $d(\gammab,0) = 2$. Note that, for fixed $j=1,\ldots,N$, the decomposition \(\gammab=\k_j+\k_{\ell}\) is unique. With this, we find that
\begin{equation}\label{eq:amplitude-Psi2}
    \begin{split}
        &0 = \hat{\Lcal}_{M^*}(\k_j + \k_\ell)\Psib_{\k_j + \k_\ell}+ 2 \Ncal_2(E_j\phib_+(\k_j), E_\ell \phib_+(\k_\ell)) A_j A_\ell E_{\k_j + \k_\ell} \\
        \Longleftrightarrow &\Psib_{\k_j + \k_\ell} = -2 \hat{\Lcal}_{M^*}(\k_j + \k_\ell)^{-1}\Ncal_2(E_j\phib_+(\k_j), E_\ell \phib_+(\k_\ell)) A_j A_\ell E_{\k_j + \k_\ell} =: \nub_{\k_j + \k_\ell} A_j A_\ell
    \end{split}
\end{equation}
with $\nub_{\k_j + \k_\ell}\in \C^2$. Here, we use that $\Re(\lambda_\pm(\k_j+\k_\ell)) < 0$ by assumption and thus the matrix $\hat{\Lcal}_M(\k_j + \k_\ell)$ is indeed invertible.

Since the $\varepsilon^2$-balance determines the compensation terms $\Psi_0$ and $\Psib_{\gammab}$, we can now consider the $\varepsilon^3$-balance at $E_j \phib_+(\k_j)$ for $j = 1,\dots, N$. Here, on the nonlinear level we find the terms
\begin{equation*}
    \begin{split}
        &\eps^3\Pcal_+(\k_j)\Bigg[2 \Ncal_2(E_0\phib_+(0),E_j \phib_+(\k_j)) A_0 A_j + \nu_0 \sum_{\ell = 1}^N \Ncal_2(E_0\phib_-(0), E_j \phib_+(\k_j)) A_j \abs{A_\ell}^2 \\
        &\qquad\qquad+\sum_{\substack{\k_j + \k_\ell = \gammab\\ d(\gammab,0)=2}} 2 \Ncal_2(E_{\k_j+\k_\ell} \nub_{\k_j+\k_\ell}, E_{-\ell}\phib_+(\k_{-\ell})) A_j \abs{A_\ell}^2 \\
        &\qquad\qquad+ 3 \Ncal_3(E_j \phib_+(\k_j), E_j \phib_+(\k_j), E_{-j} \phib_+(\k_{-j})) A_j \abs{A_j}^2 \\
        &\qquad\qquad+ \sum_{\ell = 1,\ell\neq j}^N 6 \Ncal_3(E_j \phib_+(\k_j), E_\ell \phib_+(\k_\ell), E_{-\ell} \phib_+(\k_{-\ell})) A_j \abs{A_\ell}^2\Bigg].
    \end{split}
\end{equation*}
Hence, the nonlinear expansion on the square lattice for $j=1$ is given by
\begin{equation}\label{eq:nonlin-expansion-square}
    \eps^3\Big[K_c A_0 A_1 + K_0 A_1 \abs{A_1}^2 + K_1 A_1 \abs{A_2}^2\Big].
\end{equation}

\subsection{The nonlinear expansion: hexagonal lattice}\label{sec:amplitude-equation-nonlin-hex}

We obtain $\Psi_0$ and $\Psib_{\gammab}$ with $\gammab \in \Gamma$ such that $d(\gammab,0) = 2$ as for the square lattice, see \eqref{eq:amplitude-Psi0} and \eqref{eq:amplitude-Psi2}, respectively. Additionally, due to the resonance $\k_1 + \k_2 + \k_3 = 0$ in the hexagonal lattice, the $\varepsilon^2$-balance at $E_j \phib_+(\k_j)$ now has a nontrivial contribution, which for $j = 1$ reads as
\begin{equation*}
    \eps^2\Pcal_+(\k_1) \left[2 \Ncal_2(E_{-2} \phib_+(\k_{-2}), E_{-3} \phib_+(\k_{-3})) \bar{A}_2 \bar{A}_3\right] = \eps^2 N\bar{A}_2\bar{A}_3.
\end{equation*}
The corresponding terms for $j = 2,3$ can then be obtained by cyclic permutation. Hence, the nonlinear expansion on the hexagonal lattice is of the form
\begin{equation}\label{eq:nonlin-expansion-hex}
    \eps^2 N \bar{A}_2 \bar{A}_3 + \eps^3\Big[K_c A_0 A_1 + K_0 A_1 \abs{A_1}^2 + K_2  A_1 \bigl(\abs{A_2}^2 + \abs{A_3}^2\bigr)\Big].
\end{equation}
Note that the coefficients $K_c$ and $K_0$ in \eqref{eq:nonlin-expansion-square} and \eqref{eq:nonlin-expansion-hex} are indeed the same.

\subsection{The nonlinear expansion: conservation law}\label{sec:amplitude-equation-con}

Finally, we consider the nonlinear terms, which emerge at $\varepsilon^4 E_0 \phib_+(0)$. Since the calculations are the same for the square and hexagonal lattice, we do not distinguish the different settings here. It turns out that the lowest order nonlinearity is a potentially non-elliptic, second-order operator in $|A_j|^2$. This holds under fairly generic assumptions, as the following result shows.

\begin{lemma}\label{lem:conservation-law}
    Assume that 
    \begin{equation*}
        \Pcal_+(0) \Fcal_M(\Ucal) = \nabla \cdot (f(\Ucal) \nabla g(\Ucal, \Delta \Ucal))
    \end{equation*}
    with smooth functions $f$ and $g$. Then, it holds that
    \begin{equation*}
        \Pcal_+(0) \Ncal_2(A_j \phib_+(\k_j) E_j, A_{-j}\phib_+(\k_{-j}) E_{-j}) = \eps^2 \nabla \cdot (p_c(\k_j \k_j^T) \nabla(|A_j|^2)) + \Ocal(\eps^3),
    \end{equation*}
    where $p_c$ is a polynomial of degree one.
\end{lemma}
\begin{proof}
    Let $\bar{\Ucal} = (1,1)^T$, $\Vcal = A_j \phib_+(\k_j) E_j$ and $\Wcal = A_{-j}\phib_+(\k_{-j}) E_{-j}$. Then, using the form of $\Fcal_M$, we calculate that
    \begin{equation*}
        \begin{split}
            \Pcal_+(0) \Ncal_2(\Vcal,\Wcal) &= \dfrac{1}{2} \Pcal_+(0) D^2 \Fcal_M(\bar{\Ucal})[\Vcal,\Wcal] \\
            &= \dfrac{1}{2}\Bigl(\nabla \cdot (\lrangle{\Vcal, D^2f(\bar{\Ucal}) \Wcal}_\mathrm{sym} \nabla g(\bar{\Ucal})) + \nabla \cdot (f(\bar{\Ucal}) \nabla \lrangle{\Vcal, D_1^2 g(\bar{\Ucal}) \Wcal}_\mathrm{sym}) \\
            &\qquad + \nabla \lrangle{\Vcal, D_1D_2 g(\bar{\Ucal}) \Delta \Wcal}_\mathrm{sym}) + \nabla \lrangle{\Delta\Vcal, D_2^2 g(\bar{\Ucal}) \Delta\Wcal}_\mathrm{sym}) \\
            &\qquad + \nabla \cdot (Df(\bar{U}) \Vcal (\nabla (D_1g(\bar{\Ucal}) \Wcal + D_2g(\bar{\Ucal}) \Delta \Wcal)) \\
            &\qquad + \nabla \cdot (Df(\bar{U}) \Wcal (\nabla (D_1g(\bar{\Ucal}) \Vcal + D_2g(\bar{\Ucal}) \Delta \Vcal))\Bigr),
        \end{split}
    \end{equation*}
    where $\lrangle{u,\Gcal v}_\mathrm{sym} := \lrangle{v,\Gcal u} + \lrangle{u,\Gcal v}$ with $\Gcal = D_1g(\bar{\Ucal})$ or $G = D_2 g(\bar{\Ucal})\Delta$. Here, we also use the notation $g(\bar{\Ucal}) = g(\bar{\Ucal},\Delta \bar{\Ucal})$ for convenience.

    We now discuss these terms separately. First, we note that 
    \begin{equation}\label{eq:conservation-law-1}
        \nabla \cdot (\lrangle{\Vcal, D^2 f(\bar{\Ucal}) \Wcal}_\mathrm{sym} \nabla g(\bar{\Ucal})) = 0
    \end{equation}
    since $\bar{U}$ is constant.
    Second, since $A_j = A_j(\eps \x)$ and $E_j E_{-j} = 1$, every quadratic interaction of $\Vcal$ and $\Wcal$, or any of their derivatives, only depends on $\eps \x$. Therefore, we find that
    \begin{equation}\label{eq:conservation-law-2}
        \begin{split}
            &\nabla \cdot (f(\bar{\Ucal}) \nabla \lrangle{\Vcal, D_1^2 g(\bar{\Ucal}) \Wcal}_\mathrm{sym}) + \nabla \lrangle{\Vcal, D_1D_2 g(\bar{\Ucal}) \Delta \Wcal}_\mathrm{sym}) + \nabla \lrangle{\Delta\Vcal, D_2^2 g(\bar{\Ucal}) \Delta\Wcal}_\mathrm{sym}) \\
            &= \eps^2 \tilde{\kappa}_0 \Delta (|A_j|^2) + \Ocal(\eps^3)
        \end{split}
    \end{equation}
    with $\tilde{\kappa}_0 \in \R$. Hence, after rearranging, the two remaining terms are
    \begin{equation*}
        \begin{split}
            \nabla \cdot I_1 &:= \nabla \cdot ( Df(\bar{\Ucal}) \Vcal \nabla(D_1g(\bar{\Ucal}) \Wcal) + Df(\bar{\Ucal}) \Wcal \nabla (D_1 g(\bar{\Ucal}) \Vcal)), \\
            \nabla \cdot I_2 &:= \nabla \cdot ( Df(\bar{\Ucal}) \Vcal \nabla(D_2 g(\bar{\Ucal}) \Delta \Wcal) + Df(\bar{\Ucal}) \Wcal \nabla (D_2 g(\bar{\Ucal}) \Delta\Vcal)).
        \end{split}
    \end{equation*}
    To treat $I_1$, we recall that $\Wcal = \phib_+(\k_j) \bar{A}_j \bar{E}_j$. Using this, we obtain
    \begin{equation*}
        \begin{split}
            I_1 &= (Df(\bar{\Ucal})\phib_+(\k_j))(D_1g(\bar{\Ucal})\phib_+(\k_j)) \left[(i\k_j - i\k_j) |A_j|^2 + \eps (A_j \nabla \bar{A}_j + \bar{A}_j \nabla A_j)\right] + \Ocal(\eps^2) \\ &= \eps (Df(\bar{\Ucal})\phib_+(\k_j))(D_1g(\bar{\Ucal})\phib_+(\k_j)) \nabla(|A_j|^2) + \Ocal(\eps^2).
        \end{split}
    \end{equation*}
    This yields 
    \begin{equation}\label{eq:conservation-law-3}
        \nabla \cdot I_1 = \eps^2 (Df(\bar{\Ucal})\phib_+(\k_j))(D_1g(\bar{\Ucal})\phib_+(\k_j)) \Delta (|A_j|^2) + \Ocal(\eps^3).
    \end{equation}

    For $I_2$, we note that
    \begin{equation*}
        \begin{split}
            \nabla \Delta (E_j A_j) &= (\nabla \Delta E_j) A_j + \eps (2 D^2 E_j + \Delta E_j) \nabla A_j + \Ocal(\eps^2) \\
            &= -i\k_j |\k_j|^2 E_j A_j + \eps (-2 \k_j \k_j^T - |\k_j|^2) E_j \nabla A_j + \Ocal(\eps^2).
        \end{split}
    \end{equation*}
    Then, proceeding similarly to the treatment of $I_1$, we obtain that
    \begin{equation}\label{eq:conservation-law-4}
        \nabla \cdot I_2 = - \eps^2 (Df(\bar{\Ucal})\phib_+(\k_j))(D_2g(\bar{\Ucal})\phib_+(\k_j)) \nabla \cdot ((|\k_j|^2 + 2 \k_j\k_j^T)\nabla (|A_j|)^2) + \Ocal(\eps^3).
    \end{equation}
    Therefore, combining \eqref{eq:conservation-law-1}--\eqref{eq:conservation-law-4} shows that the quadratic nonlinearities in the conservation law are indeed of the claimed form and that the polynomial $p_c$ is given by
    \begin{equation}\label{eq:conservation-law-polynom}
        \begin{split}
            p_c(\k_j\k_j^T) &= \Bigl(\tilde{\kappa}_0 + (Df(\bar{\Ucal})\phib_+(\k_j))\bigl(D_1g(\bar{\Ucal})\phib_+(\k_j) - D_2g(\bar{\Ucal})\phib_+(\k_j)|k_j|^2\bigr)\Bigr) I \\
            &\qquad - 2 (Df(\bar{\Ucal})\phib_+(\k_j))(D_2g(\bar{\Ucal})\phib_+(\k_j)) \k_j\k_j^T.
        \end{split}
    \end{equation}
    This completes the proof.
\end{proof}

In the case of the thin-film system \eqref{eq:thin-film-equation}, we can obtain the polynomial $p_c$ explicitly. First, we note that since $\phib_+^*(0) = (1,0)^T$, the projection onto the conserved mode is a linear combination of terms, which are of the form $\nabla \cdot (f(\Ucal)\nabla g(\Ucal, \Delta \Ucal))$. In fact, we find that $g(\Ucal, \Delta \Ucal)$ is always linear and $f(\Ucal)$ only depends on $h$. The former yields that $\tilde{\kappa}_0 = 0$ in \eqref{eq:conservation-law-polynom}. Recalling the nonlinear flux $\frac{h^3}{3}(\nabla \Delta h-g\nabla h) + M\frac{h^2}{2} \nabla(h-\theta)$, we thus obtain that the leading order nonlinearity of the conservation law is given by
\begin{equation*}
    \sum_{j = 1}^N \nabla \cdot \bigl(p_c(\k_j\k_j^T) \nabla (|A_j|^2)\bigr), \quad p_c(\k_j\k_j^T) = \kappa_0 + \kappa_1 \k_j\k_j^T,
\end{equation*}
where 
\begin{equation*}
    \kappa_0 = (M-g)(\phib_+(\k_j))_1^2 + g (\phib_+(\k_j))_1(\phib_+(\k_j))_2 - |\k_j|^2 (\phib_+(\k_j))_1^2, \quad \kappa_1 = - (\phib_+(\k_j))_1^2.
\end{equation*}

\begin{remark}
    We point out that other nonlinear terms do not contribute to leading-order terms. Since there is always one leading derivative due to the divergence structure of the film-height equation, quadratic combinations of $\Psib_{\gammab}$-terms are at least of order $\eps^5$. Additionally, cubic combinations of the $A_j$ cannot generate a term at Fourier mode $E_0$ and all other terms are of higher order due to the divergence structure.
\end{remark}

\subsection{The amplitude equations}
We now combine the results from previous Sections \ref{sec:amplitude-equation-lin}--\ref{sec:amplitude-equation-con} to calculate the full amplitude equations to leading order. For the square lattice, we find that the formal amplitude equations, up to higher-order terms, are given by
\begin{equation}\label{eq:amplitude-equations-square}
    \begin{split}
        \partial_T A_1 &= -\dfrac{\lambda''_+(k_m^*)}{2 (k_m^*)^2} \partial_X^2 A_1 + M_0 \kappa A_1 + K_c A_0 A_1 + K_0 A_1 \abs{A_1}^2 + K_1 A_1 \abs{A_2}^2, \\
        \partial_T A_2 &= -\dfrac{\lambda''_+(k_m^*)}{2 (k_m^*)^2} \partial_Y^2 A_2 + M_0 \kappa A_2 + K_c A_0 A_2 + K_0 A_2 \abs{A_2}^2 +  K_1 A_2 \abs{A_1}^2, \\
        \partial_T A_0 &= -\dfrac{1}{2} \lambda_+^{\prime\prime}(0) \Delta A_0 + \div \bigl(p_c(\k_1\k_1^T) \nabla \abs{A_1}^2 + p_c(\k_2 \k_2^T) \nabla\abs{A_2}^2\bigr).
    \end{split}
\end{equation}
Here, the coefficients are given in Section \ref{sec:amplitude-equation-nonlin-square} and are explicitly computed in the Supplementary Material; see also Figure \ref{fig:density-plot-square-coeffs}.

For the hexagonal lattice, the amplitude equations are to leading order given by
\begin{equation}\label{eq:amplitude-equations-hex}
    \begin{split}
        \partial_T A_1 &= -\dfrac{\lambda''_+(k_m^*)}{2 (k_m^*)^2} \partial_X^2 A_1 + M_0 \kappa A_1 + K_c A_0 A_1 + \dfrac{N}{\eps}\bar{A}_2\bar{A}_3 + K_0 A_1 \abs{A_1}^2 + K_2 A_1 (\abs{A_2}^2 + \abs{A_3}^2), \\
        \partial_T A_2 &= -\dfrac{\lambda''_+(k_m^*)}{2 (k_m^*)^2} \left(-\dfrac{1}{2} \partial_X + \dfrac{\sqrt{3}}{2}\partial_Y\right)^2 A_2 + M_0 \kappa A_2 + K_c A_0 A_2 + \dfrac{N}{\eps}\bar{A}_1\bar{A}_3  + K_0 A_2 \abs{A_2}^2 \\
        & \quad +  K_2 A_2 (\abs{A_1}^2 + \abs{A_3}^2), \\
        \partial_T A_3 &= \dfrac{\lambda''_+(k_m^*)}{2 (k_m^*)^2} \left(\dfrac{1}{2} \partial_X + \dfrac{\sqrt{3}}{2}\partial_Y\right)^2 A_3 + M_0 \kappa A_3 + K_c A_0 A_3 + \dfrac{N}{\eps}\bar{A}_1\bar{A}_2  + K_0 A_3 \abs{A_3}^2 \\
        & \quad +  K_2 A_3 (\abs{A_1}^2 + \abs{A_2}^2), \\
        \partial_T A_0 &= -\dfrac{1}{2} \lambda_+^{\prime\prime}(0) \Delta A_0 + \div \bigl(p_c(\k_1\k_1^T) \nabla \abs{A_1}^2 + p_c(\k_2 \k_2^T) \nabla\abs{A_2}^2 + p_c(\k_3 \k_3^T) \nabla \abs{A_3}^2\bigr).
    \end{split}
\end{equation}
Again, the coefficients are given in Sections \ref{sec:amplitude-equation-nonlin-square} and \ref{sec:amplitude-equation-nonlin-hex} and are explicitly computed in the Supplementary Material. We point out that the resonance $\k_1 + \k_2 + \k_3=0$ of the hexagonal lattice leads to a quadratic term, which is singular in $\eps$ in the scaling used here. If $N \neq 0$ it turns out that hexagons arise through a transcritical bifurcation, but all small amplitude hexagons are unstable, see \cite{hoyle2007}. Hence, one typically restricts to parameter regimes where $N = \Ocal(\varepsilon)$ in the analysis of the amplitude equations, see e.g.~\cite{doelman2003}. It turns out that such a parameter regime also exists for the thin-film system \eqref{eq:thin-film-equation}, see Section \ref{sec:stationary-hex} and \cite{shklyaev2012}.

\begin{remark}
    Note that the degeneracy of the principal symbol in \eqref{eq:amplitude-equations-square} and \eqref{eq:amplitude-equations-hex} reflects that only a discrete set of the unstable wave numbers is covered by the ansatz.
\end{remark}

\subsection{Dynamics of amplitude equations: steady states and fronts}\label{sec:dynamics-amplitude-equations}

To conclude this section, we briefly discuss the dynamics, which can be observed in the leading-order amplitude equations \eqref{eq:amplitude-equations-square} and \eqref{eq:amplitude-equations-hex} and how it relates to solutions to the full thin-film system \eqref{eq:thin-film-equation}. Our main focus lies on spatially homogeneous steady states, which correspond to planar patterns in the original system, and on fast travelling fronts, whose profile is homogeneous in one spatial direction. These correspond to moving pattern interfaces in the original equation. In the remainder of the paper, we then establish the existence of these structures rigorously, see Section \ref{sec:stationary} for planar patterns and Section \ref{sec:modulating-fronts} for pattern interfaces.

The starting point of our analysis is the observation that the amplitude equations have certain invariant subspaces. For the square lattice, where the amplitude equation is given by \eqref{eq:amplitude-equations-square}, these are $\{A_1 = 0, A_2\in \R\}$ (or equivalently $\{A_2=0, A_1\in \R\}$) and $\{A_1 = A_2 \in \R\}$, which in light of
\begin{equation*}
    \Ucal(t,\x) = 2\varepsilon \bigl[A_1(T,\X) \cos(\k_1 \cdot \x) + A_2(T,\X) \cos(\k_2 \cdot \x)\bigr] \phib_+(k_m^*) + \Ocal(\varepsilon^2)
\end{equation*}
correspond to modulated roll waves and modulated square patterns, respectively. Similarly, for the hexagonal lattice, where the amplitude equations are given by \eqref{eq:amplitude-equations-hex}, the subspaces $\{A_1 \in \R, A_2=A_3=0\}$ (and all cyclic permutations) and $\{A_1 = A_2 = A_3 \in \R\}$ are invariant. Again, in light of
\begin{equation*}
    \Ucal(t,\x) = 2\varepsilon \bigl[A_1(T,\X) \cos(\k_1 \cdot \x) + A_2(T,\X) \cos(\k_2 \cdot \x) + A_3(T,\X) \cos(\k_3 \cdot \x)\bigr] \phib_+(k_m^*) + \Ocal(\varepsilon^2),
\end{equation*}
these correspond to modulated roll waves and modulated hexagonal patterns, respectively. Observe that on the invariant subspace of modulated roll waves, the amplitude equations on the square and hexagonal lattice are identical. We note that on the square lattice, we can always assume that $A_1, A_2 \in \R$ by passing to polar coordinates and using that the equations \eqref{eq:amplitude-equations-square} are invariant under $A_1 \mapsto A_1 e^{i\alpha}$ and $A_2 \mapsto A_2 e^{i\alpha}$, respectively. On the hexagonal lattice, this invariance is broken by the quadratic term, and thus, we cannot assume that the solutions are real-valued outside of the invariant subspaces mentioned above.

Now, steady planar patterns are obtained as spatially homogeneous solutions in these invariant subspaces. In particular, the conservation law is then solved by any constant $A_0 \in \R$, and we always set $A_0 = 0$. Roll waves $A_1 = A \in \R$ and $A_2 = 0$ (and $A_3=0$ on the hexagonal lattice) are then solutions to
\begin{equation*}
    0 = M_0 \kappa + K_0 A^2.
\end{equation*}
Square patterns $A_1 = A_2 = A \in \R$ are given as solutions to
\begin{equation*}
    0 = M_0 \kappa + (K_0 + K_1) A^2.
\end{equation*}
Finally, hexagonal patterns $A_1 = A_2 = A_3 = A \in \R$ are obtained as solutions to
\begin{equation*}
    0 = M_0\kappa + \frac{N}{\eps} A + (K_0 + 2 K_2) A^2.
\end{equation*}
We discuss their existence in detail in Section \ref{sec:stationary}, see in particular Theorems \ref{thm:square-patterns} and \ref{thm:hex-patterns}. 

Finally, we consider travelling wave solutions, specifically solutions of the form
\begin{equation*}
    A_j(T,\X) = A_j(X-\eps^{-1}cT).
\end{equation*}
Note that $X-\eps^{-1} cT = \eps(x-ct)$ and thus this corresponds to modulated patterns, where the amplitude modulation is homogeneous in $y$-direction and moves in $x$-direction with speed $c = \Ocal(1)$. It is well-known that fast-moving front solutions $v(T,X) = V(X-\eps^{-1} c T)$ for the reaction-diffusion system
\begin{equation*}
    \partial_T v = \partial_X^2 v + f(v)
\end{equation*}
satisfy the singularly perturbed profile ODE
\begin{equation*}
    0 = \eps^2 \partial_{\Xi}^2 V + c \partial_{\Xi} V + f(V),
\end{equation*}
where $\Xi = \eps(X-\eps^{-1} c T) = \eps^2(x-ct)$ using $X = \eps x$ and $T = \eps^2 t$. Formally, a profile can then be obtained by setting $\eps = 0$ and solving the first-order ODE. This idea can indeed be made rigorous using center manifold theory, see e.g.~\cite{scheel1996}, or using geometric singular perturbation theory, see e.g.~\cite{kuehn2015}. Applying these ideas to our systems, we obtain the profile ODE system
\begin{equation*}
    \begin{split}
        0 &= c\partial_\Xi A_1 + M_0 \kappa A_1 + K_c A_0 A_1 + K_0 A_1 \abs{A_1}^2 + K_1 A_1 \abs{A_2}^2, \\
        0 & = c \partial_\Xi A_2 + M_0 \kappa A_2 + K_c A_0 A_2 + K_0 A_2 \abs{A_2}^2 +  K_1 A_2 \abs{A_1}^2, \\
        0 & = c \partial_{\Xi} A_0
    \end{split}
\end{equation*}
for the square lattice and
\begin{equation*}
    \begin{split}
        0 &= c\partial_\Xi A_1 + M_0 \kappa A_1 + K_c A_0 A_1 + \dfrac{N}{\eps}\bar{A}_2\bar{A}_3 + K_0 A_1 \abs{A_1}^2 + K_2 A_1 (\abs{A_2}^2 + \abs{A_3}^2), \\
        0 &= c\partial_\Xi A_2 + M_0 \kappa A_2 + K_c A_0 A_2 + \dfrac{N}{\eps}\bar{A}_1\bar{A}_3  + K_0 A_2 \abs{A_2}^2 +  K_2 A_2 (\abs{A_1}^2 + \abs{A_3}^2), \\
        0 &= c\partial_\Xi A_3 + M_0 \kappa A_3 + K_c A_0 A_3 + \dfrac{N}{\eps}\bar{A}_1\bar{A}_2  + K_0 A_3 \abs{A_3}^2 + K_2 A_3 (\abs{A_1}^2 + \abs{A_2}^2), \\
        0 &= c\partial_{\Xi} A_0
    \end{split}
\end{equation*}
for the hexagonal lattice. In particular, we formally obtain that $A_0$ is spatially constant, and thus we again assume that $A_0 = 0$. Recalling the form of the multiscale ansatz \eqref{eq:ansatz-physical} leading to the amplitude equations, we point out that heteroclinic connections between the fixed points in these first-order systems thus correspond to fast-moving pattern interfaces in the full system \ref{eq:thin-film-equation}. We make the formal ideas presented here rigorous in Section \ref{sec:modulating-fronts} and analyse the dynamics of the first-order systems in detail in Sections \ref{sec:modulation-center-square} and \ref{sec:modulation-center-hex}. The corresponding phase planes can be found in Figures \ref{fig:phase-diags-square1} and \ref{fig:phase-diags-square2} for square patterns and Figures \ref{fig:phase-diags-hex1} and \ref{fig:phase-diags-hex2} for hexagonal patterns. Images of the resulting pattern interfaces can be found in Figures \ref{fig:Modfronts-square}, \ref{fig:Modfronts-hex} and \ref{fig:Modfronts-hex-two}. We also refer to \cite{Hilder_Data_availability_for_2024}, where videos of modulating fronts can be found.

\section{Bifurcation of periodic planar stationary solutions}\label{sec:stationary}

In Section \ref{sec:dynamics-amplitude-equations} we predicted the existence of stationary patterns using formal multiscale analysis. The goal here is to construct these solutions rigorously using center manifold theory. We begin by defining appropriate function spaces. For this, recall the Fourier lattice $\Gamma$, see \eqref{eq:fourier-lattice}, and let
\begin{equation*}
    \Ucal(\x) = \sum_{\k \in \Gamma} \a_{\k} e^{i\k\cdot \x},\quad \x\in \R^2.
\end{equation*}
be a periodic function with respect to the dual lattice of $\Gamma$. For \(\ell_1,\ell_2 \in [0,\infty)\), we then introduce the real function spaces
\begin{equation*}
\begin{split}
    L^2_\Gamma & = \Bigl\{ \Ucal = \sum_{\k \in \Gamma} \a_{\k} e^{i\k\cdot \x} : \a_{\k} = \overline{\a_{-\k}}\in \C^2,\ \|\Ucal\|_{L^2_{\Gamma}} < +\infty\Bigr\},\\
    H^{\ell_1,\ell_2}_\Gamma & = \Bigl\{ \Ucal \in L^2_{\Gamma} : \|\Ucal\|_{H^{\ell_1,\ell_2}_{\Gamma}} < \infty \Bigr\}, \\
    H^{\ell}_{\Gamma} & := H^{\ell,\ell}_{\Gamma}, \quad \ell>0,
\end{split}
\end{equation*}
with corresponding norms defined by
\begin{equation*}
\begin{split}
    \|\Ucal\|_{L^2_{\Gamma}}^2 & :=  \sum_{\k\in \Gamma} |\a_{\k}|^2, \\
    \|\Ucal\|_{H^{\ell_1,\ell_2}_{\Gamma}}^2 & := \sum_{\k\in \Gamma} (1+|\gammab|^2)^{\ell_1}|\a_{\k,1}|^2 + (1+|\k|^2)^{\ell_2}|\a_{\k,2}|^2.
\end{split}
\end{equation*}
Additionally, we define the subspace
\begin{equation*}
    L_{\Gamma,0}^2 := \{\Ucal \in L^2_\Gamma : P_0 \a_0 = 0\}
\end{equation*}
and similarly $H^{\ell_1,\ell_2}_{\Gamma,0} = H^{\ell_1,\ell_2}_\Gamma\cap L^2_{\Gamma,0}$. This removes the zero eigenvalue at $\k = 0$ corresponding to the conservation law.

In order to apply center manifold reduction, we rewrite equation \eqref{eq:thin-film-equation} as
\begin{equation}\label{eq:cm-system}
\left\{
    \begin{array}{rcl}
    \partial_t \Ucal & = & \Lcal_{M^*}\Ucal + \Rcal(\Ucal;M), \\
    \partial_t M & = & 0,
    \end{array}
\right. \quad \text{where } \Ucal = \begin{pmatrix}
    h-1 \\ \theta -1
\end{pmatrix}
\end{equation}
and $\Rcal(\Ucal;M) := (\Lcal_M-\Lcal_{M^*}) \Ucal + \Ncal(\Ucal;M)$, which is nonlinear in $(\Ucal,M)$.

To construct bifurcation branches consisting of stationary periodic solutions with a roll wave, square or hexagonal symmetry, respectively, we use center manifold reduction and prove the hypotheses of \cite[Thm. 2.20, Hypotheses 2.1, 2.4 and 2.15]{haragus2011a} in the following three corollaries.

The first corollary establishes Hypothesis 2.1, by establishing the mapping properties of $\Lcal_{M^*}$ and $\Rcal$ in appropriate function spaces.

\begin{corollary}\label{cor:mapping-properties}
    Let \(\ell > 0\). Then
    \begin{equation*}
        \Lcal_{M^*} \in L(\Zcal,\Xcal)\quad \text{and}\quad \Rcal \in C^1(\Zcal\times \R,\Ycal)
    \end{equation*}
    with $\Zcal = H^{\ell+4,\ell+2}_{\Gamma,0}$ and $\Ycal = \Xcal = H^{\ell}_{\Gamma,0}$. Additionally, it holds
    \begin{equation*}
        \Rcal(0;M^*) = D_1\Rcal(0;M^*) = 0 \quad \text{for all } M \geq 0.
    \end{equation*}
\end{corollary}

\begin{proof}
    Since $\eqref{eq:thin-film-equation}$ contains at most fourth-order derivatives of $h$ and second-order derivatives of $\theta$, the linear operator $\Lcal$ maps $\Zcal$ into $\Xcal$. Additionally, since $\ell > 0$, the space $\Zcal$ is a Banach algebra, which, together with the fact that $\Rcal$ is a polynomial nonlinearity in $\Ucal$ and $M$, establishes the mapping properties of $\Rcal$.
\end{proof}

The second corollary establishes Hypothesis 2.4 which shows that the spectrum of $\Lcal_{M^*}$ decomposes into a central and a hyperbolic part, which are separated by a spectral gap. This follows directly from the linear analysis in Section \ref{sec:linear-analysis} and the fact that the Fourier lattice $\Gamma$ is discrete.

\begin{corollary}\label{cor:spectral-decomposition}
    It holds
    \begin{equation*}
        \sigma(\Lcal_{M^*}) = \sigma_0 \cup \sigma_-,
    \end{equation*}
    where \(\sigma_0\) is the central part of the spectrum, given by
    \begin{equation*}
        \sigma_0 = \{\lambda_+(|\pm \k_j|;M^*) : j=1,\ldots, N \} = \{0\}\subset i\R,
    \end{equation*}
    and \(\sigma_- = \sigma(\Lcal_{M^*})\setminus\sigma_0\) is the stable part of the spectrum, and there exists a \(\delta > 0\) such that
    \begin{equation*}
        \Re(\lambda) < -\delta \quad \text{for all } \lambda \in \sigma_-.
    \end{equation*}
\end{corollary}

The spectral decomposition gives rise to decompositions of \(\Xcal = \Xcal_0 \oplus \Xcal_h\) and \(\Zcal = \Zcal_0 \oplus \Zcal_h\), where
\begin{equation*}
\begin{split}
    \Xcal_0 & = \bigl\{\Ucal \in \Xcal : \Pcal(\gammab) \Ucal = 0,\ \Pcal_-(\pm \k_j) \Ucal = 0,\ \Pcal_-(0)\Ucal = 0,\ j=1,\ldots,N,\ \gammab\in \Gamma_h\bigr\}, \\
    \Zcal_0 & = \Zcal\cap \Xcal_0,
\end{split}
\end{equation*}
i.e. \(\Xcal_0\) is the finite-dimensional central eigenspace of \(\Lcal_{M^*}\).

Finally, the third corollary establishes the resolvent estimate of Hypothesis 2.15.

\begin{corollary}\label{cor:resolvent-estimate}
    There exist $\omega_0 >0$ and a constant \(C>0\) such that for all $\omega > \omega_0$ it holds \(i\omega\) is in the resolvent set of \(\Lcal_{M^*}\) and
    \begin{equation*}
        \norm{(i\omega - \Lcal_{M^*})^{-1}}_{L(\Xcal)} \leq \frac{C}{|\omega|}.
    \end{equation*}
\end{corollary}

\begin{proof}
    Observe that
    \begin{equation*}
        \norm{(i\omega - \Lcal_{M^*})^{-1}}_{L(\Xcal)} \leq \sup_{k\geq 0} |i\omega -\lambda_{\pm}(k;M^*)|^{-1} \sim \frac{1}{|\omega|} \quad \text{as } |\omega| \to \infty,
    \end{equation*}
    since \(\Im(\lambda_{\pm}(k;M^*)) \leq C\) for \(k\geq 0\), see Remark \ref{rem:imaginary-roots}.
\end{proof}

Combining Corollaries \ref{cor:mapping-properties}--\ref{cor:resolvent-estimate}, we can now apply \cite[Thm. 2.20]{haragus2011a} and obtain the following theorem.

\begin{theorem}\label{thm:center-manifold}
    Let $\ell > 0$. Then there exists a neighborhood $\Ofrak$ of $(0,M^*)$ in $\Zcal_0 \times \R$ and a smooth map $\Psib \colon \Ofrak \to \Zcal_h$ such that
    \begin{equation*}
        \Psib(0;M^*) = D\Psib(0;M^*) = 0
    \end{equation*}
    and the center manifold given by
    \begin{equation*}
        \Mcal_0 = \{\Ucal_0 + \Psib(\Ucal_0;M) : (\Ucal_0,M) \in \Ofrak\}
    \end{equation*}
    is invariant and contains all small bounded solutions of \eqref{eq:cm-system}. Furthermore, the solutions to the system of equations
    \begin{equation}\label{eq:central-equations}
        \partial_t(\Pcal_+(\k_j)\Ucal_0) = \Pcal_+(\k_j)\left[\Rcal(\Ucal_0 + \Psib(\Ucal_0;M);M)\right], \quad j=1,\ldots,N,
    \end{equation}
    give rise to solutions to the full system \eqref{eq:cm-system} via $\Ucal = \Ucal_0 + \Psib(\Ucal_0;M)$. Additionally, the symmetries of the system \eqref{eq:cm-system} are preserved by the reduction function $\Psib$.
\end{theorem}

We now analyse the emergence of stationary periodic patterns for system \eqref{eq:cm-system} by considering the bifurcation of these on the center manifold provided by Theorem \ref{thm:center-manifold}. Therefore, let $\mu M_0 := M - M^*$ with $\mu > 0$ and we introduce the slow timescale \(T=\mu t\) and coordinates on the center manifold given by
\begin{equation}\label{eq:ansatz-U0}
    \Ucal_0(t,\x) = \sum_{j=1}^{N} A_j(\mu t) e^{i\k_j\cdot \x} \phib_+(k_m^*) + c.c.
\end{equation}
We note that $\mu$ plays the role of $\eps^2$ in Section \ref{sec:amplitude-equation} and we do not prescribe any a-priori scaling of $A_j$ with respect to $\mu$.
The center manifold equation can then, to leading order, be derived by employing the following strategy.
\begin{enumerate}[label=(\arabic*)]
    \item Determine \(\Psib = \sum_{\gammab \in \Gamma_h} \Psib_{\gammab} e^{i\gammab\cdot \x} + \Psi_0 \phib_-(0) + \left(\sum_{j=1}^{N} \Psi_j \phib_-(k_m^*) e^{i\k_j\cdot \x} + c.c.\right)\) to leading order via projection onto the non-central modes, that is
    \begin{equation}\label{eq:Psi-equations}
    \begin{split}
        \mu\partial_T \Psib_{\gammab} & = \Lcal_{M^*}\Psib_{\gammab} + \Pcal(\gammab)\left[\Rcal(\Ucal_0 + \Psib(\Ucal_0;M);M)\right], \quad \gammab \in \Gamma_h, \\
        \mu\partial_T \Psi_{0} & = \lambda_-(0)\Psi_0 + P_-(0)\left[\Rcal(\Ucal_0 + \Psib(\Ucal_0;M);M)\right], \\
        \mu\partial_T \Psi_{j} & = \lambda_-(k_m^*;M^*)\Psi_{j} + \Pcal_-(\k_j)\left[\Rcal(\Ucal_0 + \Psib(\Ucal_0;M);M)\right], \quad j=1,\ldots,N.
    \end{split}
    \end{equation}
    \item Insert the expansion into the central equations \eqref{eq:central-equations} obtained by projecting onto the central modes
    \begin{equation}\label{eq:A-central-equations}
        \mu\partial_T A_j = \Pcal_+(\k_j)\left[\Rcal(\Ucal_0 + \Psib(\Ucal_0;M);M)\right], \quad j=1,\ldots,N.
    \end{equation}
\end{enumerate}

\begin{remark}
    As pointed out above, although we use the same notation as in the formal derivation of the amplitude system, the amplitudes $A_j$ are not exactly the same as in Section \ref{sec:amplitude-equation}, where we explicitly include the scaling by the distance from the bifurcation point in the ansatz \eqref{eq:ansatz-physical}.
\end{remark}
Since the combinatorics of this strategy differ between the square and hexagonal lattice, we now present the calculations for both cases separately.

\subsection{Reduced equations on square lattice}\label{sec:stationary-reduced-square}

We begin with the first step of the strategy on the square lattice. For this, observe that \(\Psib\) is at least quadratic in \(A_1\), \(\bar{A}_1\), \(A_2\), \(\bar{A}_2\) and \(\mu\). As the goal is to describe the dynamics on the center manifold to leading order, it suffices to discern the leading-order terms of \(\Psib\) which are exactly quadratic in $(A_1,\bar{A}_1,A_2,\bar{A_2},\mu)$. First, we observe that, for \(\gammab\in \Gamma_h\) with \(d(0,\gammab) \geq 3\) it holds
\begin{equation}\label{eq:expansion-Psib_3}
    \Psib_{\gammab} = \Ocal\bigl(|(A_1,\bar{A}_1,A_2,\bar{A_2},\mu)|^3\bigr).
\end{equation}
Next, we turn to \(\gammab\in \Gamma_h\) with \(d(0,\gammab) = 2\). There exists a unique pair $j,\ell\in \{\pm 1, \pm 2\}$ such that $\gammab = \k_j+\k_\ell$. Then, we may compute for
\begin{equation*}
\begin{split}
    \Pcal(\gammab)\left[\Rcal(\Ucal_0 + \Psib(\Ucal_0;M);M)\right]
    & = (2-\delta_{j\ell})\Pcal(\gammab)\left[\Ncal_2(\phib_+(k_m^*) e^{i\k_j\cdot \x},\phib_+(k_m^*) e^{i\k_\ell\cdot \x};M^*)\right] A_jA_\ell \\
    & \quad + \Ocal\bigl(|(A_1,\bar{A}_1,A_2,\bar{A_2},\mu)|^3\bigr),
\end{split}
\end{equation*}
where $\delta_{j\ell} = 1$ if $j = \ell$ and $\delta_{j\ell} = 0$ otherwise, and obtain that
\begin{equation}\label{eq:expansion-Psib_2}
    \begin{split}
        0 & = \hat{\Lcal}_{M^*}(\k_j+\k_\ell) \Psib_{\gammab} + (2-\delta_{j\ell})\Pcal(\gammab)\left[\Ncal_2(\phib_+(k_m^*) e^{i\k_j\cdot \x},\phib_+(k_m^*) e^{i\k_\ell\cdot \x};M^*)\right]A_jA_\ell \\
        &\quad + \Ocal\bigl(|(A_1,\bar{A}_1,A_2,\bar{A_2},\mu)|^3\bigr).
    \end{split}
\end{equation}
We can solve this uniquely for \(\Psib_{\gammab}\) using that $\hat{\Lcal}_M^*(\k_j+\k_\ell)$ is invertible for \(\k_j + \k_\ell = \gammab \in \Gamma_h\). This completes the first equation in \eqref{eq:Psi-equations}. We now turn to the second equation and observe that
\begin{equation*}
\begin{split}
    \Pcal_-(0)\left[\Rcal(\Ucal_0 + \Psib(\Ucal_0;M);M)\right] & = 2 \sum_{j=1}^{2} P_-(0)\left[\Ncal_2(\phib_+(k_m^*)e^{i\k_j\cdot\x},\phib_+(k_m^*)e^{-i\k_j\cdot \x};M^*)\right]|A_j|^2 \\
    & \quad + \Ocal\bigl(|(A_1,\bar{A}_1,A_2,\bar{A_2},\mu)|^3\bigr).
\end{split}
\end{equation*}
Hence, we can determine \(\Psi_0\) uniquely by solving the equation
\begin{equation}\label{eq:expansion-Psib_0}
    0 = \lambda_-(0)\Psi_0 + 2 \sum_{j=1}^{2} P_-(0)\left[\Ncal_2(\phib_+(k_m^*)e^{i\k_j\cdot\x},\phib_+(k_m^*)e^{-i\k_j\cdot \x};M^*)\right]|A_j|^2 + \Ocal\bigl(|(A_1,\bar{A}_1,A_2,\bar{A_2},\mu)|^3\bigr).
\end{equation}
Here, note that $\lambda_-(0) = -\beta$ by explicit computation.
Finally, for the third equation in \eqref{eq:Psi-equations}, observe that
\begin{equation*}
    \Pcal_-(\k_j)\left[\Rcal(\Ucal_0 + \Psib(\Ucal_0;M);M)\right] = \Ocal\bigl(|(A_1,\bar{A}_1,A_2,\bar{A_2},\mu)|^3\bigr)
\end{equation*}
since \(\k_1\) and \(\k_2\) are non-resonant. Hence,
\begin{equation}\label{eq:expansion-Psib_1}
    \Psi_{j} = \Ocal\bigl(|(A_1,\bar{A}_1,A_2,\bar{A_2},\mu)|^3\bigr).
\end{equation}
Now, we can insert this expansion into the central equations \eqref{eq:A-central-equations}. Therefore, we compute the right-hand sides
\begin{equation*}
\begin{split}
    & \Pcal_+(\k_1)\left[\Rcal(\Ucal_0 + \Psib(\Ucal_0;M);M)\right]\\
    & = P_+(k_m^*)\left[\bigl(\hat{\Lcal}_M(k_m^*) - \hat{\Lcal}_{M^*}(k_m^*) \bigr)A_1\phib_+(k_m^*)\right] + 2\Pcal_+(\k_1)\left[\Ncal_2(\phib_-(0),\phib_+(k_m^*)e^{i\k_1\cdot\x};M^*)\Psi_0 A_1 \right. \\
    &\quad \left. + \Ncal_2(\Psib_{2\k_1}e^{2i\k_1\cdot \x},\phib_+(k_m^*) e^{-i\k_1\cdot \x};M^*)\bar{A}_1 + \Ncal_2(\Psib_{\k_1+\k_2}e^{i(\k_1+\k_2)\cdot \x},\phib_+(k_m^*) e^{-i\k_2\cdot \x};M^*)\bar{A}_2\right. \\
    & \quad \left. + \Ncal_2(\Psib_{\k_1-\k_2}e^{i(\k_1-\k_2)\cdot \x},\phib_+(k_m^*) e^{i\k_2\cdot \x};M^*)A_2\right]\\
    & \quad + 3 \Pcal_+(\k_1)\left[\Ncal_3(\phib_+(k_m^*)e^{i\k_1\cdot\x},\phib_+(k_m^*)e^{i\k_1\cdot\x},\phib_+(k_m^*)e^{-\k_1\cdot\x};M^*)\right]|A_1|^2A_1 \\
    & \quad + 6 \Pcal_+(\k_1)\left[\Ncal_3(\phib_+(k_m^*)e^{i\k_1\cdot\x},\phib_+(k_m^*)e^{i\k_2\cdot\x},\phib_+(k_m^*)e^{-\k_2\cdot\x};M^*)\right]|A_2|^2A_1 + \Ocal\bigl(|(A_1,\bar{A}_1,A_2,\bar{A_2},\mu)|^4\bigr).
\end{split}
\end{equation*}
Note that one obtains $\Pcal_+(\k_2)\left[\Rcal(\Ucal_0 + \Psib(\Ucal_0;M);M)\right]$ by permuting the roles of $\k_1$ and $\k_2$. Observe furthermore that
\begin{equation}\label{eq:expansion-linear-coeff}
    P_+(k_m^*)\left[\bigl(\hat{\Lcal}_M(k_m^*) - \hat{\Lcal}_{M^*}(k_m^*)\bigr)\phib_+(k_m^*)\right] = \mu M_0 \partial_M \lambda_+(k_m^*;M^*) + \Ocal(\mu^2) =: \mu M_0 \kappa + \Ocal(\mu^2) > 0
\end{equation}
for $\mu$ sufficiently small. Summarising, we obtain a system of the form
\begin{equation}\label{eq:CM-equation-squares}
    \begin{split}
        \mu \partial_T A_1 &= \mu M_0 \kappa A_1 + (K_0 \abs{A_1}^2 + K_1 \abs{A_2}^2) A_1+ \Ocal\bigl(|(A_1,\bar{A}_1,A_2,\bar{A_2},\mu)|^4\bigr), \\
        \mu \partial_T A_2 &= \mu M_0 \kappa A_2 + (K_1 \abs{A_1}^2 + K_0 \abs{A_2}^2) A_2+ \Ocal\bigl(|(A_1,\bar{A}_1,A_2,\bar{A_2},\mu)|^4\bigr),
    \end{split}
\end{equation}
where we note that the \emph{self-interaction coefficients}
\begin{equation}\label{eq:self-interaction-coeff}
\begin{split}
    K_0 & := 2\Pcal_+(\k_1)\left[\Ncal_2(\phib_-(0),\phib_+(k_m^*)e^{i\k_1\cdot\x};M^*) \nu_{0}
    + \Ncal_2(\nub_{2\k_1}e^{2i\k_1\cdot \x},\phib_+(k_m^*) e^{-i\k_1\cdot \x};M^*) \right] \\
    & \quad +3\Pcal_+(\k_1)\left[\Ncal_3(\phib_+(k_m^*)e^{i\k_1\cdot\x},\phib_+(k_m^*)e^{i\k_1\cdot\x},\phib_+(k_m^*)e^{-\k_1\cdot\x};M^*)\right], \\
    \nu_0 &:= \dfrac{2}{\beta} P_-(0) \Ncal_2(\phib_+(k_m^*)e^{i\k_j\cdot\x},\phib_+(k_m^*)e^{-i\k_j\cdot \x};M^*), \\
    \nub_{2\k_1} &:= -\hat{\Lcal}_M^*(2\k_1)^{-1} \Pcal(2\k_1)\left[\Ncal_2(\phib_+(k_m^*) e^{i\k_1\cdot \x},\phib_+(k_m^*) e^{i\k_1\cdot \x};M^*)\right]
\end{split}
\end{equation}
and \emph{cross-interaction coefficients}
\begin{equation}\label{eq:cross-interaction-coeff-square}
\begin{split}
    K_1 & := 2\Pcal_+(\k_1)\left[\Ncal_2(\phib_-(0),\phib_+(k_m^*)e^{i\k_1\cdot\x};M^*)\nu_0
    + \Ncal_2(\nub_{\k_1+\k_2}e^{i(\k_1+\k_2)\cdot \x},\phib_+(k_m^*) e^{-i\k_2\cdot \x};M^*) \right. \\
    & \qquad \qquad \left. + \Ncal_2(\nub_{\k_1-\k_2}e^{i(\k_1-\k_2)\cdot \x},\phib_+(k_m^*) e^{i\k_2\cdot \x};M^*) \right] \\
    & \quad
    + 6 \Pcal_+(\k_1)\left[\Ncal_3(\phib_+(k_m^*)e^{i\k_1\cdot\x},\phib_+(k_m^*)e^{i\k_2\cdot\x},\phib_+(k_m^*)e^{-\k_2\cdot\x};M^*)\right], \\
    \nub_{\k_1\pm \k_2} & := -2\hat{\Lcal}_{M^*}(\k_1\pm\k_2)^{-1} \Pcal(\k_1\pm\k_2)\left[\Ncal_2(\phib_+(k_m^*) e^{i\k_1\cdot \x},\phib_+(k_m^*) e^{\pm i\k_2\cdot \x};M^*)\right]
\end{split}
\end{equation}
are the same in both equations due to rotation and reflection symmetry, which is reflected on the center manifold by the symmetry $(A_1, A_2) \mapsto (A_2, A_1)$. Explicit algebraic expressions of the self-interaction and cross-interaction coefficient can be found in the Supplement Material, and a plot of the coefficients can be found in Figure \ref{fig:density-plot-square-coeffs}.

\begin{remark}\label{rem:square-invariant-sets}
    We observe that the higher-order terms in the $A_1$-equation $\eqref{eq:CM-equation-squares}_1$ are of the form
    \begin{equation*}
        A_1p(|A_1|^2,|A_2|^2)
    \end{equation*}
    for a polynomial $p(|A_1|^2,|A_2|^2) = \Ocal((|A_1|^2+|A_2|^2)^2)$. Indeed, this follows from translation invariance in $x$-direction and $y$-direction, respectively. In particular, this yields that $\{A_1 = 0\}$ and $\{A_2 = 0\}$ are invariant under the dynamics of \eqref{eq:CM-equation-squares}. Furthermore, due to rotation invariance the set $\{A_1 = A_2\}$ is also invariant.
\end{remark}

\begin{figure}[H]
    \centering
    \includegraphics[width=0.45\linewidth]{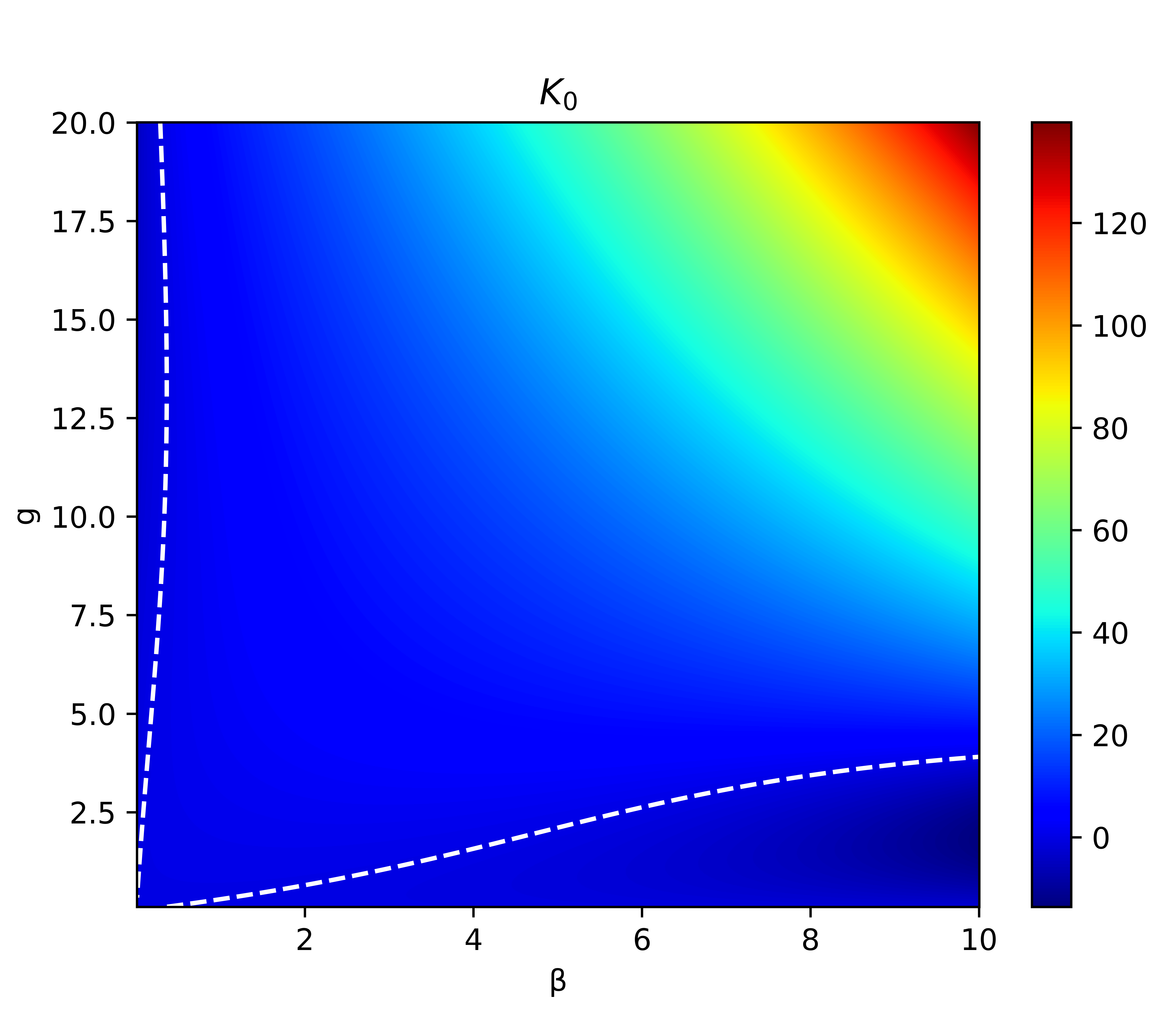}
    \includegraphics[width=0.45\linewidth]{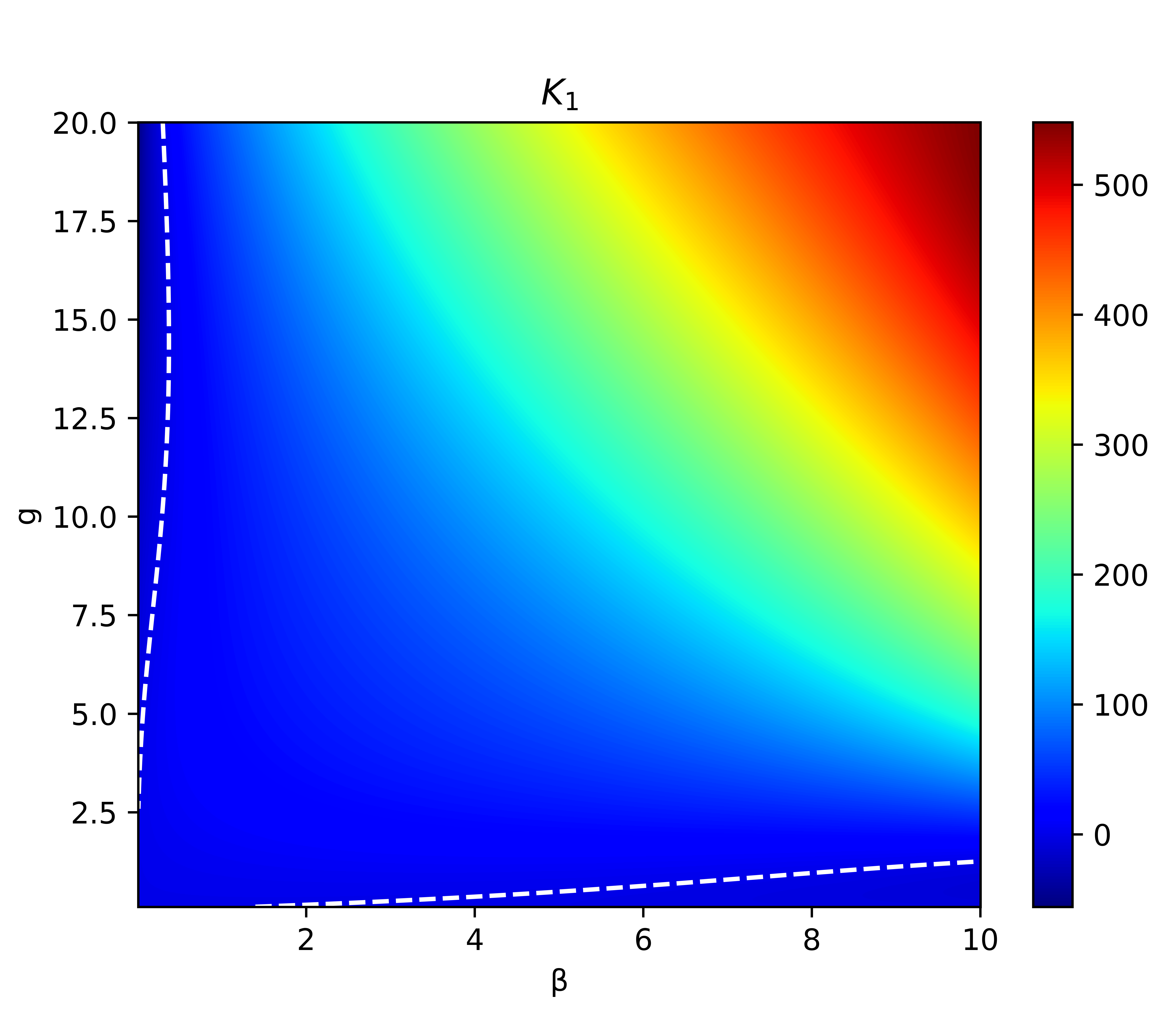}
    \caption{Density plots of the coefficients $K_0$ and $K_1$ for $0 < g < 20$ and $0 < \beta < 10$. The white, dashed lines indicate parameters, where the coefficients vanish.}
    \label{fig:density-plot-square-coeffs}
\end{figure}

\subsection{Reduced equations on hexagonal lattice}

Next, we derive the reduced equations for the hexagonal lattice. Since some of the calculations are similar to the square case we only point out the differences. Again, we start by calculating the leading order behaviour of $\Psib$, noting that $\Psib$ is now at least quadratic in $(A_1,\bar{A}_1,A_2,\bar{A}_2,A_3,\bar{A}_3,\mu)$. The equations \eqref{eq:expansion-Psib_3}--\eqref{eq:expansion-Psib_0} hold with obvious modifications. However, since the $(\k_1,\k_2,\k_3)$, which generate the hexagonal lattice are resonant, that is, $\k_1 + \k_2 + \k_3 = 0$, the $\Psi_j$ now have a quadratic contribution. This is given by
\begin{equation}\label{eq:expansion-Psib_1-hex}
\begin{split}
    0 & = \lambda_-(k_m^*;M^*) \Psi_j + 2\Pcal_-(\k_j)\left[\Ncal_2(\phib_+(k_m^*)e^{-i\k_\ell\cdot \x},\phib_+(k_m^*)e^{-i\k_n\cdot \x};M^*)\right]\bar{A}_\ell \bar{A}_n \\
    &\quad + \Ocal\bigl(|(A_1,\bar{A}_1,A_2,\bar{A_2},A_3,\bar{A_3},\mu)|^3\bigr),
\end{split}
\end{equation}
where \(j,\ell,n\in \{1,2,3\}\) are pairwise distinct.

With this, we can now derive the equation on the center manifold by computing that
\begin{equation*}
\begin{split}
    & \Pcal_+(\k_1)\left[\Rcal(\Ucal_0 + \Psib(\Ucal_0;M);M)\right]\\
    & = P_+(k_m^*)\left[\bigl(\hat{\Lcal}_M(k_m^*) - \hat{\Lcal}_{M^*}(k_m^*) \bigr)A_1\phib_+(k_m^*)\right] + 2\Pcal_+(\k_1)\left[\Ncal_2(\phib_+(k_m^*)e^{-i\k_2\cdot \x},\phib_+(k_m^*)e^{-i\k_3\cdot \x};M^*)\bar{A}_2\bar{A}_3\right]\\
    & \quad + 2\Pcal_+(\k_1)\left[\Ncal_2(\phib_-(0),\phib_+(k_m^*)e^{i\k_1\cdot\x};M^*)\Psi_0 A_1 + \Ncal_2(\bar{\Psi}_{3}\phib_-(k_m^*)e^{-i\k_3\cdot \x},\phib_+(k_m^*) e^{-i\k_2\cdot \x};M^*)\bar{A}_2\right.\\
    & \quad + \left. \Ncal_2(\bar{\Psi}_{2}\phib_-(k_m^*)e^{-i\k_2\cdot \x},\phib_+(k_m^*) e^{-i\k_3\cdot \x};M^*)\bar{A}_3+ \Ncal_2(\Psib_{2\k_1} e^{2i\k_1\cdot \x},\phib_+(k_m^*) e^{-i\k_1\cdot \x};M^*)\bar{A}_1\right. \\ 
    & \quad \left.
     + \Ncal_2(\Psib_{\k_1-\k_2}e^{i(\k_1-\k_2)\cdot \x},\phib_+(k_m^*) e^{i\k_2\cdot \x};M^*)A_2 +
    \Ncal_2(\Psib_{\k_1-\k_3}e^{i(\k_1-\k_3)\cdot \x},\phib_+(k_m^*) e^{i\k_3\cdot \x};M^*)A_3\right]\\
    & \quad + 3 \Pcal_+(\k_1)\left[\Ncal_3(\phib_+(k_m^*)e^{i\k_1\cdot\x},\phib_+(k_m^*)e^{i\k_1\cdot\x},\phib_+(k_m^*)e^{-\k_1\cdot\x};M^*)\right]|A_1|^2A_1 \\
    & \quad + 6 \Pcal_+(\k_1)\left[\Ncal_3(\phib_+(k_m^*)e^{i\k_1\cdot\x},\phib_+(k_m^*)e^{i\k_2\cdot\x},\phib_+(k_m^*)e^{-\k_2\cdot\x};M^*)\right]|A_2|^2A_1 \\
    & \quad + 6 \Pcal_+(\k_1)\left[\Ncal_3(\phib_+(k_m^*)e^{i\k_1\cdot\x},\phib_+(k_m^*)e^{i\k_3\cdot\x},\phib_+(k_m^*)e^{-\k_3\cdot\x};M^*)\right]|A_3|^2A_1 \\
    & \quad + \Ocal\bigl(|(A_1,\bar{A}_1,A_2,\bar{A_2},A_3,\bar{A}_3,\mu)|^4\bigr).
\end{split}
\end{equation*}
Again we note that the expansion for the $\k_2$ and $\k_3$ equations are obtained by cyclic permutation of the $(\k_1,\k_2,\k_3)$. The linear coefficient is again given by \eqref{eq:expansion-linear-coeff}. Summarising, we obtain a system of the form
\begin{equation}\label{eq:CM-equation-hexagon}
    \begin{split}
        \mu \partial_T A_1 &= \mu M_0 \kappa A_1 + N \bar{A}_2 \bar{A}_3 + (K_0 \abs{A_1}^2 + K_2 (\abs{A_2}^2 + \abs{A_3}^2) A_1 + \Ocal\bigl(|(A_1,\bar{A}_1,A_2,\bar{A_2},A_3,\bar{A}_3,\mu)|^4\bigr), \\
        \mu \partial_T A_2 &= \mu M_0 \kappa A_2 + N \bar{A}_1 \bar{A}_3 + (K_0 \abs{A_2}^2 + K_2 (\abs{A_1}^2 + \abs{A_3}^2) A_2 + \Ocal\bigl(|(A_1,\bar{A}_1,A_2,\bar{A_2},A_3,\bar{A}_3,\mu)|^4\bigr), \\
        \mu \partial_T A_3 &= \mu M_0 \kappa A_3 + N \bar{A}_1 \bar{A}_2 + (K_0 \abs{A_3}^2 + K_2 (\abs{A_1}^1 + \abs{A_2}^2) A_3 + \Ocal\bigl(|(A_1,\bar{A}_1,A_2,\bar{A_2},A_3,\bar{A}_3,\mu)|^4\bigr).
    \end{split}
\end{equation}
The self-interaction coefficient is the same as in the square case, cf. \eqref{eq:self-interaction-coeff}, the \emph{quadratic coefficient}
\begin{equation}\label{eq:quadratic-coeff}
    N = 2\Pcal_+(\k_1)\left[\Ncal_2(\phib_+(k_m^*)e^{-i\k_2\cdot \x},\phib_+(k_m^*)e^{-i\k_3\cdot \x};M^*)\right],
\end{equation}
and the \emph{cross-interaction coefficients}
\begin{equation}\label{eq:cross-interaction-coeff-hex}
\begin{split}
    K_2 &= 2\Pcal_+(\k_1)\left[\Ncal_2(\phib_-(0),\phib_+(k_m^*)e^{i\k_1\cdot\x};M^*)\nu_0 \right. \\
    & \qquad \qquad \left. + \Ncal_2(\nub_{\k_1-\k_2}e^{i(\k_1-\k_2)\cdot \x},\phib_+(k_m^*) e^{i\k_2\cdot \x};M^*) + \Ncal_2(\phib_-(k_m^*)e^{-i\k_3\cdot \x},\phib_+(k_m^*) e^{-i\k_2\cdot \x};M^*) \bar{\nu}_{3}\right] \\
    & \quad
    + 6 \Pcal_+(\k_1)\left[\Ncal_3(\phib_+(k_m^*)e^{i\k_1\cdot\x},\phib_+(k_m^*)e^{i\k_2\cdot\x},\phib_+(k_m^*)e^{-\k_2\cdot\x};M^*)\right],\\
    \nu_3 &= -2\lambda_-(k_m^*;M^*)^{-1}\Pcal_-(\k_j)\left[\Ncal_2(\phib_+(k_m^*)e^{-i\k_1\cdot \x},\phib_+(k_m^*)e^{-i\k_2\cdot \x};M^*)\right]
 \end{split}
\end{equation}
are the same in all equations due to rotation and reflection symmetry, which is preserved by the center manifold reduction. Algebraic expressions of the quadratic, self- and cross-interaction coefficients can be found in the Supplementary Material.

\begin{figure}[H]
    \centering
    \includegraphics[width=0.45\linewidth]{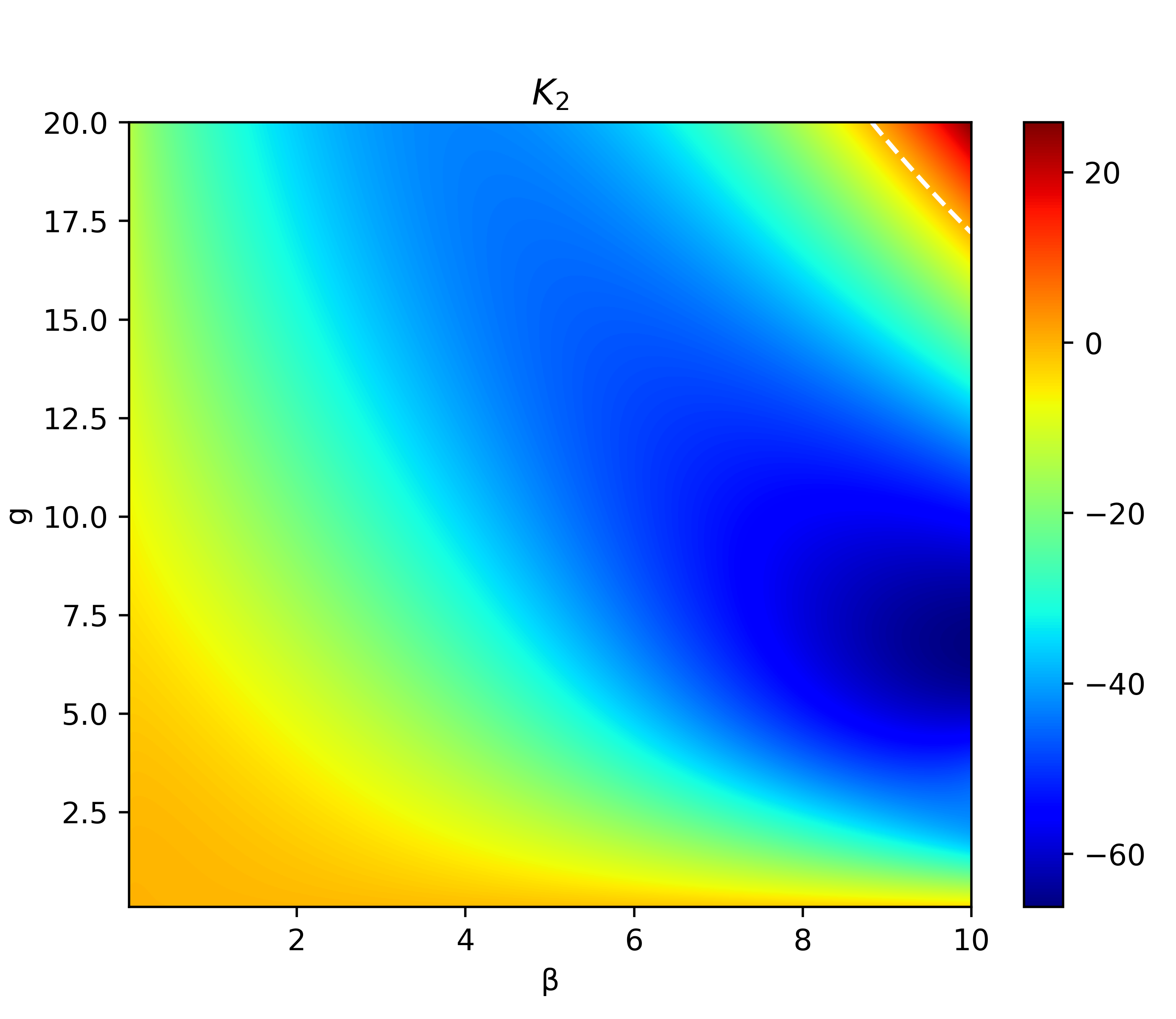}
    \caption{Density plots of the coefficients $K_2$ for $0 < g < 20$ and $0 < \beta < 10$. The white, dashed line in the top right indicates parameters, where the coefficients vanish.}
    \label{fig:density-plot-K2}
\end{figure}

\begin{remark}\label{rem:hexagon-invariant-sets}
    We observe that the higher-order terms in the $A_1$-equation $\eqref{eq:CM-equation-hexagon}_1$ are of the form 
    \begin{equation*}
        A_1p_1(|A_1|^2,|A_2|^2,|A_3|^2,q) + \bar{A_2}\bar{A_3} p_2(|A_1|^2,|A_2|^2,|A_3|^2,q)
    \end{equation*}
    with $q = A_1 A_2 A_3 + \bar{A}_1 \bar{A}_2 \bar{A}_3$ for polynomials $p_{j}(|A_1|^2,|A_2|^2,|A_3|^2,q) = \Ocal((|A_1|^2+|A_2|^2+|A_3|^2+q)^2)$, $j=1,2$. This follows again from translation invariance and the fact that $\k_1 + \k_2 + \k_3 = 0$, c.f.~\cite[Prop. 3.1]{buzano1983}.
    In particular, this yields that $\{A_1,A_2 = 0\}$, $\{A_2,A_3 = 0\}$ and $\{A_3,A_1=0\}$ are invariant under the dynamics of \eqref{eq:CM-equation-squares}. Furthermore, due to rotation invariance the set $\{A_1 = A_2 = A_3\}$ is also invariant.
\end{remark}

\subsection{Stationary solutions on the square lattice}\label{sec:stationary-square}

On the square lattice, we are interested in two patterns: roll waves, which lie in the invariant set $\{A_2 = 0\}$, and square patterns, which lie in the invariant set $\{A_1 = A_2\}$. We first discuss their existence in the reduced equations \eqref{eq:CM-equation-squares} without higher-order terms and then show their persistence for $\mu > 0$ sufficiently small.

Nontrivial roll waves with $A_1 \neq 0$ are solutions of
\begin{equation*}
    0 = \mu M_0 \kappa + K_0 \abs{A_1}^2,
\end{equation*}
which exist for $M_0 > 0$ if $K_0 < 0$ and for $M_0 < 0$ if $K_0 > 0$ since $\kappa > 0$ and $\mu > 0$ by assumption. Similarly, nontrivial square patterns with $A = A_1 = A_2 \neq 0$ are solutions of
\begin{equation*}
    0 = \mu M_0 \kappa + (K_0 + K_1) \abs{A}^2
\end{equation*}
which exist for $M_0 > 0$ if $K_0+K_1 < 0$ and for $M_0 < 0$ if $K_0 + K_1 > 0$ since $\kappa > 0$ and $\mu > 0$.

To show the persistence of roll waves, we use that according to Remark \ref{rem:square-invariant-sets} the higher-order terms preserve the $S_1$-invariance and thus, writing $A_1 = \rho_1 e^{i \alpha_1}$ we obtain
\begin{equation*}
    0 = \mu M_0 \kappa + K_0 \rho_1^2 + p(\rho_1^2), \quad \partial_T \alpha_1 = 0.
\end{equation*}
Linearising about $\rho_1 = \sqrt{-\tfrac{\mu M_0 \kappa}{K_0}}$ we find
\begin{equation*}
    2 \sqrt{-\mu M_0 \kappa K_0} + \Ocal(\mu^{3/2}) \neq 0
\end{equation*}
for $0 < \mu \ll 1$. Hence, persistence follows from an application of the implicit function theorem. Similarly, introducing polar coordinates $A = \rho e^{i\alpha}$ for the square patterns, we obtain
\begin{equation*}
    0 = \mu M_0 \kappa + (K_0 + K_1) \rho^2 + p(\rho^2).
\end{equation*}
Then, linearising about $\rho = \sqrt{{-\tfrac{\mu M_0 \kappa}{K_0 + K_1}}}$ gives
\begin{equation*}
    2 \sqrt{-\mu M_0 \kappa (K_0+K_1)} + \Ocal(\mu^{3/2}) \neq 0
\end{equation*}
and persistence follows from the implicit function theorem. To conclude, we obtain the following result.

\begin{theorem}[Patterns on square lattice]\label{thm:square-patterns}
    Let $\mu M_0 = M - M^*$. Then there exists a $\mu_0 > 0$ such that for all $0 < \mu < \mu_0$ the following holds.
    \begin{itemize}
        \item If $M_0 K_0 < 0$, the system \eqref{eq:cm-system} exhibits roll waves of the form
        \begin{equation*}
            \Ucal = 2 \sqrt{-\tfrac{\mu M_0 \kappa}{K_0}} \cos(k_m^* x) \phib_+(k_m^*) + \Ocal(\mu).
        \end{equation*}
        \item If $M_0 (K_0 + K_1) < 0$, the system \eqref{eq:cm-system} exhibits square patterns of the form
        \begin{equation*}
            \Ucal = 2\sqrt{-\tfrac{\mu M_0 \kappa}{K_0 + K_1}} (\cos(k_m^* x) + \cos(k_m^* y))\phib_+(k_m^*) + \Ocal(\mu).
        \end{equation*}
    \end{itemize}
\end{theorem}

\begin{remark}\label{rem:rotation-square}
    We point out that due to rotation and translation invariance of \eqref{eq:thin-film-equation}, rotated and translated versions of the patterns found in Theorem \ref{thm:square-patterns} are also stationary solutions of \eqref{eq:thin-film-equation}.
\end{remark}

The different parameter regimes, in which the sign conditions for the existence of the patterns are met, are shown in \ref{fig:existence-square-lattice}.

\begin{figure}[H]
    \centering
    \includegraphics[width=0.48\linewidth]{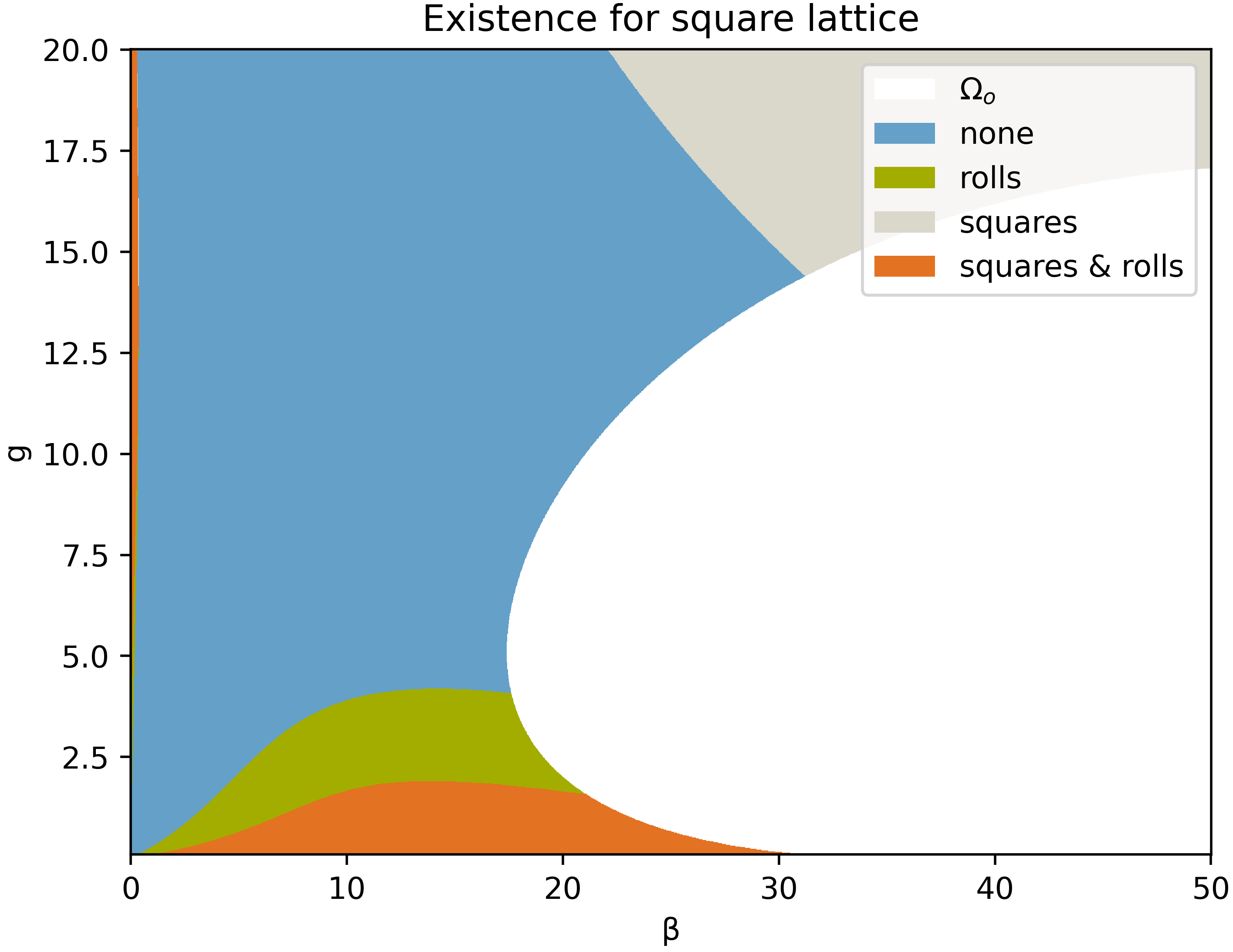}
    \hfill\includegraphics[width=0.48\linewidth]{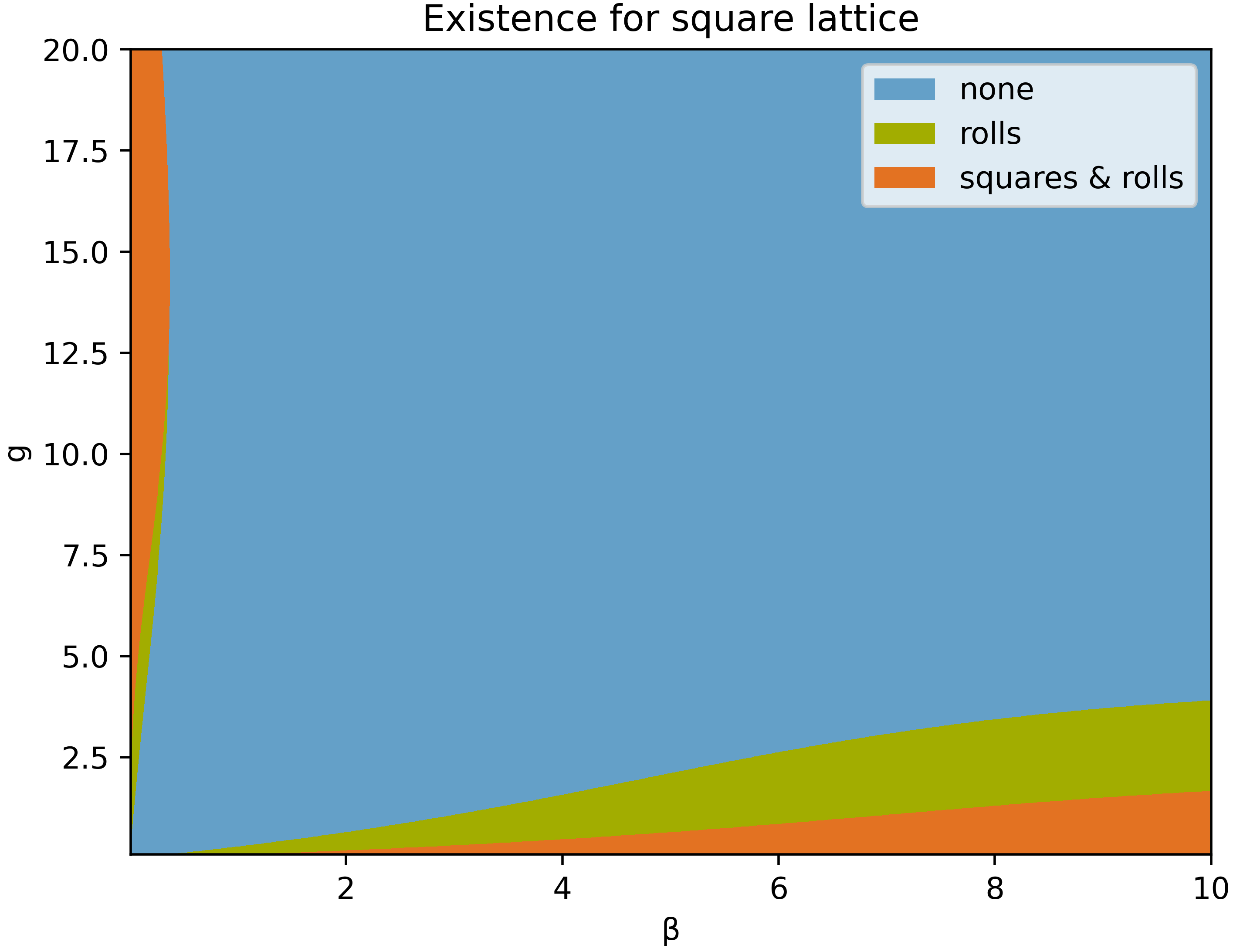}
    \caption{Depiction of the regions in parameter space, where the existence criteria from Theorem \ref{thm:square-patterns} are satisfied. The plots show the case of $M_0 > 0$ in two different parameter regimes (left: $0 < \beta < 50$; right: $0 < \beta < 10$). Note, that the case $M_0 < 0$ can be directly recovered from this.}
    \label{fig:existence-square-lattice}
\end{figure}

\begin{remark}\label{rem:pattern-selection-square}
    Since the dynamics on the center manifold given by \eqref{eq:CM-equation-squares} also includes the temporal dynamics of the patterns, we can predict pattern selection through a linear stability analysis of the fixed points corresponding to planar patterns. Similar to \cite{shklyaev2012}, we find that if $M_0 > 0$ square patterns are selected if $K_1 - K_0 < 0$, whereas roll waves are selected if $K_1 - K_0 > 0$. Here, we point out that in \cite{shklyaev2012}, the signs of $K_0$ and $K_1$ are reversed.
\end{remark}

Finally, Figure \ref{fig:selection-square-lattice} shows numerical plots of the regimes in which the sign conditions for the selection of the established patterns are satisfied.

\begin{figure}[H]
    \centering
    \includegraphics[width=0.48\linewidth]{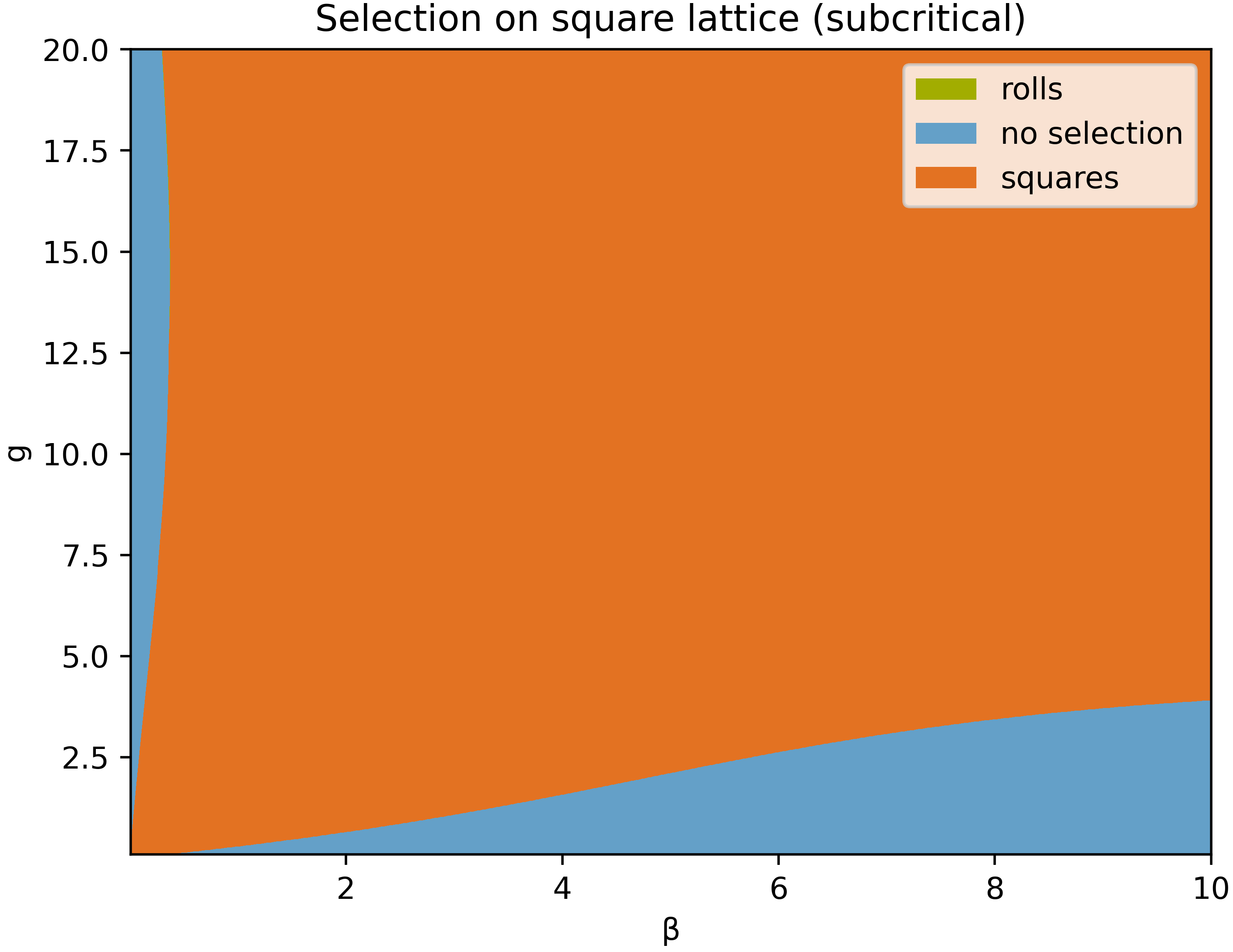}
    \hfill\includegraphics[width=0.48\textwidth]{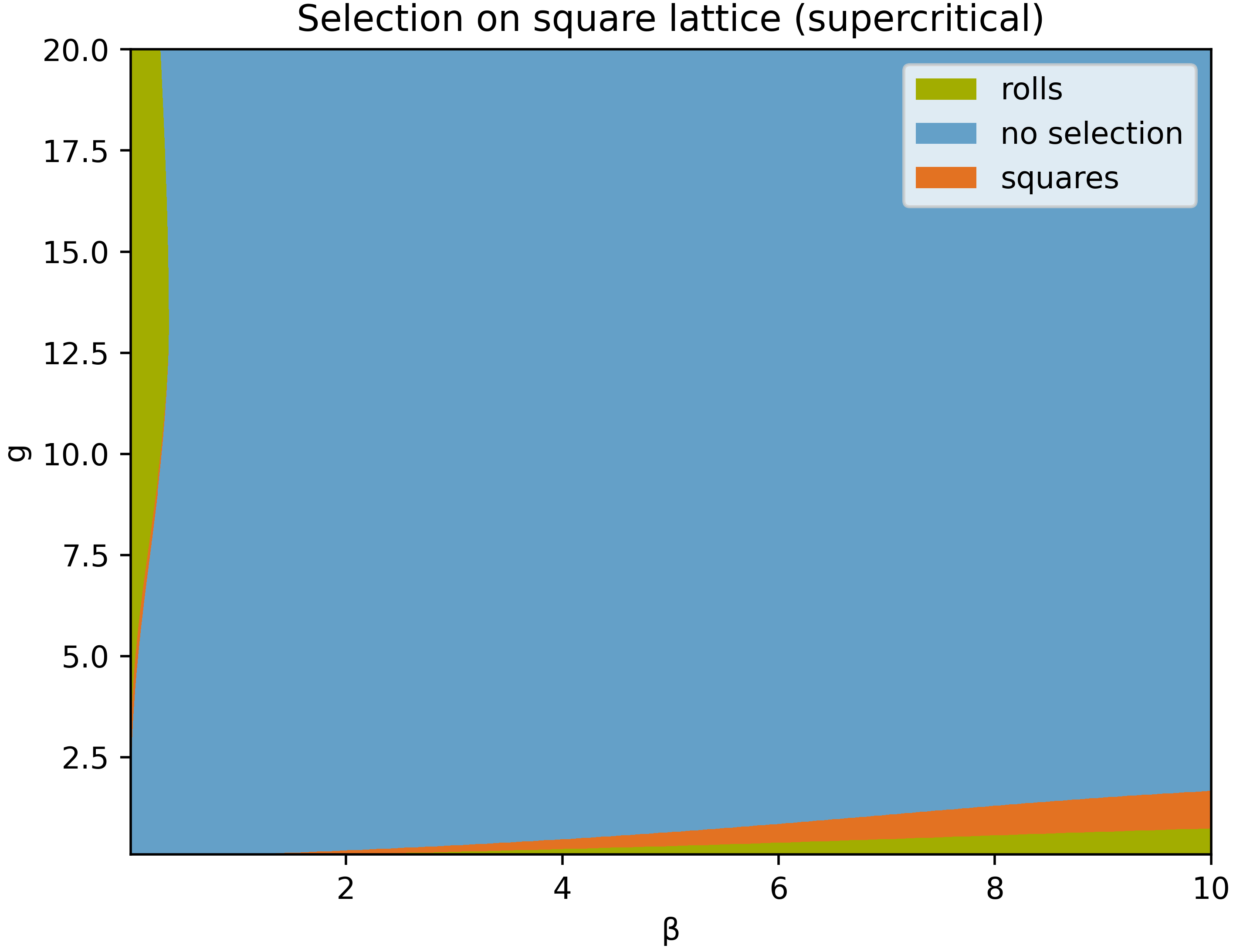}
    \caption{Depiction of the regions in parameter space, where the existence and selection criteria from Remark \ref{rem:pattern-selection-square} are satisfied. The selection plots show both the subcritical ($M_0 < 0$, bottom left) and the supercritical ($M_0 > 0$, bottom right) regime. In both cases, ``no selection'' denotes the regime, where roll waves or square patterns exist, but not both.}
    \label{fig:selection-square-lattice}
\end{figure}

\subsection{Stationary solutions on the hexagonal lattice}\label{sec:stationary-hex}

As pointed out in Remark \ref{rem:hexagon-invariant-sets} there are invariant sets on the hexagonal lattice. These model roll waves, where without loss of generality $A_1 \neq 0$ and $A_2 = A_3 = 0$ and hexagons $A_1 = A_2 = A_3$. Similar to the square lattice, we first discuss the existence of roll waves and hexagonal patterns in the reduced equations and then show their persistence.

The reduced equations \eqref{eq:CM-equation-hexagon} on the invariant set $\{A_2 = A_3 = 0\}$ reduce to
\begin{equation*}
    \mu \partial_T A_1 = \mu M_0 \kappa A_1 + K_0  \abs{A_1}^2A_1.
\end{equation*}
Since this is the same equation as for roll waves on the square lattice, we find that nontrivial roll waves exist if $M_0 K_0 < 0$. On the invariant set $\{A_1 = A_2 = A_3\}$ the reduced equations read to leading order as
\begin{equation*}
\begin{split}
    \mu \partial_T A_1 &= \mu M_0 \kappa A_1 + N A_1^2 + (K_0 + 2K_2 ) A_1^3,
\end{split}
\end{equation*}
where we can assume that $A_1 \in \R$ since the coefficients are real due to rotation symmetry. Hence, nontrivial hexagonal patterns with $A = A_1 = A_2 = A_3 \neq 0$ are given by solutions of
\begin{equation*}
    0 = \mu M_0 \kappa + N A + (K_0 + 2 K_2) A^2,
\end{equation*}
which leads to
\begin{equation*}
    A_\pm = \dfrac{-N \pm \sqrt{N^2 - 4\mu M_0 \kappa(K_0 + 2K_2)}}{2(K_0 + 2K_2)}
\end{equation*}
if $K_0 + 2K_2 \neq 0$ and $A = -\tfrac{\mu M_0 \kappa}{N}$ otherwise.
The former solutions are real provided that $N^2 - 4\mu M_0 \kappa (K_0 + 2K_2) > 0$. Moreover, Taylor expanding $A_+$ in $\mu$ yields $A_+ = -\tfrac{\mu M_0 \kappa}{N} + \Ocal(\mu^{3/2})$. If $N \neq 0$, hexagons emerge from a transcritical bifurcation, whereas if $N = 0$ and $K_0 + 2K_2 \neq 0$, they emerge from a pitchfork bifurcation. In the literature, two types of hexagons are distinguished, namely up-hexagons satisfying $A > 0$ and down-hexagons satisfying $A < 0$. Depending on the signs of $N$, $K_0 + 2K_2$ and $M_0$, both types may appear in this setting. For a full discussion of the bifurcation branches and their stability in the case $N > 0$ and $K_0 + 2K_2 < 0$, we refer to Section \ref{sec:modulation-center-hex}, Figure \ref{fig:bifurcation-diag-hex} and \cite{doelman2003}.

Since the persistence of roll waves follows with the same calculations as in the square case, we only discuss the persistence of hexagonal patterns here. Again, we can restrict to the invariant subspace $A_1 = A_2 = A_3 = A \in \R$, where the full equation is given by
\begin{equation*}
    0 = \mu M_0 \kappa + N A + (K_0+2K_2) A^2 + p(A^2),
\end{equation*}
where $p$ is a polynomial with $p(0) = p'(0)=0$, cf. Remark \ref{rem:hexagon-invariant-sets}. Then, linearising about the hexagonal patterns $A_\pm$ yields
\begin{equation*}
    N + 2(K_0 + 2K_2) A_\pm + \Ocal(A_\pm^3) = \pm \sqrt{N^2 - 4\mu M_0 \kappa(K_0 + 2K_2)} + \Ocal(A_\pm^3).
\end{equation*}
In particular, if $N \neq 0$, we find that $A_+ = \Ocal(\mu)$ and therefore, this fixed point persists under higher-order perturbations using the implicit function theorem. However, $A_- = \Ocal(N)$ and thus, we can guarantee its persistence only if $N = \Ocal(\sqrt{\mu})$. Therefore, we obtain the following result.
\begin{theorem}[Patterns on hexagonal lattice]\label{thm:hex-patterns}
    Let $\mu M_0 = M - M^*$. Then there exists a $\mu_0 > 0$ such that for all $0 < \mu < \mu_0$ the following holds.
    \begin{itemize}
        \item If $M_0 K_0 < 0$, the system \eqref{eq:cm-system} exhibits roll waves of the form
        \begin{equation*}
            \Ucal = 2 \sqrt{-\tfrac{\mu M_0 \kappa}{K_0}} \cos(k_m^* x) \phib_+(k_m^*) + \Ocal(\mu).
        \end{equation*}
        \item If $N \neq 0$, the system \eqref{eq:cm-system} exhibits hexagonal patterns of the form
        \begin{equation*}
            \Ucal = -2\dfrac{\mu M_0 \kappa}{N} \sum_{j=1}^{3}\cos(\k_j\cdot\x)\phib_+(k_m^*) + \Ocal(\mu^2).
        \end{equation*}
        \item If $N = \sqrt{\mu} N_0$ with $N_0 = \Ocal(1)$ for $\mu \rightarrow 0$, $K_0+2K_2\neq 0$ and $ (N_0^2-4 M_0 \kappa(K_0+2K_2)) > 0$, the system \eqref{eq:cm-system} exhibits hexagonal patterns of the form
        \begin{equation*}
            \Ucal_\pm = 2\sqrt{\mu}\dfrac{-N_0 \pm \sqrt{(N_0^2 - 4\kappa M_0 (K_0 + 2K_2)})}{2(K_0 + 2K_2)}\sum_{j=1}^{3}\cos(\k_j\cdot\x)\phib_+(k_m^*) + \Ocal(\mu).
        \end{equation*}
    \end{itemize}
\end{theorem}

\begin{remark}
    As in Remark \ref{rem:rotation-square}, rotated and translated versions of the patterns found in Theorem \ref{thm:hex-patterns} are also stationary solutions to \eqref{eq:thin-film-equation} due to rotation and translation invariance. 
\end{remark}

Figure \ref{fig:coeffs-hex} shows numerical plots of the regimes in which the sign conditions necessary for the existence of the found patterns are satisfied.

\begin{figure}[H]
    \centering
    \includegraphics[width=0.49\linewidth]{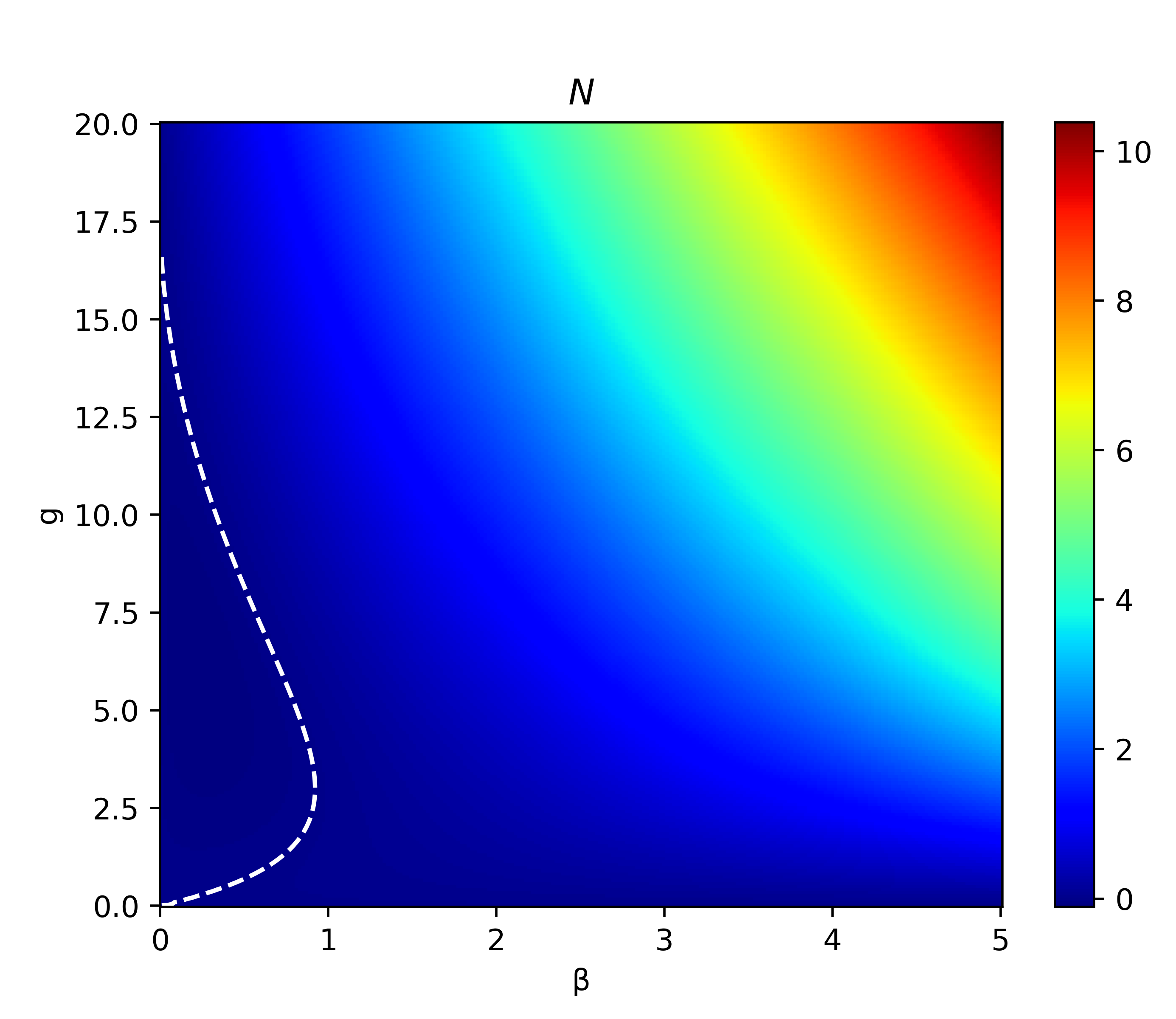} \hfill \includegraphics[width=0.49\linewidth]{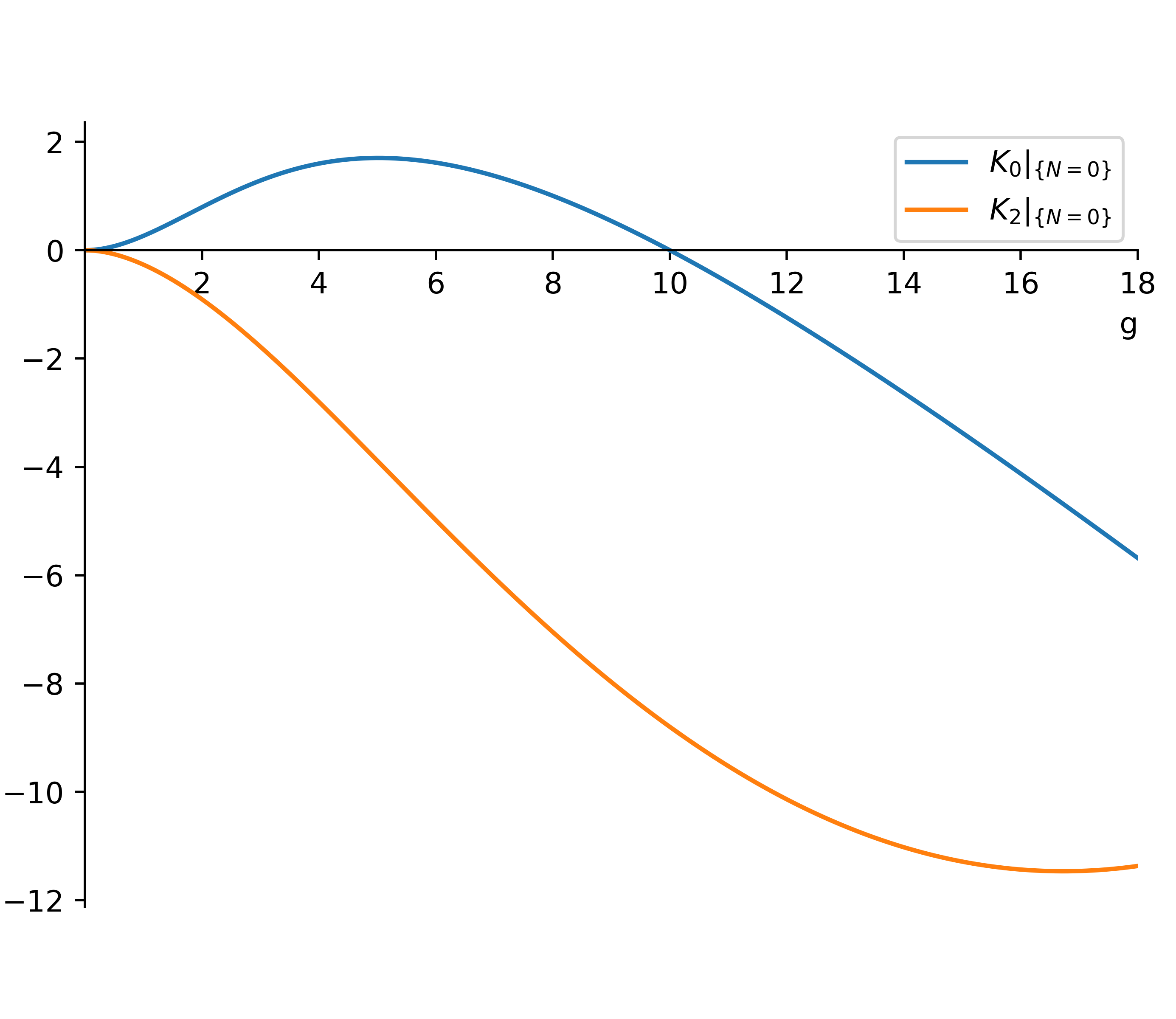}
    \caption{On the left: density plot of the coefficient $N$ for $0<g<20$ and $0<\beta <5$. The white dashed line indicates the parameter curve $\beta(g)$, where $N$ vanishes, see \eqref{eq:beta-von-g}. On the right: plot of the coefficients $K_0$ and $K_2$ on the curve $\{N=0\}$. Note that $K_2<0$ for every $g$, while $K_0$ changes sign at $g=10$.}
    \label{fig:coeffs-hex}
\end{figure}

\begin{remark}\label{rem:pattern-selection-hex}
    Using the temporal dynamics on the center manifold given by \eqref{eq:CM-equation-hexagon}, it is again possible to obtain criteria for pattern selection. However, since this is more involved we refrain from a detailed discussion and refer to \cite{golubitsky1984a,shklyaev2012}.
\end{remark}

\begin{remark}
    In addition to the hexagonal patterns and roll waves obtained in Theorem \ref{thm:hex-patterns}, there are also additional planar patterns, which lie neither in the invariant set $\{A_1 = A_2 = A_3\}$, nor in $\{A_2 = A_3 = 0\}$. These stationary solutions are referred to as \emph{mixed modes} and facilitate through a secondary bifurcation from the branch of roll waves, see Figure \ref{fig:bifurcation-diag-hex} as well as the analysis in Section \ref{sec:modulation-center-hex}. Continuing the branch of mixed modes further, it eventually crosses the branch of hexagons. After that, they are also referred to as \emph{false hexagons}. For a detailed discussion, we refer to \cite{golubitsky1988,hoyle2007}.
\end{remark}

\section{Fast-moving modulating fronts}\label{sec:modulating-fronts}

In this section, we study the existence of modulating travelling fronts for equation \eqref{eq:thin-film-equation}. Modulating travelling fronts are solutions $\Ucal(t,\x)$ to \eqref{eq:evolution-equation-U}, which are of the form 
\begin{equation}\label{eq:ansatz-modulating-front}
    \Ucal(t,\x) = \Vcal(x-ct,\x) = \Vcal(\xi,\p), \quad \xi = x-ct, \ \p = \x.
\end{equation}
Additionally, we assume that $\Vcal(\xi,\p)$ is periodic in $\p=(p_x,p_y)$ with respect to the dual lattice $\Gamma$ in Fourier space, see \eqref{eq:fourier-lattice}.
In particular, this means that we can write $\Vcal$ as a Fourier series
\begin{equation}\label{eq:Fourier-series}
    \Vcal(\xi,\p) = \sum_{\k\in \Gamma} \hat{\Vcal}_{\k}(\xi) e^{i\k\cdot \p}, \quad \hat{\Vcal}_{\k}(\xi) = \overline{\hat{\Vcal}_{-\k}(\xi)} \in \C^2
\end{equation}
since $\Vcal$ is real-valued. Further, we denote by \(\hat{\Vcal}(\xi) = (\hat{\Vcal}_{\k}(\xi))_{\k\in \Gamma}\). As in Sections \ref{sec:amplitude-equation} and \ref{sec:stationary} we restrict to the square lattice spanned by $\k_1 = k_m^*(1,0)$ and $\k_2 = k_m^*(0,1)$ and the hexagonal lattice spanned by $\k_1 = k_m^* (1,0)$, $\k_2 = \tfrac{k_m^*}{2}(-1,\sqrt{3})$ and $\k_3 = -\tfrac{k_m^*}{2}(1,\sqrt{3})$.

Finally, we assume that $\Vcal$ satisfies asymptotic boundary conditions in $\xi$. Assume that \eqref{eq:thin-film-equation} has two stationary solutions $\Ucal_1(\p)$ and $\Ucal_2(\p)$ which are periodic with respect $\Gamma$. Note that the existence of these solutions has been discussed in Section \ref{sec:stationary}. Then $\Vcal$ is a modulating travelling front connecting $\Ucal_1$ to $\Ucal_2$ if
\begin{equation*}
    \lim_{\xi \rightarrow -\infty} \Vcal(\xi,\p) = \Ucal_1(\p), \text{ and } \lim_{\xi \rightarrow +\infty} \Vcal(\xi,\p) = \Ucal_2(\p)
\end{equation*}
holds. The modulating travelling front $\Vcal$ then describes the spatial transition between the stationary states $\Ucal_1$ and $\Ucal_2$ in the form of the invasion of the state $\Ucal_2$ by the state $\Ucal_1$.

\begin{remark}
    \begin{enumerate}
        \item We restrict to modulating travelling fronts, which travel purely in one horizontal direction and which are periodic in the transverse horizontal direction. While it is possible to construct these fronts for any arbitrary horizontal direction, we choose, without loss of generality, solutions, which travel in $x$-direction.
        \item Note that since there is no discrete Fourier lattice, which contains both hexagonal and square patterns, it is not possible to model a transition between hexagons and squares using the theory presented here.
    \end{enumerate}
\end{remark}

The strategy for the construction of modulating fronts is as follows. First, we derive a spatial dynamics formulation by plugging the ansatz \eqref{eq:ansatz-modulating-front} into \eqref{eq:evolution-equation-U} and write this as an evolution system with respect to the spatial variable $\xi$. Then, we use periodicity in $\p$ to obtain an infinite-dimensional dynamical system for the Fourier modes of $\Vcal$. We then analyse the spectrum of the linearisation of this dynamical system in order to apply center manifold theory. It turns out that, since we are only considering fast-moving fronts, there are only finitely many central modes. Finally, we derive and analyse the dynamics of the finite-dimensional reduced equations on the center manifold both for the square lattice and the hexagonal lattice. Here, we construct both fixed points, which correspond to planar patterns in \eqref{eq:thin-film-equation}, and heteroclinic orbits, which correspond to modulating fronts in \eqref{eq:thin-film-equation}.

\subsection{Spatial dynamics formulation}

We plug the ansatz \eqref{eq:ansatz-modulating-front} into equation \eqref{eq:evolution-equation-U} to obtain
\begin{equation}\label{eq:modulating-thin-film-equation}
    \begin{pmatrix}
        1 & 0 \\
        0 & h
    \end{pmatrix} (-c\partial_\xi \Vcal) = \Fcal_M(\e_1 \partial_\xi + \nabla_\p)(\Vcal) 
\end{equation}
with $\e_1 = (1,0)^T$ and $\nabla_\p = (\partial_{p_x},\partial_{p_y})^T$. Here, we use that we can write $\Fcal_M$ as $\Fcal_M(\nabla)$ to denote the dependence on the spatial derivatives. Then, from the ansatz \eqref{eq:ansatz-modulating-front} we obtain that $\nabla = \nabla_\x \rightarrow \e_1 \partial_\xi + \nabla_\p$.
Using the notation $\Vcal = (h-1,\theta-1) =: (\tilde{h},\tilde{\theta})$, the highest order terms in $\Fcal$ with respect to $\partial_\xi$ can be written as
\begin{equation*}
    \Qcal(\tilde{h};M)\begin{pmatrix}
        \partial_{\xi}^4\tilde{h} \\ \partial_{\xi}^2\tilde{\theta}
    \end{pmatrix}, \quad \text{where } \Qcal(\tilde{h};M) = \begin{pmatrix}
        -\frac{(1+\tilde{h})^3}{3} & M \frac{(1+\tilde{h})^2}{2} \\
        -\frac{(1+\tilde{h})^4}{8} & (1+\tilde{h}) + M \frac{(1+\tilde{h})^3}{6}
    \end{pmatrix}.
\end{equation*}
Note that
\begin{equation}\label{eq:determinant-Q}
    \det \Qcal(\tilde{h};M) = -\frac{M(1+\tilde{h})^6}{144} + \frac{(1+\tilde{h})^4}{3} = -\frac{(1+\tilde{h})^4}{144} (M(1+\tilde{h})^2-48).
\end{equation}
In particular, for $\tilde{h} = 0$ it holds that $\det \Qcal(0;M) = -\tfrac{M-48}{144}$. Remark \ref{rem:Msmaller48} then implies that $\det \Qcal(0;M^*) > 0$ since $M^* < 48$. Due to continuity, we therefore find that $\Qcal(\tilde{h};M)$ is invertible in a neighborhood $\Ofrak_{g,\beta}\subset H^{\ell}_{\Gamma} \times \R$, $\ell > 0$, of $(0,M^*)$, which depends on the parameters $(g,\beta)$. Therefore, in $\Ofrak_{g,\beta}$ we can rewrite equation \eqref{eq:modulating-thin-film-equation} by multiplying with $\Qcal(\tilde{h};M)^{-1}$ to obtain
\begin{equation}\label{eq:spat-dyn-system}
    \begin{pmatrix}
        \partial_\xi^4 \tilde{h} \\
        \partial_\xi^2 \tilde{\theta}
    \end{pmatrix} = \Qcal(\tilde{h};M)^{-1}\Biggl( \begin{pmatrix}
        1 & 0 \\
        0 & 1+\tilde{h}
    \end{pmatrix} (-c \partial_\xi \Vcal) - \Fcal_M(\e_1\partial_\xi+\nabla_\p)(\Vcal) + \Qcal(\tilde{h};M) \begin{pmatrix}
        \partial_\xi^4 \tilde{h} \\
        \partial_\xi^2 \tilde{\theta}
    \end{pmatrix}\Biggr) =: \Gcal(\Vcal;M).
\end{equation}
Using the Fourier series for $\Vcal$ \eqref{eq:Fourier-series}, we obtain $\nabla_\p \rightarrow i\k$ with $\k \in \Gamma$ and find
\begin{equation*}
     \begin{pmatrix}
        \partial_\xi^4 \hat{h}_\k(\xi) \\
        \partial_\xi^2 \hat{\theta}_\k(\xi)
    \end{pmatrix} = \hat{\Gcal}_\k(\hat{\Vcal};M), \quad \k\in \Gamma.
\end{equation*}
Note that, by construction, $\hat{\Gcal}_\k$ contains at most third derivatives with respect to $\xi$. We also point out that $\hat{\Gcal}_\k$ depends on all Fourier modes $\hat{\Vcal} = (\hat{\Vcal}_\k)_{\k \in \Gamma}$ as the nonlinear terms yield convolutions in Fourier space.

We can write this as a first-order spatial dynamics system with respect to $\xi$ which reads as
\begin{equation}\label{eq:spat-dyn}
    \partial_\xi \hat{\Wcal}_\k(\xi) = \hat{\Lfrak}_M(\k) \hat{\Wcal}_\k(\xi) + \hat{\Nfrak}(\hat{\Wcal}(\xi);M,\k)
\end{equation}
for all $\k \in \Gamma$. Here $\hat{\Wcal}_\k = (\hat{h}_\k,\partial_\xi \hat{h}_\k,\partial_\xi^2 \hat{h}_\k \partial_\xi^3\hat{h}_{\k},\hat{\theta}_\k,\partial_\xi \hat{\theta}_\k)^T$, $\hat{\Wcal} = (\hat{\Wcal}_{\k})_{\k\in \Gamma}$, $\hat{\Lfrak}_M(\k) \in \C^{6\times 6}$ denotes the linear terms, and $\hat{\Nfrak}(\cdot;M,\k)$ denotes the nonlinearities. We note in particular that
\begin{equation}\label{eq:Nfrak}
    \hat{\Nfrak}(\cdot;M,\k) = (0,0,0,\hat{\Gcal}_\k^{\text{nl}}(\cdot;M)_1,0,\hat{\Gcal}_\k^{\text{nl}}(\cdot;M)_2)^T,
\end{equation}
where $\hat{\Gcal}_\k^{\text{nl}}$ denotes the nonlinear terms of $\hat{\Gcal}_\k$. Since $\k \in \Gamma$, the spatial dynamics system \eqref{eq:spat-dyn} is an infinite-dimensional dynamical system on $\hat{H}^{\ell}_{\Gamma}$, $\ell >0$, where we define the spaces
\begin{equation*}
\begin{split}
    \hat{H}^{\ell_1,\ell_2}_{\Gamma} & := \Bigl\{(\hat{\Wcal}_\k)_{\k\in \Gamma} \subset \C^6 : \hat{\Wcal}_{\k} = \overline{\hat{\Wcal}_{-\k}} , \ \sum_{\k\in \Gamma} (1+|\k|^2)^{\ell_1} |\hat{\Wcal}_{\k,1-4}|^2 + (1+|\k|^2)^{\ell_2} |\hat{\Wcal}_{\k,5-6}|^2<\infty\Bigr\} \\
    \hat{H}^{\ell}_{\Gamma} & := \hat{H}^{\ell,\ell}_{\Gamma},
\end{split}
\end{equation*}
where $\hat{\Wcal}_{\k,1-4} = (\hat{\Wcal}_{\k,1},\hat{\Wcal}_{\k,2},\hat{\Wcal}_{\k,3},\hat{\Wcal}_{\k,4})^T$ and $\hat{\Wcal}_{\k,5-6} = (\hat{\Wcal}_{\k,5},\hat{\Wcal}_{\k,6})^T$. By a slight abuse of notation, we also use the notation $\hat{H}_\Gamma^{\ell_1,\ell_2}$ for the corresponding space of functions defined by the Fourier series of coefficients in $\hat{H}_\Gamma^{\ell_1,\ell_2}$. The remainder of this section deals with constructing heteroclinic orbits of the system \eqref{eq:spat-dyn} using center manifold reduction. Therefore, as in Section \ref{sec:stationary}, we write \eqref{eq:spat-dyn} as
\begin{equation}\label{eq:spat-dyn-CM}
    \begin{split}
        \partial_\xi \hat{\Wcal}_\k(\xi) &=  \hat{\Lfrak}_{M^*}(\k) \hat{\Wcal}_\k(\xi) + \hat{\Rfrak}(\hat{\Wcal}(\xi);M,\k) , \\
        \partial_\xi M &= 0,
    \end{split}
\end{equation}
for all $\k \in \Gamma$, where $ \hat{\Rfrak}(\hat{\Wcal}(\xi);M,\k) := (\hat{\Lfrak}_M(\k) - \hat{\Lfrak}_{M^*}(\k)) \hat{\Wcal}_\k(\xi) + \hat{\Nfrak}(\hat{\Wcal}(\xi);M,\k)$.

\begin{remark}\label{rem:equivalence-spat-dyn-thin-film}
    We point out that the first-order spatial dynamics formulation \eqref{eq:spat-dyn-CM} and the system \eqref{eq:modulating-thin-film-equation} obtained from inserting the modulating front ansatz into the thin-film system \eqref{eq:thin-film-equation} are equivalent. We use this to determine the spectral properties of the linear operator $\hat{\Lfrak}_{M^*}$ from the spectral properties obtained in Section \ref{sec:linear-analysis} and relate the coefficients on the center manifold to the ones obtained in Sections \ref{sec:amplitude-equation} and \ref{sec:stationary}.
\end{remark}

\subsection{Spectral properties of $\Lfrak_{M^*}$}

To apply center manifold reduction to the system \eqref{eq:spat-dyn-CM}, we study the spectrum of the linear operator $\Lfrak_{M^*}$ given by 
\begin{equation}\label{eq:def-Lfrak}
    \Lfrak_{M^*} \Wcal = \sum_{\k \in \Gamma} \hat{\Lfrak}_{M^*}(\k) \hat{\Wcal}_\k e^{i\k \cdot \p}, \quad \Wcal = \sum_{\k\in\Gamma} \hat{\Wcal}_{\k} e^{i\k\cdot \p}.
\end{equation}
Since $\Lfrak_{M^*}$ acts diagonally on the Fourier modes $e^{i\k\cdot\p}$ the spectrum of $\Lfrak_{M^*}$ in $H_{\Gamma}^{\ell}$ is given by the union of the eigenvalues of $\hat{\Lfrak}_{M^*}(\k)$ over $\k \in \Gamma$. However, instead of calculating the eigenvalues of $\Lfrak_{M^*}(\k)$ directly, we aim to exploit the structures of the temporal problem \eqref{eq:thin-film-equation}, similar to \cite{hărăguş–courcelle1999}.

We start by linearising \eqref{eq:modulating-thin-film-equation} about $\Vcal = (0,0)^T$. This gives
\begin{equation*}
    -c\partial_\xi \Vcal = \Lcal_{M^*}(\e_1 \partial_\xi + \nabla_\p) \Vcal,
\end{equation*}
where
\begin{equation*}
    \Lcal_{M^*}(\mub) = 
    \begin{pmatrix}
        -\frac{1}{3}(\mub\cdot\mub)^2 + \bigl(\tfrac{g}{3} - \tfrac{M^*}{2}\bigr) \mub\cdot\mub & \frac{M^*}{2} \mub\cdot\mub \\
        - \frac{1}{8} (\mub\cdot\mub)^2 + \bigl(\tfrac{g}{8} - \tfrac{M^*}{6}\bigr) \mub\cdot\mub + \beta & \bigl(1+ \tfrac{M^*}{6}\bigr)\mub\cdot\mub - \beta
    \end{pmatrix}, \quad \mub \in \C^2.
\end{equation*}
Now we make the ansatz $\Vcal(\xi,\p) = \Vcal_{\k,\tilde{\mu}} e^{i\k\cdot \p} e^{\tilde{\mu}\xi}$ with \(\k\in \Gamma\) and \(\tilde{\mu}\in \C\), and obtain
\begin{equation*}
    -c\tilde{\mu} \Vcal_{\k,\tilde{\mu}} = \Lcal_{M^*}(\e_1 \tilde{\mu} + i\k) \Vcal_{\k,\tilde{\mu}}.
\end{equation*}
This equation has a nontrivial solution if and only if the dispersion relation
\begin{equation*}
    d(c,\tilde{\mu},\k) := \det\bigl(\Lcal_{M^*}(\e_1 \tilde{\mu} + i\k) + c\tilde{\mu}I\bigr), \quad \tilde{\mu}\in \C,\ \k\in \Gamma
\end{equation*}
vanishes. It turns out that the zeros of the dispersion relation characterise the spectrum of $\Lfrak_{M^*}$.

\begin{lemma}\label{lem:spectrum-dispersion-rel}
    Fix $\k \in \Gamma$ and $\lambda \in \C$. Then the following statements are equivalent:
    \begin{enumerate}
        \item\label{spec-alt1} $\lambda$ is an eigenvalue of $\hat{\Lfrak}_{M^*}(\k)$ with eigenvector $\hat{\phib} \in \C^6$.
        \item\label{spec-alt2} There exists a $\psib \in \C^2$ such that $\Vcal(\xi) = e^{\lambda \xi} \psib$ solves
        \begin{equation*}
            \begin{pmatrix}
                \partial_\xi^4 & 0 \\
                0 & \partial_\xi^2
            \end{pmatrix} \Vcal = D\Gcal(0;M^*)\Vcal.
        \end{equation*}
        \item\label{spec-alt3} It holds that $d(c,\lambda,\k) = 0$.
    \end{enumerate}
    Additionally, the relation
    \begin{equation}\label{eq:relation-eigenvectors}
        \hat{\phib} = \begin{pmatrix}
            1 & 0 \\
            \lambda & 0 \\
            \lambda^2 & 0 \\
            \lambda^3 & 0 \\
            0 & 1 \\
            0 & \lambda
        \end{pmatrix}\psib =: T(\lambda) \psib
    \end{equation}
    holds.
\end{lemma}
\begin{proof}
    \textbf{Step 1: \ref{spec-alt1} $\Longleftrightarrow$ \ref{spec-alt2}.} Assume that statement \ref{spec-alt1} is true. We recall that $\hat{\Lfrak}_{M^*}(\k)$ is a matrix of the form
    \begin{equation*}
        \hat{\Lfrak}_{M^*}(\k) = \begin{pmatrix}
            0 & 1 & 0 & 0 & 0 & 0 \\
            0 & 0 & 1 & 0 & 0 & 0 \\
            0 & 0 & 0 & 1 & 0 & 0 \\
            a_0 & a_1 & a_2 & a_3 & a_4 & a_5 \\
            0 & 0 & 0 & 0 & 0 & 1 \\
            b_0 & b_1 & b_2 & b_3 & b_4 & b_5
        \end{pmatrix}.
    \end{equation*}
    Hence, $(\lambda I - \hat{\Lfrak}_{M^*}(\k))\hat{\phib} = 0$ implies that $\hat{\phib}_j = \lambda \hat{\phib}_{j-1}$ for $j = 2,3,4$ and $\hat{\phib}_6 = \lambda \hat{\phib}_5$. Therefore, by defining $\psib_1 = \hat{\phib}_1$ and $\psib_2 = \hat{\phib}_5$ we find that $\hat{\phib}$ and $\psib$ satisfy \eqref{eq:relation-eigenvectors}. Next, we note that the first-order formulation
    \begin{equation}\label{eq:pf-spec-lemma-lin-first-order-spatdyn}
        \partial_\xi \hat{\Wcal}_\k = \hat{\Lfrak}_{M^*}(\k) \hat{\Wcal}_\k
    \end{equation}
    is equivalent to
    \begin{equation}\label{eq:pf-spec-lemma-lin-spatdyn}
        \begin{pmatrix}
                \partial_\xi^4 & 0 \\
                0 & \partial_\xi^2
            \end{pmatrix} \hat{\Vcal}_\k = D\hat{\Gcal}_\k(0;M^*)\hat{\Vcal}_\k
    \end{equation}
    through the relation $\hat{\Wcal}_\k = (\hat{h}_\k, \partial_\xi \hat{h}_\k, \partial_\xi^2 \hat{h}_\k, \partial_\xi^3 \hat{h}_\k, \hat{\theta}_\k, \partial_\xi \hat{\theta}_\k)^T$, where $\hat{\Vcal}_\k = (\hat{h}_\k,\hat{\theta}_\k)^T$. Since, by assumption, statement \ref{spec-alt1} holds true, the first-order formulation has a solution of the form $\hat{\Wcal}_\k(\xi) = e^{\lambda \xi} \hat{\phib}$ and therefore, $\hat{\Vcal}_\k(\xi) = e^{\lambda \xi} \psib$ solves \eqref{eq:pf-spec-lemma-lin-spatdyn} and the statement \ref{spec-alt2} holds.

    Conversely, we assume that \ref{spec-alt2} is true and let $\hat{\phib}$ be defined by \eqref{eq:relation-eigenvectors}. Then, by equivalence of the linear spatial dynamics system \eqref{eq:pf-spec-lemma-lin-spatdyn} and its first-order formulation \eqref{eq:pf-spec-lemma-lin-first-order-spatdyn} we find that $\hat{\Wcal}_\k(\xi) = e^{\lambda \xi} \hat{\phib}$, which shows that $\lambda$ is an eigenvalue of $\hat{\Lfrak}_{M^*}(\k)$ with eigenvector $\hat{\phib}$.

    \textbf{Step 2: \ref{spec-alt2} $\Longleftrightarrow$ \ref{spec-alt3}.} First, we calculate 
    \begin{equation*}
        D\hat{\Gcal}_\k(0;M^*)\hat{\Vcal}_\k = Q(0,M^*)^{-1}\left(-c \partial_\xi \hat{\Vcal}_\k - \Lcal_{M^*}(\e_1 \partial_\xi + i\k)\hat{\Vcal}_\k + Q(0;M^*) \begin{pmatrix}
            \partial_\xi^4 & 0 \\
            0 & \partial_\xi^2
        \end{pmatrix}\hat{\Vcal}_\k\right).
    \end{equation*}
    Since $Q(0;M^*)$ is invertible as $M^* < 48$, cf. \eqref{eq:determinant-Q} and Remark \ref{rem:Msmaller48}, we find that the linear spatial dynamics system \eqref{eq:pf-spec-lemma-lin-spatdyn} is equivalent to
    \begin{equation*}
        0 = \bigl(c I \partial_\xi + \Lcal_{M^*}(\e_1 \partial_\xi + i\k)\bigr)\hat{\Vcal}_\k.
    \end{equation*}
    Hence, \eqref{eq:pf-spec-lemma-lin-spatdyn} has a solution $\hat{\Vcal}_\k(\xi) = e^{\lambda \xi} \psib$ with $\psib \in \C^2 \setminus \{0\}$ if and only if
    \begin{equation}\label{eq:psib-is-eigenvector}
        0 = \bigl((c \lambda I + \Lcal_{M^*}(\e_1 \lambda + i\k)\bigr)\psib.
    \end{equation}
    Such a $\psib$ exists if and only if $d(c,\lambda,\k) = 0$ and hence statements \ref{spec-alt2} and \ref{spec-alt3} are equivalent.
\end{proof}

Lemma \ref{lem:spectrum-dispersion-rel} characterises the eigenvalues and eigenvectors of $\hat{\Lfrak}_{M^*}$ using the eigenvalues and eigenvectors of the temporal linearisation $\Lcal_{M^*}$, which have been studied in Section \ref{sec:linear-analysis}. In fact, we can also relate the geometric and algebraic multiplicities of eigenvalues on the imaginary axis, which are the relevant ones for the center manifold reduction.

\begin{lemma}\label{lem:multiplicity}
    Fix $\k \in \Gamma$. Then $i\gamma_0 \in i\R$ is a purely imaginary eigenvalue of $\hat{\Lfrak}_{M^*}(\k)$ if and only if $-ic\gamma$ is an eigenvalue of $\hat{\Lcal}_{M^*}(|\e_1 \gamma + \k|)$. Additionally, their geometric multiplicities coincide. Furthermore, if $i\gamma_0$ is a geometrically simple eigenvalue of $\hat{\Lfrak}_{M^*}(\k)$ and $\lambda(\gamma)$ is an eigenvalue of $\hat{\Lcal}_{M^*}(\e_1 \gamma + \k) + ic\gamma I$ with $\lambda(\gamma) = e_{0,m_0} (\gamma - \gamma_0)^{m_0} + \Ocal(|\gamma-\gamma_0|^{m_0+1})$ with $e_{0,m_0} \neq 0$, then $i\gamma_0$ has algebraic multiplicity $m_0$.
\end{lemma}
\begin{proof}
    Lemma \ref{lem:spectrum-dispersion-rel} yields that $i\gamma_0 \in i\R$ is a purely imaginary eigenvalue of $\hat{\Lfrak}_{M^*}(\k)$ if and only if $d(c,i\gamma_0,\k) = 0$ and thus
    \begin{equation*}
        \det\bigl(i c\gamma_0 I + \Lcal_{M^*}(\e_1 i \gamma_0 + i\k)\bigr) = 0.
    \end{equation*}
    Since $\e_1 i\gamma_0 + i \k$ has purely imaginary entries, we have $\Lcal_{M^*}(\e_1 i\gamma_0 + i\k) = \hat{\Lcal}_{M^*}(|\e_1 \gamma_0 + \k|)$ and $-ic\gamma_0$ is an eigenvalue of $\hat{\Lcal}_{M^*}(|\e_1 \gamma_0 + \k|)$. Following the proof of Lemma \ref{lem:spectrum-dispersion-rel}, the corresponding eigenvectors $\hat{\phib}$ for $\hat{\Lfrak}_{M^*}(\k)$ and the eigenvectors $\phib$ of $\hat{\Lcal}_{M^*}(|\e_1 \gamma + \k|)$ satisfy \eqref{eq:relation-eigenvectors}. Since the matrix $T(i\gamma) \in \C^{6 \times 2}$ has full column rank the dimensions of the corresponding eigenspaces are identical.

    To find the algebraic multiplicity, we use a result from \cite{afendikov1995}, as in \cite{hărăguş–courcelle1999}. To capture the behaviour for $i\gamma$ close to $i\gamma_0$, we make the modified ansatz
    \begin{equation*}
        \Ucal(t,\x) = \tilde{\Vcal}(t,\xi,\p).
    \end{equation*}
    Plugging this ansatz into \eqref{eq:evolution-equation-U} and linearising about $(0,0)$ we find that
    \begin{equation*}
        \partial_t \tilde{\Vcal} - c\partial_\xi \tilde{\Vcal} = \Lcal_{M^*}(\e_1\partial_\xi + \nabla_\p) \tilde{\Vcal}. 
    \end{equation*}
    Writing this as a first-order spatial dynamics system in $\xi$ yields
    \begin{equation*}
        \partial_\xi \tilde{\Wcal} = \tilde{\Lfrak}_{M^*}(\nabla_\p,\partial_t) \tilde{\Wcal},
    \end{equation*}
    where we note that $\tilde{\Lfrak}_{M^*}(i\k,0) = \hat{\Lfrak}_{M^*}(\k)$. We solve the linear equation for $\tilde{\Vcal}$ with the ansatz $\tilde{\Vcal} = e^{\lambda t + i\gamma \xi + i\k \cdot \p} \tilde{\Vcal}_{\k,\gamma,\lambda}$ with $\k\in \Gamma$, $\gamma\in \R$, and $\lambda \in \C$. This gives
    \begin{equation*}
        \lambda \tilde{\Vcal}_{\k,\gamma,\lambda} = \bigl(\hat{\Lcal}_{M^*}(|\e_1 \gamma + \k|) + ic\gamma I\bigr)\tilde{\Vcal}_{\k,\gamma,\lambda}.
    \end{equation*}
    The corresponding spatial dynamics problem then reads as
    \begin{equation*}
        i\gamma \tilde{\Wcal}_{\k,\gamma,\lambda} = \tilde{\Lfrak}_{M^*}(i\k,\lambda) \tilde{\Wcal}_{\k,\gamma,\lambda}.
    \end{equation*}
    Assuming that $i\gamma_0$ is a geometrically simple eigenvalue of $\tilde{\Lfrak}_{M^*}(i\k,\lambda_0)$ yields, similarly to the first part of this proof, that $\lambda_0$ is a geometrically simple eigenvalue of $\hat{\Lcal}_{M^*}(|e_1 \gamma_0 + \k|) + ic\gamma_0 I$. In particular, there exists a smooth eigenvalue curve $\lambda(\gamma)$ for $\gamma$ in a neighborhood of $\gamma_0$ with $\lambda(\gamma_0) = \lambda_0$. The result \cite[Thm.~2.1]{afendikov1995} then states that the algebraic multiplicity of $i\gamma_0$ is $m_0$ if and only if $\lambda(\gamma) = \lambda(\gamma_0) + e_{0,m_0}(\gamma - \gamma_0)^{m_0} + \Ocal(|\gamma - \gamma_0|^{m_0+1})$ with $e_{0,m_0} \neq 0$. In particular, since $i\gamma_0$ is an eigenvalue of $\hat{\Lfrak}_{M^*}(\k)$, which is geometrically simple by assumption, it holds that $\lambda(\gamma_0) = 0$. This concludes the proof.
\end{proof}

Combining Lemmas \ref{lem:spectrum-dispersion-rel} and \ref{lem:multiplicity} we can then conclude the following result on the spectrum of $\Lfrak_{M^*}$.

\begin{proposition}\label{prop:spat-dyn-spectrum}
    Let $c \neq 0$ and $\Gamma$ be a discrete Fourier lattice generated by $\k_1, \dots, \k_N$. Then the spectrum $\sigma(\Lfrak_{M^*})$ splits into a central part $\sigma_0 = \{0\}$ and a hyperbolic part $\sigma_h$. The hyperbolic part satisfies a spectral gap, that is, there exists a $\delta > 0$ such that $|\Re(\lambda_h)| \geq \delta$ for all $\lambda_h \in \sigma_h$. Additionally, the central eigenvector $\lambda = 0$ is semisimple with $(2N+1)$-dimensional eigenspace, which is spanned by $\hat{\phib}_+(0) := T(0)\phib_+(0)$ and $e^{i\k_j \cdot \x}\hat{\phib}_+(k_m^*) := e^{i\k_j \cdot \x} T(0)\phib_+(k_m^*)$ for $|j| = 1,\dots,N$.
\end{proposition}
\begin{proof}
    As mentioned above, the spectrum of $\Lfrak_{M^*}$ is given by the union of eigenvalues of $\hat{\Lfrak}_{M^*}(\k)$ over $\k \in \Gamma$.
    
    \noindent\textbf{Step 1: Characterisation of central part.} Using Lemma \ref{lem:spectrum-dispersion-rel}, the matrix $\hat{\Lfrak}_{M^*}$ has an imaginary eigenvalue $i\gamma_0$ if and only if $d(c,i\gamma_0,\k) = 0$. This holds if and only if $\hat{\Lcal}_{M^*}(|\e_1 \gamma_0 + \k|)$ has an eigenvalue $ic\gamma_0$. Recalling Section \ref{sec:linear-analysis}, the eigenvalues of $\hat{\Lcal}_{M^*}(k)$ are given by the curves $\lambda_\pm(k;M^*)$, which have strictly negative real part unless $k = 0,k_m^*$. At $k = 0,k_m^*$ we have $\lambda_+(k;M^*) = 0$ and $\lambda_-(k;M^*) < 0$. This yields $\gamma_0 = 0$ and therefore, $\sigma_0 = \{0\}$. Additionally, zero eigenvalues are obtained only from $\hat{\Lcal}_{M^*}(0)$ and $\hat{\Lcal}_{M^*}(\k_j)$ for $|j|=1,\ldots,N$. In particular, since $\lambda_+(k;M^*) \neq \lambda_-(k;M^*)$ for $k=0,k_m^*$, they are simple eigenvalues and therefore, $0$ is a geometrically simple eigenvalue of $\hat{\Lfrak}_{M^*}(\k)$ for $\k \in \{0,\k_j:|j|=1,\ldots,N\}$. Therefore, Lemma \ref{lem:multiplicity} applies. Since $c \neq 0$ the eigenvalue curve of $\hat{\Lcal}_{M^*}(\e_1 \gamma + \k) + ic\gamma I$ with $\lambda(0) = 0$ is given by $\lambda(\gamma) = \lambda_+(|\e_1 \gamma + \k|;M^*) + ic\gamma$. Using that $\lambda_+(k;M^*)$ has a quadratic root at $k = 0,k_m^*$, see Proposition \ref{prop:linear-monotonic}, $\lambda = 0$ is an algebraically simple eigenvalue of $\hat{\Lfrak}_{M^*}(\k)$ for $\k \in \{0,\k_j : |j|=1,\ldots,N\}$. Recalling the definition of $\Lfrak_{M^*}$, \eqref{eq:relation-eigenvectors} and \eqref{eq:psib-is-eigenvector}, the corresponding eigenvectors are given by $T(0)\psib_+(0)$ and $e^{i\k_j\cdot \x}T(0)\phib_+(k_m^*)$, respectively. This shows the second part of the proposition.

    \noindent\textbf{Step 2: Characterisation of hyperbolic part.} As shown above, $\lambda \in \C$ is a hyperbolic eigenvalue of $\Lfrak_{M^*}$ if and only if $\Lcal_{M^*}(\e_1 \lambda + i\k) + c\lambda I$ has a nontrivial kernel for some $\k \in \Gamma$. We show that there is a spectral gap by contradiction. Assume that there is a sequence $(\lambda_n,\k_n) \subset \C \times \Gamma$ such that $\lambda_n$ accumulates at a purely imaginary value $i \gamma_1$. By continuity of eigenvalues and Step 1, we find $\gamma_1 = 0$. Similarly, since the eigenvalues of $\hat{\Lcal}_{M^*}(\k)$ are bounded away from the imaginary axis uniformly in $\k \in \Gamma_h$ this also yields that for large enough $n \in \N$ it holds that $\k_n \in \Gamma_0$. Since, by assumption $\lambda_n \rightarrow 0$ for $n \rightarrow \infty$ we find that $\lambda_n = 0$ for a finite, but large $n \in N$. This is a contraction to $\lambda_n \in \sigma_h$. This completes the proof.
\end{proof}

\subsection{Center manifold theorem and reduced equations}

We now show that the spatial dynamics formulation \eqref{eq:spat-dyn-CM} has a finite-dimensional invariant center manifold. Using Proposition \ref{prop:spat-dyn-spectrum} we know that the spectrum of $\Lfrak_{M^*}$ decomposes into the stable part $\sigma_s$ in the left complex half-plane, the central part $\sigma_0$ on the imaginary axis and the unstable part $\sigma_u$ in the right complex half-plane. We now define the spectral projections onto the stable, central and unstable eigenspaces. For this we use that for each $\hat{\Lfrak}_{M^*}(\k)$, $k \in \Gamma$, we can define a stable projection $\Pfrak_s(\k)$, a central projection $\Pfrak_0(\k)$ and an unstable projection $\Pfrak_u(\k)$. Then, since $\Lfrak_{M^*}$ can be written as the direct sum of the finite-dimensional matrices $\hat{\Lfrak}_{M^*}(\k)$ over all $\k \in \Gamma$, see \eqref{eq:def-Lfrak}, the spectral projections on the stable eigenspace $\Pfrak_s$, central eigenspace $\Pfrak_0$, and unstable eigenspace $\Pfrak_u$ are well-defined as direct sums over the $\Pfrak_j(\k)$ with $j = s,0,u$. Using this, we denote the projection of $\Lfrak_{M^*}$ as $\Lfrak_{M^*,s}$ and $\Lfrak_{M^*,u}$. Since $\Lfrak_{M^*,s}$ and $\Lfrak_{M^*,u}$ are again direct sums of matrices with strictly negative and positive spectrum, respectively, they generate $C_0$-semigroups on $t \geq 0$ and $t \leq 0$ respectively and the bounds
\begin{equation*}
\begin{split}
    \norm{e^{t\Lfrak_{M^*,s}}}_{\hat{H}^\ell_{\Gamma} \rightarrow \hat{H}^{\ell}_\Gamma} \leq C e^{-\delta t}, \text{ for } t \geq 0, \\
    \norm{e^{t\Lfrak_{M^*,u}}}_{\hat{H}^\ell_{\Gamma} \rightarrow \hat{H}^{\ell}_\Gamma} \leq C e^{\delta t}, \text{ for } t \leq 0
\end{split}
\end{equation*}
hold, where $\delta > 0$ denotes the spectral gap from Proposition \ref{prop:spat-dyn-spectrum}. Therefore, \cite[Theorem 13.1.3]{schneider2017} applies, and the following center manifold theorem holds. To formulate the theorem, we introduce the function space $\Zgoth = \hat{H}^{\ell+4,\ell+2}_\Gamma$, which, similar to Section \ref{sec:stationary}, can be written as the direct sum $\Zgoth = \Zgoth_0 \oplus \Zgoth_h$. Here, $\Zgoth_0$ is the finite-dimensional space containing all $\hat{\Wcal} \in \Zgoth$ with $\Pfrak_0\hat{\Wcal} = \hat{\Wcal}$.

\begin{theorem}\label{thm:center-manifold-fronts}
    Let $\ell > 0$. Then there exists a neighborhood $\Ofrak$ of $(0,M^*)$ in $\Zgoth_0 \times \R$ and a smooth map $\Psib \colon \Ofrak \to \Zgoth_h$ such that
    \begin{equation*}
        \Psib(0;M^*) = D\Psib(0;M^*) = 0
    \end{equation*}
    and the center manifold is given by
    \begin{equation*}
        \Mcal_0 = \{\Wcal_0 + \Psib(\Wcal_0;M) : (\Wcal_0,M) \in \Ofrak\}
    \end{equation*}
    is invariant and contains all small bounded solutions of \eqref{eq:spat-dyn-CM}. Furthermore, solutions to the system of equations
    \begin{equation}\label{eq:central-equations-fronts}
        \partial_t(\Pfrak_0\Wcal_0) = \Pfrak_0\left[\Rcal(\Wcal_0 + \Psib(\Wcal_0;M);M)\right]
    \end{equation}
    give rise to solutions to the full system \eqref{eq:spat-dyn-CM} via $\Wcal = \Wcal_0 + \Psib(\Wcal_0;M)$. Additionally, the symmetries of the system \eqref{eq:spat-dyn-CM} are preserved by the reduction function $\Psib$.
\end{theorem}

In the remainder of this section, we derive the reduced equation \eqref{eq:central-equations-fronts} to leading order and analyse its dynamics. In particular, we focus on heteroclinic orbits and their persistence under higher-order perturbations to establish the existence of modulating fronts in the full system. For this, we introduce the rescaling $M-M^* = \eps^2 M_0$ with $M_0 \in \R$, which also guarantees that for $\eps$ sufficiently small, $M = M^* + \eps^2 M_0$ is contained in the neighborhood of $M^*$, where the center manifold result holds. We also introduce a slow spatial scale $\Xi = \eps^2 \xi$ and make the ansatz
\begin{equation}\label{eq:ansatz-Wcal}
    \Wcal = \sum_{j = 1}^N \eps A_j(\Xi) e^{i\k_j \cdot \p} \hat{\phib}_+(k_m^*) + c.c. + \eps^2 A_0(\Xi) \hat{\phib}_+(0) + \Psib(A_j, A_0;\eps)
\end{equation}
with $A_j(\Xi) \in \C$ and $A_0(\Xi) \in \R$ and
\begin{equation}\label{eq:ansatz-Psib}
     \Psib(A_j,A_0;\eps) = \sum_{\gammab \in \Gamma_h} \Psib_{\gammab} e^{i\gammab \cdot \p} + \Psib_0 + \sum_{j = 1}^N \Psib_j e^{i\k_j \cdot \p} + c.c.
\end{equation}
We point out that the $\eps$-scaling of the slow spatial scale $\Xi$ is the same as the rescaling used in Section \ref{sec:dynamics-amplitude-equations} to obtain the travelling wave ODEs for fast-moving fronts.

\begin{remark}
    Observe that in this ansatz, it is $A_j = \Ocal(1)$ since we encode the anticipated scaling in $\eps$ explicitly into the ansatz. In particular, we restrict to small amplitude solutions, which are also the ones that can be captured using the center manifold. Indeed, the square patterns constructed in Theorem \ref{thm:square-patterns} have small amplitude of order $\eps$ and therefore fit this scaling. In contrast, to guarantee that the hexagonal patterns in Theorem \ref{thm:hex-patterns} have the correct scaling, we need to restrict to the parameter regime where the quadratic coefficient $N$ is of order $\eps$.
\end{remark}

In order to relate the coefficients in the reduced equations on the center manifold to the ones in the amplitude equations in Section \ref{sec:amplitude-equation} we recall that \eqref{eq:spat-dyn-CM} is equivalent to the system
\begin{equation*}
    \begin{split}
        0 &= \begin{pmatrix}
            1 & 0 \\
            0 & 1+\tilde{h}
        \end{pmatrix} (-c \partial_\xi \Vcal) - \Fcal_M(\e_1\partial_\xi+\nabla_\p)(\Vcal), \\
        0 &= \partial_\xi M.
    \end{split}
\end{equation*}
Applying the Fourier transform in $\p$ on this system and splitting $\Fcal_M$ into linear and nonlinear parts, we obtain
\begin{equation} \label{eq:spat-dyn-CM-modfront-Fourier}
    \begin{split}
        0 &= \begin{pmatrix}
            1 & 0 \\
            0 & 1+\hat{\tilde{h}}
        \end{pmatrix} * (-c \partial_\xi \hat{\Vcal}_\k) + \Lcal_M(\e_1 \partial_\xi + i\k)\hat{\Vcal}_\k + \hat{\Ncal}(\e_1 \partial_\xi + i\gammab;\k)(\hat{\Vcal}), \\
        0 &= \partial_\xi M
    \end{split}
\end{equation}
for $\k \in \Gamma$. Note that the wave number $\gammab \in \Gamma$ in the nonlinearities represents a placeholder for different Fourier wave numbers appearing in convolution terms. We recall that $\hat{\Wcal}_\k$ can be reconstructed from $\hat{\Vcal}_\k = (\hat{\Vcal}_{\k,1}, \hat{\Vcal}_{\k,2})^T$ via the invertible operator $\Tcal \colon \Zgoth \to \Zcal$
\begin{equation}\label{eq:Tcal-operator}
    \hat{\Wcal}_\k = (\Tcal^{-1} \Vcal)_{\k} := (\hat{\Vcal}_{\k,1},\partial_{\xi}\hat{\Vcal}_{\k,1},\partial_{\xi}^2\hat{\Vcal}_{\k,1},\partial_{\xi}^3\hat{\Vcal}_{\k,1},\hat{\Vcal}_{\k,2},\partial_{\xi}\hat{\Vcal}_{\k,2}).
\end{equation}

We now outline the strategy to derive the reduced equations on the center manifold. Since we want to use the properties of the physical system \eqref{eq:spat-dyn-CM-modfront-Fourier} and relate the coefficients of the reduced equations to the ones obtained in Sections \ref{sec:amplitude-equation} and \ref{sec:stationary}, we aim to derive a system for $\Tcal \Wcal$ instead of a system for the full spatial dynamics variable $\Wcal$. The main challenge for this is to show that after applying the map $\Tcal$ to the ansatz \eqref{eq:ansatz-Wcal}, the leading-order terms remain unchanged, and $\Tcal\Wcal$ is of the form as in Section \ref{sec:stationary}, cf.~\eqref{eq:ansatz-U0}. Afterwards, we can proceed to derive the reduced equations similarly to Sections \ref{sec:amplitude-equation} and \ref{sec:stationary} by using that nonlinearities including $\Xi$-derivatives are of higher order in $\eps$.

We recall that the central eigenvectors $\hat{\phib}_+(\k)$ of $\hat{\Lfrak}_{M^*}(\k)$ are characterised by the central eigenvectors $\phib_+(\k)$ of $\hat{\Lcal}_{M^*}(\k)$, see Proposition \ref{prop:spat-dyn-spectrum}. In particular, we have that $\phib_+(\k_j) = \Tcal \hat{\phib}_+(\k_j)$ and $\phib_+(0) = \Tcal \hat{\phib}_+(0)$. Therefore, in order to determine the leading order dynamics of the central modes $A_0$ and $A_j$, $j = 1, \dots, N$, we project \eqref{eq:spat-dyn-CM-modfront-Fourier} onto the corresponding eigenspaces using the spectral projections $P_+(\k_j)$ of $\hat{\Lcal}_{M^*}(\k_j)$ for $j = 1, \dots, N$, and $P_+(0)$ of $\hat{\Lcal}_{M^*}(0)$, respectively. The main challenge here is the fact that $\Tcal \Zgoth_h \cap \Zcal_0 \neq \emptyset$. That is, after restricting the hyperbolic part $\Psib$ via $\Tcal$, we can obtain nontrivial contributions in the central eigenspaces. In particular, 
\begin{equation*}
\begin{split}
    \Tcal\Psib_0 &= \tilde{\Psi}_0^- \phib_-(0) + \tilde{\Psi}_0^+ \phib_+(0), \\
    \Tcal\Psib_j &= \tilde{\Psi}_j^- \phib_-(\k_j) + \tilde{\Psi}_j^+ \phib_+(\k_j),
\end{split}
\end{equation*}
with $j = 1,\dots, N$, where the central contributions $\tilde{\Psi}_0^+$ and $\tilde{\Psi}_j^+$ are, in general, nontrivial. However, the following lemma proves that these central contributions are of higher order.

\begin{lemma}\label{lem:reduction-expansion}
    Let $\Wcal = \Wcal_0 + \Psib$ be a solution of \eqref{eq:spat-dyn-CM} with $\Psib$ given by \eqref{eq:ansatz-Psib}. Then there exist $\Psi_0^\pm, \Psi_j^\pm = \Ocal(1)$ for $\eps \rightarrow 0$, $j=1,\ldots,N$, such that
    \begin{equation*}
        \begin{split}
            \Tcal\Psib_0 &= \eps^2 \Psi_0^- \phib_-(0) + \eps^4 \Psi_0^+\phib_+(0), \\
            \Tcal\Psib_j &= \eps^2 \Psi_j^- \phib_-(\k_j) + \eps^2 \Psi_j^+ \phib_+(\k_j).
        \end{split}
    \end{equation*}
\end{lemma}

\begin{proof}
    Using the fact that $\Psib$ is at least quadratic in its arguments and that the central part of $\Wcal$ scales at least with $\eps$, see \eqref{eq:ansatz-modulating-front}, we find that $\Psib = \Ocal(\eps^2)$. Hence, the only part of the lemma that is not obvious is the $\eps$ scaling of the $\phib_+(0)$-contribution of $\Tcal\Psib_0$.

    Recalling the derivation of the spatial dynamics system \eqref{eq:spat-dyn-CM}, we can write
    \begin{equation*}
        \partial_\xi \hat{\Wcal}_\k = \Qfrak_0^{-1} \Qfrak_0 \hat{\Lfrak}_{M^*}(\k) + \hat{\Rfrak}(\hat{\Wcal}_\k;M,\k),
    \end{equation*}
    where $\Qfrak_0 \in \R^{6 \times 6}$ is the action of $\Qcal(0;M^*)$ in the spatial-dynamics formulation, i.e.
    \begin{equation*}
        \Qfrak_0 = \begin{pmatrix}
            1 & 0 & 0 & 0 & 0 & 0 \\
            0 & 1 & 0 & 0 & 0 & 0 \\
            0 & 0 & 1 & 0 & 0 & 0 \\
            -\tfrac{1}{3} & -\tfrac{1}{3} & -\tfrac{1}{3} & -\tfrac{1}{3} & \tfrac{M^*}{2} & \tfrac{M^*}{2} \\
            0 & 0 & 0 & 0 & 1 & 0 \\
            -\tfrac{1}{8} & -\tfrac{1}{8} & -\tfrac{1}{8} & -\tfrac{1}{8} & 1 +\tfrac{M^*}{6} & 1 + \tfrac{M^*}{6},
        \end{pmatrix}
    \end{equation*}
    which satisfies $\det(\Qfrak_0) = \det(\Qcal(0;M^*))$ and is thus invertible. Using this formulation, we find that
    \begin{equation*}
        \Qfrak_0 \hat{\Lfrak}_{M^*}(0) = \begin{pmatrix}
            0 & 1 & 0 & 0 & 0 & 0 \\
            0 & 0 & 1 & 0 & 0 & 0 \\
            0 & 0 & 0 & 1 & 0 & 0 \\
            0 & -c & \tfrac{g}{3} - \tfrac{M^*}{2} & 0 & 0 & 0 \\
            0 & 0 & 0 & 0 & 0 & 1 \\
            \beta & 0 & \tfrac{g}{8} - \tfrac{M^*}{6} & 0 & -\beta & -c
        \end{pmatrix}.
    \end{equation*}    
    We first note that $\hat{\Lfrak}_M(0) - \hat{\Lfrak}_{M^*}(0) \hat{\Wcal}_0 = \Ocal(\eps^4)$, since $\Wcal$ is given by \eqref{eq:ansatz-modulating-front}.
    Next, we argue that $\Qfrak_0(I - \Pfrak_0)\hat{\Nfrak}(\hat{\Wcal;M,0}) = (0,0,0,\ast,0,\ast) + \Ocal(\eps^4)$. For that, we observe that quadratic terms in \eqref{eq:spat-dyn-system} are given by 
    \begin{equation*}
        \begin{split}
            &\tilde{\Qcal}_1 \tilde{h} \Bigg(-c\partial_\xi \Vcal - \Lcal_{M^*}(\e_1 \partial_\xi + \nabla_\p) \Vcal + \Qcal(0;M^*) \begin{pmatrix}
                \partial_\xi^2 \tilde{h} \\ \partial_\xi^2 \tilde{\theta}
            \end{pmatrix}\Bigg) \\ 
            &\qquad+ \Qcal(0;M^*)^{-1} \Bigg(-c \begin{pmatrix}
                0 & 0 \\ 0 & \tilde{h}
            \end{pmatrix} \partial_\xi \Vcal - \Ncal_2(\e_1\partial_\xi + \nabla_\p)(\Vcal,\Vcal;M^*)\Bigg),
        \end{split}
    \end{equation*}
    where we have used the Taylor expansion $\Qcal(\tilde{h};M^*)^{-1} = \Qcal(0;M^*)^{-1} + \tilde{\Qcal}_1 \tilde{h} + \Ocal(\tilde{h}^2)$. We first point out that each $\xi$-derivative adds an additional $\eps^2$, and thus, these terms are at least of order $\eps^4$ since $\Vcal = \Ocal(\eps)$. Next, we perform a Fourier transform in $\p$ to obtain that the leading order quadratic nonlinearity at $\k = 0$ is given by
    \begin{equation*}
        - \sum_{\gammab \in \Gamma} \tilde{\Qcal}_1 \hat{\Vcal}_{-\gammab,1} \hat{\Lcal}_{M^*}(\gammab) \hat{\Vcal}_{\gammab} - \Qcal(0;M^*)^{-1} \hat{\Ncal}_2(\hat{\Vcal},\hat{\Vcal};M^*,0).
    \end{equation*}
    Now, we note that the first term is in fact of order $\eps^4$ by using that
    \begin{equation*}
        \Vcal = \Tcal \Wcal = \eps \sum_{j = 1}^N A_j e^{i\k_j \cdot \p} \Tcal \hat{\phib}_+(\k_j) + \Ocal(\eps^2)
    \end{equation*}
    and the fact that $\Tcal\hat{\phib}_+(\k_j) = \phib_+(\k_j)$, see Lemma \ref{lem:spectrum-dispersion-rel}. Therefore, $\eps^2$ terms can only arise as quadratic combinations of the $A_j$, that is they are of the form
    \begin{equation*}
        \tilde{\Qcal}_1 \phib_+(\k_j)_1 \hat{\Lcal}_{M^*}(\k_j) \phib_+(\k_j) |A_j|^2 = 0,
    \end{equation*}
    which vanishes since $\phib_+(\k_j)$ is the eigenvector to the $0$-eigenvalue of $\hat{\Lcal}_{M^*}(\k_j)$. Hence, all $\eps^2$-contributions vanish and the lowest order terms are of order $\eps^4$.

    By construction of the $6d$ spatial dynamics system \eqref{eq:spat-dyn-CM-modfront-Fourier}, we thus find that
    \begin{equation*}
        \begin{split}
            \Pfrak_0 \hat{\Nfrak}_2(\hat{\Wcal},\hat{\Wcal};M,0) &= \dfrac{\lrangle{\Qcal(0;M^*)^*\phib_+^*(0),\Qcal(0;M^*)^{-1} \hat{\Ncal}_2(\Tcal\hat{\Wcal},\Tcal\hat{\Wcal};M^*,0)}}{\lrangle{\hat{\phib}_+^*(0),\hat{\phib}_+(0)}} \hat{\phib}_+(0) + \Ocal(\eps^4) \\
            &= \dfrac{\lrangle{(1,0)^T, \hat{\Ncal}_2(\Tcal\hat{\Wcal},\Tcal\hat{\Wcal};M^*,0)}}{\lrangle{\hat{\phib}_+^*(0),\hat{\phib}_+(0)}} \hat{\phib}_+(0) + \Ocal(\eps^4)
        \end{split}
    \end{equation*}
    where we use that $\hat{\Nfrak}_2$ is constructed using \eqref{eq:Nfrak} and we obtained the relation of $\hat{\phib}_+^*(0)$ and $\phib_+^*(0)$ by a direct computation from the adjoint eigenvalue problem, which gives in particular that $(\hat{\phib}_+^*(0)_4,\hat{\phib}_+^*(0)_6)^T = \Qcal(0;M^*)^* \phib_+^*(0)$. Finally, using that the first component of $\hat{\Ncal}_2(\cdot, \cdot;M^*,0)$ is in divergence form and thus has a leading $\partial_\xi$ derivative, we find that $\Pfrak_0 \hat{\Nfrak}_2(\hat{\Wcal},\hat{\Wcal};M,0) = \Ocal(\eps^4)$ since $\partial_\xi = \eps^2\partial_\Xi$ and $\Tcal\hat{\Wcal} = \Ocal(\eps)$.

    To conclude we find that $(I-\Pfrak_0) \hat{\Nfrak}_2(\hat{\Wcal},\hat{\Wcal};M,0) = (0,0,0,\ast,0,\ast) + \Ocal(\eps^4)$ and therefore, using the form of $\Qfrak_0 \hat{\Lfrak}_{M^*}(0)$, the leading order of $\Psib_0$ is in $\operatorname{span}(\e_1, \e_5)$. In particular, it holds that
    \begin{equation*}
        0 = \Pfrak_0 \Psib_0 = \lrangle{\hat{\phib}_+^*(0),\Psib_0} = \lrangle{\Tcal\hat{\phib}_+^*(0), \Tcal\Psib_0} + \Ocal(\eps^4).
    \end{equation*}
    Using the form of $\hat{\Lfrak}_{M^*}(0)$, a direct computation shows that 
    \begin{equation*}
        \Tcal\hat{\phib}_+^*(0) = \big(Q(0;M^*)^{-1} (c I)\big)^* Q(0;M^*)^* \phib_+^*(0) = c \begin{pmatrix}
            1 \\ 0
        \end{pmatrix},
    \end{equation*}
    where we also recall that $\phib_+^*(0) = (1,0)^T$.
    This shows that the leading-order term of $\Psib_0$ is in the span of $\e_5$ and thus $\Tcal\Psib_0 = \eps^2 \Psi_0^- \phib_-(0) + \Ocal(\eps^4)$. This concludes the proof.
\end{proof}

Using that the $A_j$, $j = 0,\dots,N$ in the ansatz \eqref{eq:ansatz-Wcal} depend on the slow spatial scale $\Xi = \eps^2 \xi$, we write \eqref{eq:spat-dyn-CM-modfront-Fourier} as
\begin{equation}
    0 = \begin{pmatrix}
        1 & 0 \\
        0 & 1+\hat{\tilde{h}}
    \end{pmatrix} * (-c \eps^2 \partial_\Xi \hat{\Vcal}_\k) + \Lcal_M(\e_1 \eps^2 \partial_\Xi + i\k)\hat{\Vcal}_\k + \hat{\Ncal}(\e_1 \eps^2 \partial_\Xi + i\gammab;\k)(\hat{\Vcal}).
\end{equation}
Since $\hat{\Ncal}$ depends smoothly on both arguments, we can first expand with respect to its second argument and obtain
\begin{equation*}
    \hat{\Ncal}(\e_1 \eps^2 \partial_\Xi + i\gammab;\k)(\hat{\Vcal}) = \hat{\Ncal}_2(\e_1 \eps^2 \partial_\Xi + i\gammab;\k)(\hat{\Vcal},\hat{\Vcal}) + \hat{\Ncal}_3(\e_1 \eps^2 \partial_\Xi + i\gammab;\k)(\hat{\Vcal},\hat{\Vcal},\hat{\Vcal}) + \Ocal(\|\Vcal\|^4),
\end{equation*}
where $\hat{\Ncal}_2(\e_1 \eps^2 \partial_\Xi + i\gammab;\k)$ is a bilinear and $\hat{\Ncal}_3(\e_1 \eps^2 \partial_\Xi + i\gammab;\k)$ is a trilinear form. Next, we expand $\hat{\Ncal}_2$ and $\hat{\Ncal}_3$ in $\eps$ and find that
\begin{equation}\label{eq:reduction-nonlinear-terms}
    \begin{split}
        \hat{\Ncal}_2(\e_1 \eps^2 \partial_\Xi + i\gammab;\k)(\Tcal\hat{\Wcal},\Tcal\hat{\Wcal}) & = \hat{\Ncal}_2(i\gammab;\k)(\Tcal\hat{\Wcal},\Tcal\hat{\Wcal}) + \Ocal(\eps^4), \\
        \hat{\Ncal}_3(\e_1 \eps^2 \partial_\Xi + i\gammab;\k)(\Tcal\hat{\Wcal},\Tcal\hat{\Wcal},\Tcal\hat{\Wcal}) &= \hat{N}_3(i\gammab;\k)(\Tcal\hat{\Wcal},\Tcal\hat{\Wcal},\Tcal\hat{\Wcal}) + \Ocal(\eps^5).
    \end{split}
\end{equation}
Here, we use that $\Tcal \hat{\Wcal} = \Ocal(\eps)$. In particular, we note that 
\begin{equation*}
    \begin{split}
        \Ncal_2(\Tcal\Wcal,\Tcal\Wcal) & = \sum_{\k \in \Gamma}\hat{\Ncal}_2(i\gammab;\k)(\Tcal\hat{\Wcal},\Tcal\hat{\Wcal}), \\
        \Ncal_3(\Tcal\Wcal,\Tcal\Wcal,\Tcal\Wcal) & = \sum_{\k \in \Gamma} \hat{N}_3(i\gammab;\k)(\Tcal\hat{\Wcal},\Tcal\hat{\Wcal},\Tcal\hat{\Wcal})
    \end{split}
\end{equation*}
with $\Ncal_2$ and $\Ncal_3$ from \eqref{eq:def-nonlinearities}. This allows us to easily relate combinatorial constants to the ones in the formal amplitude equations that are derived in Section \ref{sec:amplitude-equation}. In addition, following Lemma \ref{lem:conservation-law} we find that
\begin{equation*}
    P_+(0)\hat{\Ncal}(\e_1 \eps^2 \partial_\Xi + i\gammab;0)(\Tcal\hat{\Wcal}) = \eps^6  \sum_{j = 1}^N (\kappa_0+\kappa_1k_{j,1}^2) \partial_\Xi^2 (|A_j|^2) + \Ocal(\eps^{6}).
\end{equation*}
Here, we note that the remainder is of order $\eps^6$ since it contains at least one derivative due to the conservation-law structure and relevant quadratic terms can only occur due to interactions of $\Psib_{\gammab}$ and $\Psib_{-\gammab}$ with $\gammab \in \Gamma_h\cup \{0\}$.

\begin{remark}
    Since $\Ocal(\eps^6)$-terms are irrelevant for the dynamics on the center manifold, we do not elaborate further on the structure of the remainder even though it is of the same order in $\eps$ as the first term. In fact, a symmetrisation argument shows that the remainder is of higher order, following the ideas of Lemma \ref{lem:conservation-law}.
\end{remark}

This discussion reveals that the relevant combinatorics for the nonlinear terms follows in the same fashion as in the derivation of the amplitude equation, see Section \ref{sec:amplitude-equation}. We are left with the analysis of the linear terms. The main result needed here is the fact that
\begin{equation}\label{eq:dispair-Taylor}
    P_+(\k_j) \Lcal_{M^*}(\e_1 \eps^2 \partial_\Xi + i\k_j) \phib_+(\k_j) = \Ocal(\eps^4)
\end{equation}
for $j = 1,\dots, N$. For the Fourier mode $\k = 0$, this expansion is straightforward due to reflection symmetry. Indeed, the expansion \eqref{eq:dispair-Taylor} follows from the following expansion
\begin{equation*}
    \begin{split}
        & P_+(\k_j) \Lcal_{M^*}(\e_1 \eps^2 \partial_\Xi + i\k_j) \phib_+(\k_j) = \bigl(\lambda_+(\e_1 \eps^2 \partial_\Xi + i\k_j;M^*) - \lambda_+(\k_j;M^*)\bigr)\phib_+(\k_j)\\
        &\qquad+ P_+(\k_j) \bigl(\Lcal_{M^*}(\e_1 \eps^2 \partial_\Xi + i\k_j) - \lambda_+(\e_1 \eps^2 \partial_\Xi + i\k_j;M^*)\bigr) \bigl(\phib_+(\k_j) - \phib_+(\e_1 \eps^2 \partial_\Xi + i\k_j)\bigr).
    \end{split}
\end{equation*}
The first term is of order $\eps^4$ since $\lambda_+$ has a quadratic root at $\k_j$, see Proposition \eqref{prop:linear-monotonic}. Using Taylor expansion (in $\eps^2$) for the second term we find that
\begin{equation*}
    \begin{split}
        &P_+(\k_j) \bigl(\Lcal_{M^*}(\e_1 \eps^2 \partial_\Xi + i\k_j) - \lambda_+(\e_1 \eps^2 \partial_\Xi + i\k_j;M^*)\bigr) \bigl(\phib_+(\k_j) - \phib_+(\e_1 \eps^2 \partial_\Xi + i\k_j)\bigr) \\
        &= \eps^2 P_+(\k_j) \bigl(\Lcal_{M^*}(i\k_j) - \lambda_+(\k_j;M^*)\bigr) \partial_{\eps^2}\phib_+(\k_j) + \Ocal(\eps^4) = \Ocal(\eps^4)
    \end{split}
\end{equation*}
since the projection operator $P_+(\k_j)$ commutes with $\Lcal_{M^*}(i\k_j)$. Here, we also used that $\lambda_+$ and $\phib_+$ from Proposition \ref{prop:linear-monotonic} can be analytically continued and therefore can be evaluated at $\e_1 \eps^2 \partial_\Xi + i\k_j$ with the property that
\begin{equation*}
    \Lcal_{M^*}(\e_1 \eps^2 \partial_\Xi + i\k_j) \phib_+(\e_1 \eps^2 \partial_\Xi + i\k_j) = \lambda_+(\e_1 \eps^2 \partial_\Xi + i\k_j;M^*) \phib_+(\e_1 \eps^2 \partial_\Xi + i\k_j).
\end{equation*}

To conclude the discussion of the linear terms, we show that
\begin{equation}\label{eq:full-dispair-Taylor}
    P_+(\k_j)\big(\Lcal_M(\e_1 \eps^2 \partial_\Xi + i\k_j) - \Lcal_{M^*}(\e_1 \eps^2 \partial_\Xi + i\k_j)\big)\phib_+(\k_j;M^*) = \eps^2 \partial_M\lambda_+(\k_j;M^*) \phib_+(\k_j;M^*) + \Ocal(\eps^4).
\end{equation}
This follows from
\begin{equation*}
    \begin{split}
        &P_+(\k_j)\big(\Lcal_M(\e_1 \eps^2 \partial_\Xi + i\k_j) - \Lcal_{M^*}(\e_1 \eps^2 \partial_\Xi + i\k_j)\big)\phib_+(\k_j;M^*) \\
        &= P_+(\k_j)\big(\Lcal_M(\e_1 \eps^2 \partial_\Xi + i\k_j) - \Lcal_{M}(i\k_j)\big)\phib_+(\k_j;M^*) \\
        &\quad + P_+(\k_j)\big(\Lcal_M(i\k_j) - \Lcal_{M^*}(i\k_j)\big)\phib_+(\k_j;M^*) \\
        &\quad + P_+(\k_j)\big(\Lcal_{M^*}(i\k_j) - \Lcal_{M}(\e_1 \eps^2 \partial_\Xi + i\k_j)\big)\phib_+(\k_j;M^*).
    \end{split}
\end{equation*}
The third term is of order $\eps^4$ using \eqref{eq:dispair-Taylor} and the second term is equal to $\eps^2 \partial_M \lambda_+(\k_j;M^*) \phib_+(\k_j;M^*) + \Ocal(\eps^4)$ using \eqref{eq:expansion-linear-coeff}. Finally, the first term is again of order $\eps^4$ using that
\begin{equation*}
\begin{split}
    & P_+(\k_j)\big(\Lcal_M(\e_1 \eps^2 \partial_\Xi + i\k_j) - \Lcal_{M}(i\k_j)\big)\phib_+(\k_j;M^*) \\
    & = P_+(\k_j)\big(\Lcal_{M^*}(\e_1 \eps^2 \partial_\Xi + i\k_j) - \Lcal_{M^*}(i\k_j)\big)\phib_+(\k_j;M^*) \\
    & \quad + \eps^2 \big(\partial_M P_+(\k_j) \Lcal_{M^*}(\e_1 \eps^2 \partial_\Xi + i\k_j) - \partial_M\lambda_+(\k_j;M^*)\big) \phib_+(\k_j;M^*) + \Ocal(\eps^4).
\end{split}
\end{equation*}
Here, the first term is of order $\eps^4$ due to \eqref{eq:dispair-Taylor} and the second one is of order $\eps^4$ using that 
\begin{equation*}
    \partial_M P_+(\k_j) \Lcal_{M^*}(\e_1 \eps^2 \partial_\Xi + i\k_j)\phib_+(\k_j;M^*) = \partial_M\lambda_+(\k_j;M^*) \phib_+(\k_j) + \Ocal(\eps^2)
\end{equation*} 
by Taylor expansion in $\eps$. This concludes the discussion of the linear terms.

Collecting the results of Lemma \ref{lem:reduction-expansion}, the discussion of the nonlinear terms \eqref{eq:reduction-nonlinear-terms}, and the expansions of linear terms \eqref{eq:dispair-Taylor} and \eqref{eq:full-dispair-Taylor} we obtain the reduced equations on the center manifold by following the same calculations as in the formal derivation of the amplitude equations in Section \ref{sec:amplitude-equation}. In particular, on the square lattice, we find that
\begin{equation*}
    \begin{split}
        0 &= c\partial_\Xi A_1 + M_0 \kappa A_1 + K_c A_0 A_1 + K_0 A_1 \abs{A_1}^2 + K_1 A_1 \abs{A_2}^2, \\
        0 & = c \partial_\Xi A_2 + M_0 \kappa A_2 + K_c A_0 A_2 + K_0 A_2 \abs{A_2}^2 +  K_1 A_2 \abs{A_1}^2, \\
        0 & = c \partial_{\Xi} A_0
    \end{split}
\end{equation*}
to leading order in $\eps$. On the hexagonal lattice, the reduced equations are given by
\begin{equation*}
    \begin{split}
        0 &= c\partial_\Xi A_1 + M_0 \kappa A_1 + K_c A_0 A_1 + \dfrac{N}{\eps}\bar{A}_2\bar{A}_3 + K_0 A_1 \abs{A_1}^2 + K_2 A_1 (\abs{A_2}^2 + \abs{A_3}^2), \\
        0 &= c\partial_\Xi A_2 + M_0 \kappa A_2 + K_c A_0 A_2 + \dfrac{N}{\eps}\bar{A}_1\bar{A}_3  + K_0 A_2 \abs{A_2}^2 +  K_2 A_2 (\abs{A_1}^2 + \abs{A_3}^2), \\
        0 &= c\partial_\Xi A_3 + M_0 \kappa A_3 + K_c A_0 A_3 + \dfrac{N}{\eps}\bar{A}_1\bar{A}_2  + K_0 A_3 \abs{A_3}^2 + K_2 A_3 (\abs{A_1}^2 + \abs{A_2}^2), \\
        0 &= c\partial_{\Xi} A_0
    \end{split}
\end{equation*}
again to leading order in $\eps$. The coefficients are given in Sections \ref{sec:stationary-reduced-square} and \ref{sec:stationary-hex} as well as in the Supplementary Material. We again notice that $A_0$ is constant to leading order and, therefore, can be set to zero. In particular, we point out that the amplitude modulation is indeed given to leading order as fast front solutions of the formal amplitude equations, see Section \ref{sec:dynamics-amplitude-equations}.

\subsection{Dynamics on the center manifold: square lattice}\label{sec:modulation-center-square}

We now analyse the dynamics of the reduced equations on the center manifold. The stationary patterns obtained in Section \ref{sec:stationary} appear as stationary points of the dynamics and we are interested in heteroclinic orbits connecting the stationary points. First, we note that the subspace $A_1(\Xi), A_2(\Xi) \in \R$ of real-valued amplitudes is invariant. This can be seen by applying the symmetry $c\to -c$, $\Xi \to -\Xi$ and $p \to -p$, which is inherited from the reflection symmetry of \eqref{eq:thin-film-equation} and preserved by the center manifold reduction. Hence, the leading-order system that we need to analyse reads as
\begin{equation}\label{eq:leading-order-system-fronts-square}
    \begin{split}
        0 &= c\partial_\Xi A_1 + M_0 \kappa A_1 + K_0 A_1^3 + K_1 A_1 A_2^2, \\
        0 & = c \partial_\Xi A_2 + M_0 \kappa A_2 + K_0 A_2^3 +  K_1 A_2 A_1^2.
    \end{split}
\end{equation}
As in \cite{doelman2003}, we find that this system has a Lyapunov functional
\begin{equation*}
    \Ecal(A_1,A_2) = \dfrac{M_0\kappa}{2c}[A_1^2 + A_2^2] + \dfrac{K_0}{4c} [A_1^4 + A_2^4] + \dfrac{K_1}{2c} A_1^2 A_2^2,
\end{equation*}
which is strictly decreasing along the orbits of the reduced system. Therefore, there are no periodic orbits in \eqref{eq:leading-order-system-fronts-square}. Moreover, we find that the leading-order system is a gradient flow since it can be written as
\begin{equation*}
    \partial_\Xi \begin{pmatrix}
        A_1 \\ A_2
    \end{pmatrix} = - \nabla\Ecal(A_1,A_2).
\end{equation*}
Therefore, equilibria of the reduced dynamics are given as critical points of $\Ecal$. Since we already analysed the equilibria in Section \ref{sec:stationary-square}, we only briefly recall the results here. If $M_0 K_0 < 0$ there is a critical point of $\Ecal$ at 
\begin{equation*}
    R = \Big(\sqrt{-\dfrac{M_0\kappa}{K_0}}, 0\Big),
\end{equation*}
corresponding to roll waves in the full system \eqref{eq:thin-film-equation} after reverting the center manifold reduction. Additionally, if $M_0(K_0+K_1) < 0$ there is a critical point of $\Ecal$ at
\begin{equation*}
    S = \Big(\sqrt{-\dfrac{M_0 \kappa}{K_0 + K_1}},\sqrt{-\dfrac{M_0 \kappa}{K_0 + K_1}}\Big)
\end{equation*}
corresponding to square patterns, see Theorem \ref{thm:square-patterns}. Finally, the trivial state $T = (0,0)$ is always a critical point.

To determine the linear stability, we derive the Hessian of $\Ecal$ at the respective critical points and calculate the corresponding eigenvalues. At the trivial state, $D^2\Ecal(T)$ possesses a double eigenvalue
\begin{equation*}
    \lambda_{1,T} = \lambda_{2,T} = -\frac{M_0\kappa}{c}.
\end{equation*}
For roll waves we find that $D^2\Ecal(R)$ has the eigenvalues
\begin{equation*}
    \lambda_{1,R} =  \dfrac{2M_0\kappa}{c}, \quad \lambda_{2,R} = \dfrac{(K_1-K_0)M_0\kappa}{cK_0}
\end{equation*}
with corresponding eigenvectors given by $(1,0)^T$ and $(0,1)^T$, respectively. Finally, for square patterns we find that $D^2\Ecal(S)$ has the eigenvalues
\begin{equation*}
    \lambda_{1,S} = \dfrac{2M_0\kappa}{c}, \quad \lambda_{2,S} = -\dfrac{2(K_1-K_0)M_0\kappa}{c(K_0+K_1)}
\end{equation*}
with corresponding eigenvectors given by $(1,1)^T$ and $(-1,1)^T$, respectively. Recall that $c > 0$ by definition and $\kappa = \partial_M\lambda_+(\k_j;M^*) > 0$ by Proposition \ref{prop:linear-monotonic}. Additionally, given the sign of $M_0$, the existence conditions for roll waves and square patterns determine the signs of $K_0$ and $K_0 + K_1$, respectively. Hence, it remains to distinguish four cases depending on the signs of $M_0$ and $K_1 - K_0$, and we summarise the results in Tables \ref{tab:trivial}--\ref{tab:square}. 

\begin{table}[H]
\caption{Linear stability of trivial state $T$.}
\begin{tabular}{p{0.25\textwidth}|p{0.25\textwidth}}
     \makecell{$M_0>0$} & \makecell{$M_0<0$} \\ \hline
     \makecell{$\lambda_{1,T} = \lambda_{2,T} < 0$\\[3pt] $T$ is stable} & \makecell{$\lambda_{1,T} = \lambda_{2,T} > 0$ \\[3pt] $T$ is unstable}
\end{tabular}
\label{tab:trivial}
\end{table}

\begin{table}[H]
\caption{Linear stability of roll waves $R$.}
{%
\begin{tabular}{p{0.12\textwidth}|p{0.25\textwidth}|p{0.25\textwidth}}
    & \makecell{$M_0>0$ ($K_0<0$)} & \makecell{$M_0<0$ ($K_0>0$)} \\ \hline
    $K_1-K_0>0$ & \makecell{$\lambda_{1,R}>0,  \lambda_{2,R} < 0$\\[3pt] $R$ is a saddle point} & \makecell{$\lambda_{1,R}<0, \lambda_{2,R} < 0$ \\[3pt] $R$ is stable} \\
    \hline
    $K_1-K_0<0$ &\makecell{$\lambda_{1,R} >0 ,\lambda_{2,R} > 0$\\[3pt] $R$ is unstable} & \makecell{$\lambda_{1,R} <0,\lambda_{2,R} > 0$ \\[3pt] $R$ is a saddle point}
\end{tabular}
}
\label{tab:roll}
\end{table}

\begin{table}[H]
\caption{Linear stability of square patterns $S$.}
\begin{tabular}{p{0.12\textwidth}|p{0.25\textwidth}|p{0.25\textwidth}}
    & \makecell{$M_0>0$ ($K_0+K_1<0$)} & \makecell{$M_0<0$ ($K_0+K_1>0$)} \\ \hline
    $K_1-K_0>0$ & \makecell{$\lambda_{1,S}>0,  \lambda_{2,S} > 0$\\[3pt] $S$ is unstable} & \makecell{$\lambda_{1,S}<0, \lambda_{2,S} > 0$ \\[3pt] $S$ is a saddle point} \\\hline
    $K_1-K_0<0$ &\makecell{$\lambda_{1,S} > 0 ,\lambda_{2,S} < 0$\\[3pt] $S$ is a saddle point} & \makecell{$\lambda_{1,S} <0,\lambda_{2,S} < 0$ \\[3pt] $S$ is stable}
\end{tabular}
\label{tab:square}
\end{table}

We now discuss the existence of heteroclinic connections between the different fixed points in the leading order system. First, we note that the connections between $T$ and $R$ and between the $T$ and $S$ can be constructed straightforwardly by noting that these connections lie in the one-dimensional invariant subspaces $\{A_1 \in \R, A_2 = 0\}$ and $\{A_1 = A_2\in \R\}$, respectively. Additionally, note that recalling the eigenvectors, the corresponding eigenvalue in the relevant direction is $\lambda_{1,R}$ for $R$ and $\lambda_{1,S}$ for $S$. Therefore, from the signs of the relevant eigenvalues in Tables \ref{tab:trivial}--\ref{tab:square} we find that for $M_0 > 0$ there are heteroclinic connections from the nontrivial fixed points to the trivial state, while for $M_0 < 0$ there are heteroclinic connections from the trivial state to the nontrivial fixed points.

Next, we discuss the existence of heteroclinic orbits connecting $S$ and $R$. Hence, we assume that we are in a parameter regime, where both fixed points exist, that is, $M_0 K_0 < 0$ and $M_0(K_0 + K_1) < 0$. We then construct these orbits via phase-plane analysis of the leading-order system \eqref{eq:leading-order-system-fronts-square}. For this, we first note that in the case $M_0 > 0$ there is a sufficiently large ball around the origin of the $A_1$-$A_2$-plane, which is an overflowing invariant set with respect to the dynamics of \eqref{eq:leading-order-system-fronts-square}. This follows from the fact that the Lyapunov function $\Ecal$ tends to $-\infty$ as $(|A_1|,|A_2|) \rightarrow +\infty$, which follows from an application of Young's inequality. Similarly, for $M_0 < 0$ there exists a sufficiently large ball around the origin in the $A_1$-$A_2$-plane, which is an inflowing invariant set.

We now present the construction of the heteroclinic orbits in detail in the case $M_0 < 0$. For $M_0 > 0$ the existence follows from the same arguments after a time-reversal. If $M_0 < 0$ and $K_1 - K_0 > 0$, we first find that $T$ is an unstable fixed point, $R$ is a stable fixed point, and $S$ is a saddle with a one-dimensional unstable manifold, which is perpendicular to the invariant subspace $\{A_1 = A_2 \in \R\}$ at $S$. Since both $R$ and $S$ are contained in an inflowing invariant set and $T$ is unstable, there must be a transverse intersection of the one-dimensional unstable manifold of $S$ with the two-dimensional stable manifold of $R$. Therefore, there exists a heteroclinic orbit from $S$ to $R$. In the case $M_0 < 0$ and $K_1 - K_0 < 0$, the roles of $S$ and $R$ are interchanged, and we can construct a heteroclinic orbit from $S$ and $R$ with the same arguments.

We collect the results on the existence of fixed points and heteroclinic orbits between them in the following phase-plane diagrams in Figures \ref{fig:phase-diags-square1} and \ref{fig:phase-diags-square2}.

\begin{figure}[H]
    \centering
    \includegraphics[width=0.4\linewidth]{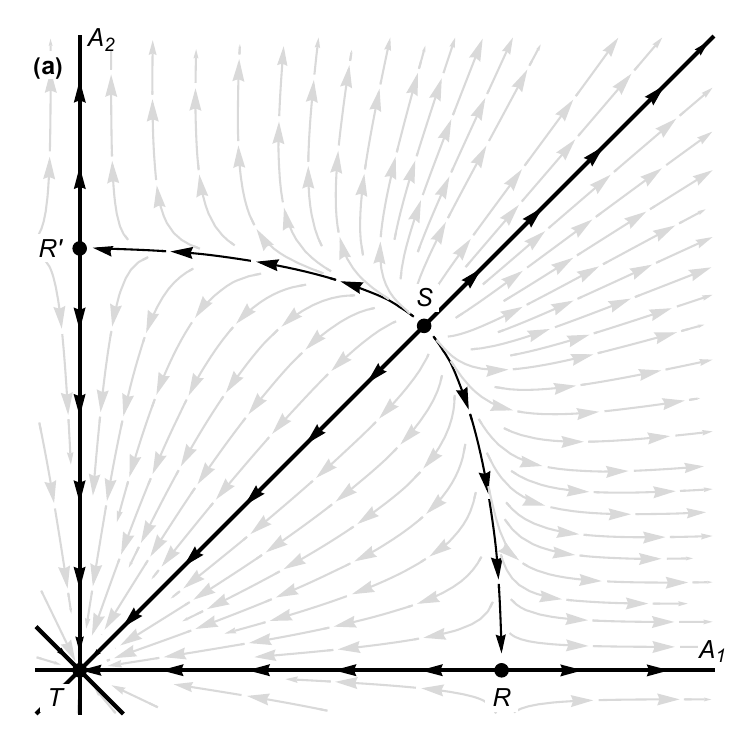}
    \hspace{0.1cm}
    \includegraphics[width=0.4\linewidth]{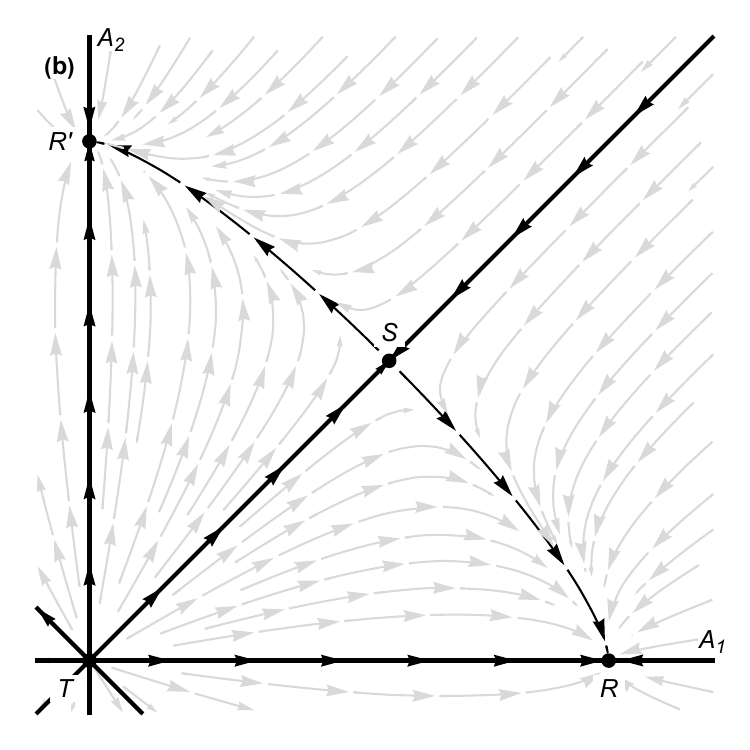}
    
    \caption{Phase planes of the dynamical system \eqref{eq:leading-order-system-fronts-square} in the following parameter regimes: a) $M_0>0$ and $K_1-K_0 >0$, b) $M_0<0$ and $K_1-K_0 >0$.}
    \label{fig:phase-diags-square1}
\end{figure}

\begin{figure}[H]
    \centering
    \includegraphics[width=0.4\linewidth]{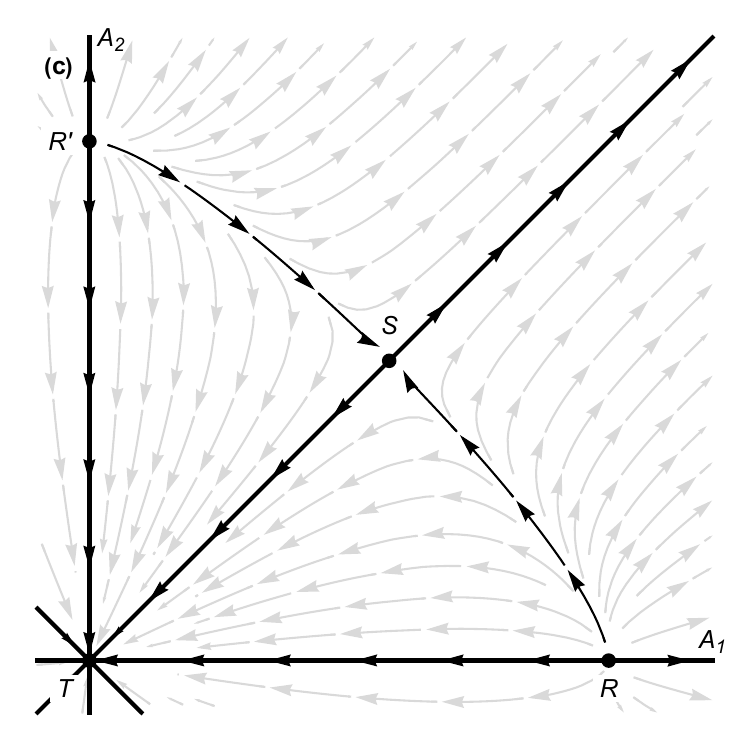}
    \hspace{0.1cm}
    \includegraphics[width=0.4\linewidth]{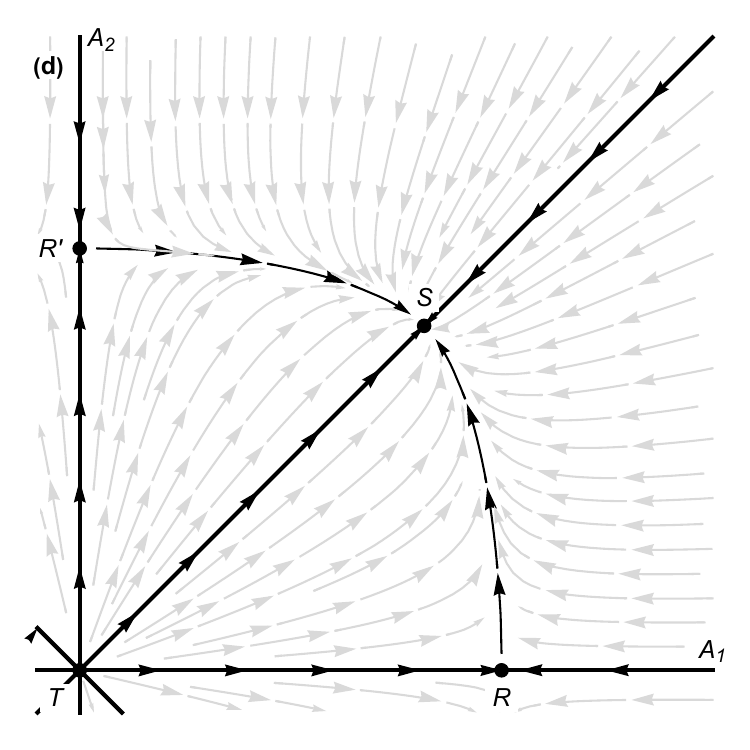}
    \caption{Phase planes of the dynamical system \eqref{eq:leading-order-system-fronts-square} in the following parameter regimes: c) $M_0>0$ and $K_1-K_0 <0$, d) $M_0<0$ and $K_1-K_0 <0$.}
    \label{fig:phase-diags-square2}
\end{figure}

\subsection{Dynamics on the center manifold: hexagonal lattice}\label{sec:modulation-center-hex}

We now study the dynamics of the reduced equation on the hexagonal lattice. Following \cite{doelman2003} we restrict to the invariant set $\{A_1 \in \R, A_2 = A_3 \in \R\}$, which contains both hexagonal patterns with $A_1 = A_2 = A_3$ as well as roll waves with $A_1 \in \R$ and $A_2=A_3 = 0$. Additionally, recalling the discussion in Section \ref{sec:stationary-hex}, we assume that $N = \eps N_0$ by restricting to an appropriate parameter regime in $\beta$ and $g$. Furthermore, without loss of generality, we assume that $N_0>0$. Indeed, the case $N_0<0$ can be recovered by the symmetry transformation $N\to -N$ and $A_j \to -A_j$ for $j = 1,2$. Then the resulting reduced system reads as

\begin{equation}\label{eq:leading-order-system-fronts-hex}
\begin{split}
        0 &= c\partial_\Xi A_1 + M_0 \kappa A_1 + N_0 A_2^2 + K_0 A_1^3 + 2K_2 A_1 A_2^2 , \\
        0 &= c\partial_\Xi A_2 + M_0 \kappa A_2 + N_0 A_1A_2  + (K_0+K_2) A_2^3 +  K_2 A_2A_1^2, 
    \end{split}
\end{equation}
which we abbreviate with $\partial_\Xi (A_1,A_2) = f(A_1,A_2)$.

We observe that there is a neighborhood of the curve $\{(\beta,g) \in (0,72) \times (0,\infty) \,:\, N=0\}$, which is parametrised by
\begin{equation}\label{eq:beta-von-g}
    \beta(g) =\dfrac{2 g^2-\sqrt{484 g^3+7425 g^2+34992 g+46656}+87 g+216}{3 (g+2)},
\end{equation}
for $g \in (0,18)$. It turns out that on this curve, $K_0+2K_2 < 0$ holds, and we restrict the following analysis to this neighborhood. Note that furthermore, $K_0$ changes sign along the curve \(\beta(g)\) at $g=10$, see Figures \ref{fig:coeffs-hex} and \cite[Fig.~5]{shklyaev2012}.

As for the square lattice, this system also has a Lyapunov functional given by
\begin{equation*}
    \Ecal(A_1,A_2) = \dfrac{M_0\kappa}{2c}[A_1^2 + 2A_2^2] + \dfrac{N_0}{c} A_1A_2^2 + \dfrac{K_0}{4c} A_1^4 + \dfrac{K_2}{c} A_1^2A_2^2  + \dfrac{K_0+K_2}{2c}A_2^4,
\end{equation*}
cf.~\cite{doelman2003}. Note that in this case, the leading-order system is not a gradient flow with respect to $\Ecal$, but it holds $\nabla \Ecal \cdot f = -f_1^2-2f_2^2$. In particular, the Lyapunov functional is strictly decreasing along orbits of the reduced system, which again excludes the existence of periodic orbits.

\paragraph{Fixed points and dynamics on the invariant sets}
We first discuss the existence of fixed points and heteroclinic connections on the invariant sets $\{A_1 \in \R, A_2 = 0\}$, $\{A_1 = A_2\}$ and $\{A_1 = - A_2\}$. As discussed in Section \ref{sec:stationary-hex} we find the following fixed points on these sets. First, $T = (0,0)$ is again a fixed point, as well as $R = \bigl(\sqrt{\tfrac{-M_0\kappa}{K_0}},0\bigr)$, provided $M_0 K_0 < 0$. Additionally, if $N_0^2 - 4M_0\kappa(K_0 + 2K_2) > 0$, there are two fixed points on the invariant subspace $\{A_1=A_2\}$
\begin{equation*}
    H_{1,\pm} = \Big(\dfrac{-N_0 \mp \sqrt{N_0^2 - 4\kappa M_0 (K_0 + 2K_2)}}{2(K_0+2K_2)}, \dfrac{-N_0 \mp \sqrt{N_0^2 - 4\kappa M_0 (K_0 + 2K_2)}}{2(K_0+2K_2)}\Big)
\end{equation*}
and there are two additional fixed points on the invariant subspace $\{A_1 = -A_2\}$
\begin{equation*}
    H_{2,\pm} = \Big(\dfrac{-N_0 \mp \sqrt{N_0^2 - 4\kappa M_0 (K_0 + 2K_2)}}{2(K_0+2K_2)}, -\dfrac{-N_0 \mp \sqrt{N_0^2 - 4\kappa M_0 (K_0 + 2K_2)}}{2(K_0+2K_2)}\Big)
\end{equation*}
which correspond to hexagonal patterns. If $A_1 A_2^2 > 0$ these fixed points correspond to up-hexagons in the full system \eqref{eq:thin-film-equation}, whereas they correspond to down-hexagons (sometimes called $\pi$-hexagons) if $A_1 A_2^2 < 0$, see \cite{hoyle2007}. In particular, since $N_0 > 0$ by assumption and $K_0 + 2K_2 < 0$, we find that if $M_0 < 0$, all fixed points correspond to up-hexagons. If $M_0 > 0$, it holds that $H_{1,+}$ and $H_{2,+}$ correspond to up-hexagons whereas $H_{1,-}$ and $H_{2,-}$ correspond to down-hexagons in the full system \eqref{eq:thin-film-equation}, see Figure \ref{fig:bifurcation-diag-hex} for the structure of the bifurcation.

\begin{figure}[H]
    \centering
    \includegraphics[width=0.9\textwidth]{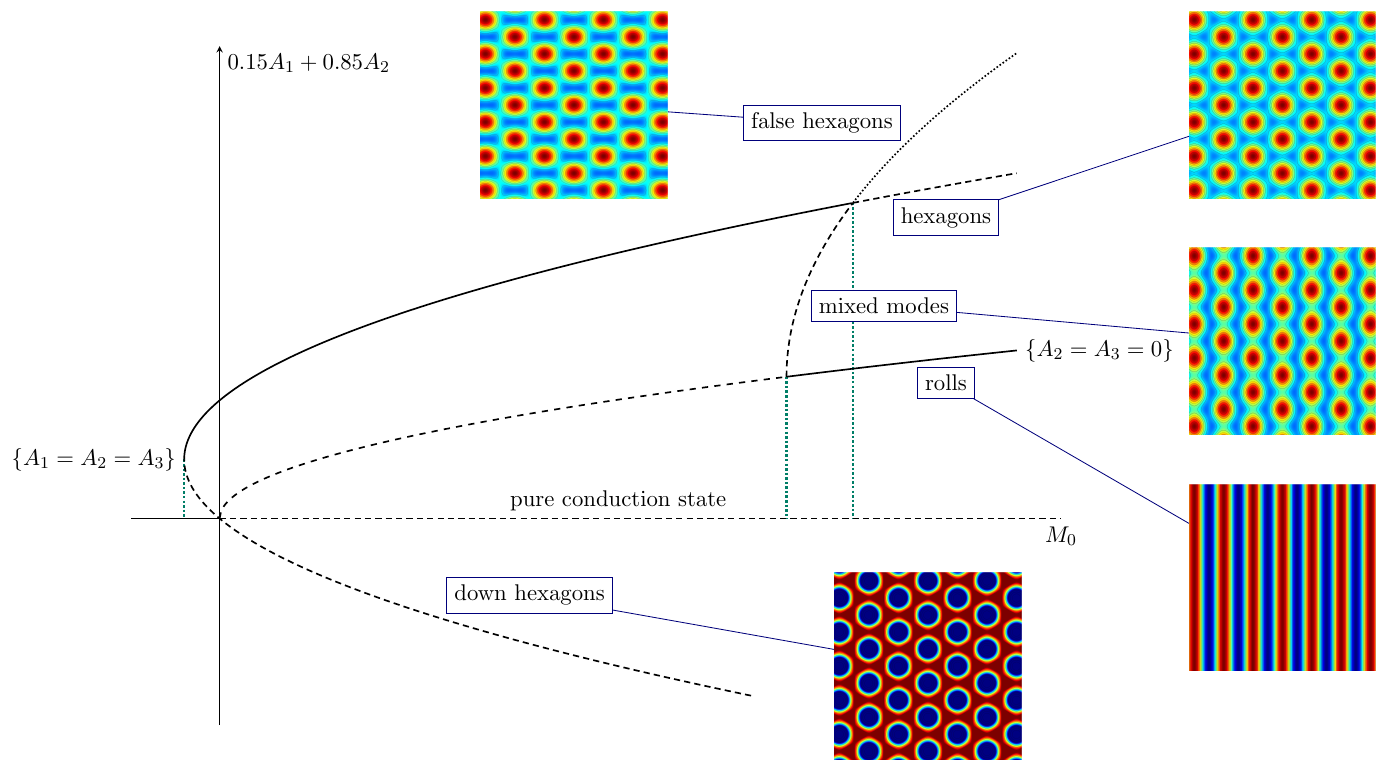}
    \caption{Bifurcation diagram on the hexagonal lattice for $N_0 > 0$, $K_0 <  0$ and $K_2 < 0$.}
    \label{fig:bifurcation-diag-hex}
\end{figure}

To determine the linear stability of the fixed points, we study $\nabla f$ at the fixed points and calculate the corresponding eigenvalues. At the trivial state, $\nabla f(T)$ possesses a double eigenvalue
\begin{equation*}
    \lambda_{1,T} = \lambda_{2,T} = -\frac{M_0\kappa}{c}.
\end{equation*}
For roll waves, we find that $\nabla f(R)$ has the eigenvalues
\begin{equation*}
    \lambda_{1,R} =  \dfrac{2M_0\kappa}{c}, \quad \lambda_{2,R} = -\dfrac{(K_0-K_2)M_0\kappa}{cK_0} - \dfrac{N_0\sqrt{\tfrac{-M_0\kappa}{K_0}}}{c}
\end{equation*}
with corresponding eigenvalues given by $(1,0)^T$ and $(0,1)^T$, respectively. Finally, for hexagonal patterns, we find that $\nabla f(H_{1,\pm})$ has the eigenvalues
\begin{equation*}
    \lambda_{1,H_{1,\pm}} = M_0\kappa - (K_0+2K_2)A_{1,\pm}^2, \quad \lambda_{2,H_{1,\pm}} = - 2M_0\kappa - 2(2K_0+K_2)A_{1,\pm}^2
\end{equation*}
with eigenvectors $(1,1)^T$ and $(-2,1)^T$ and where $A_{1,\pm}$ is the first component of $H_{1,\pm}$, respectively. Additionally, since the dynamics of \eqref{eq:leading-order-system-fronts-hex} is invariant under $A_2 \mapsto - A_2$, we find that $\nabla f(H_{2,\pm})$ has the same eigenvalues as $\nabla f(H_{1,\pm})$ with eigenvectors $(1,-1)^T$ and $(-2,-1)^T$.

Since the invariant sets are one-dimensional subsets of the phase space, it is sufficient to consider the signs of the eigenvalues with eigenspaces in the direction of these invariant subsets to analyse the dynamics. First, we find that $T$ is an unstable fixed point for $M_0 < 0$ and a stable one for $M_0 > 0$. Additionally, $R$ is stable along $\{A_1 \in \R, A_2 = 0\}$ for $M_0 < 0$ and unstable for $M_0 > 0$. Therefore, provided that $R$ exists, we obtain a heteroclinic connection from $T$ to $R$ for $M_0 < 0$ and from $R$ to $T$ for $M_0 > 0$. 

For $H_{1,\pm}$, the relevant eigenvalue is $\lambda_{1,H_{1,\pm}} = M_0 \kappa - (K_0 + 2K_2) A_1^2$. For $M_0 > 0$, we find that $\lambda_{1,H_{1,\pm}} > 0$ since $K_0 + 2K_2 < 0$ in a neighborhood of the curve $\{N=0\}$. Thus, noting again that $T$ is a stable fixed point for $M_0 > 0$, we obtain heteroclinic connections from $H_{1,\pm}$ to $T$. By again exploiting the invariance of \eqref{eq:leading-order-system-fronts-hex} under $A_2 \mapsto -A_2$ we thus also obtain heteroclinic connections from $H_{2,\pm}$ to $T$ for $M_0 > 0$.

For $M_0<0$, both $H_{1,+}$ and $H_{1,-}$ correspond to up-hexagons and in particular $0 < A_{1,-} < A_{1,+}$, which yields $\lambda_{1,H_{1,-}} < \lambda_{1,H_{1,+}}$. To find the signs we first note that $M_0\kappa - (K_0 + 2K_2) x^2$ is negative if $x < \sqrt{\tfrac{M_0 \kappa}{K_0 + 2K_0}} =: x_\star$ and positive if $x > x_\star$. Using that $A_{1,\pm} > 0$ we obtain $A_{1,-} < x_\star < A_{1,+}$ and thus we find that $\lambda_{1,H_{1,-}}<0$ and $\lambda_{1,H_{1,+}}>0$. In particular, this yields the heteroclinic connections from $T$ to $H_{1,-}$ and from $H_{1,+}$ to $H_{1,-}$ as well as connections from $T$ to $H_{2,-}$ and from $H_{2,+}$ to $H_{2,-}$, provided the hexagons exist. We collect the results on the existence of fixed points and heteroclinic orbits in Figures \ref{fig:phase-diags-hex1} and \ref{fig:phase-diags-hex2}.

\paragraph{Fixed points and dynamics outside the invariant sets}
For the dynamics outside of the invariant sets, we restrict to parameter regimes where $K_0 < 0$, which guarantees that roll waves bifurcate supercritically at $M_0 = 0$ and thus exist for $M_0 > 0$. Numerically, we observe that on $\{N=0\}$ it holds $K_0 < 0$ for $g \in (10,18)$ and thus, such a parameter regime indeed exists. Since it holds that $K_0 + 2K_2 < 0$ we obtain that also $H_{1,\pm}$ and $H_{2,\pm}$ exist for $M_0 > 0$.

We start again by analysing the fixed points of the system \eqref{eq:leading-order-system-fronts-hex}. Outside the invariant subsets discussed above, there are at most two additional fixed points at
\begin{equation*}
   M\!\!M_{\pm} = \Big(\dfrac{N_0}{K_0 - K_2},\pm\dfrac{1}{K_0 - K_2} \sqrt{-\dfrac{K_0 N_0^2 + (K_0 - K_2)^2 M_0 \kappa}{K_0 + K_2}}\Big),
\end{equation*}
which exist if $\tfrac{K_0 N_0^2 + (K_0 - K_2)^2 M_0 \kappa}{K_0 + K_2}<0$. Note that if they exist, then necessarily $K_0\neq K_2$. In fact, we find numerically that $K_0 - K_2 > 0$ for $10 < g < 18$. Since the reduced system \eqref{eq:leading-order-system-fronts-hex} is invariant under $A_2 \mapsto - A_2$ we restrict the following analysis to the fixed point $M\!\!M_+ =: M\!\!M$. In particular, this fixed point bifurcates from $R$ at $M_0 \kappa = -\tfrac{K_0 N_0^2}{(K_0 - K_2)^2}$ and crosses the invariant subset $\{A_1 = A_2\}$ at $M_0 \kappa = -\tfrac{N_0^2(2K_0 + K_2)}{(K_0 - K_2)^2}$. If $M\!\!M$ is below the diagonal they are referred to as mixed modes in the literature. If $M\!\!M$ is above the diagonal, they are called false hexagons, see Figure \ref{fig:bifurcation-diag-hex}

The eigenvalues of the mixed modes $\lambda_{1,M\!\!M}$ and $\lambda_{2,M\!\!M}$ are given in the Supplementary Material. Numerically, we find that $\lambda_{1,M\!\!M} > 0$ for all $(\beta,g)$ in the neighborhood of $\{N=0\}$ with $10 < g < 18$. Further, we find that $\lambda_{2,M\!\!M} < 0$ on the initial part of the branch where $M\!\!M$ lies below the diagonal, while $\lambda_{2,M\!\!M} > 0$ when $M\!\!M$ lies above the diagonal. Therefore, if $M\!\!M$ is below the diagonal, it has one unstable and one stable eigenspace, whereas if $M\!\!M$ is above the diagonal, it has two unstable eigenspaces.

To understand the dynamics outside of the invariant sets it is also necessary to discuss the signs of the eigenvalues $\lambda_{2,R}$ and $\lambda_{2,H_{1,\pm}}$, which correspond to the eigenspaces of $R$ and $H_{1,\pm}$, respectively, which lie outside of the invariant sets. To understand $\lambda_{2,R}$, note that $\lambda_{2,R}=0$ precisely at $M_0\kappa = 0$ and $M_0\kappa = -\tfrac{K_0 N_0^2}{(K_0 - K_2)^2}$, where the latter condition is the bifurcation point of the mixed modes. Since the second term dominates for small $M_0$, we necessarily find that $\lambda_{2,R} < 0$ for $0 < M_0\kappa < -\tfrac{K_0 N_0^2}{(K_0 - K_2)^2}$ and $\lambda_{2,R} >0$ for $M_0\kappa > -\tfrac{K_0 N_0^2}{(K_0 - K_2)^2}$. For $\lambda_{2,H_{1,\pm}}$, we note that at $M_0=0$ it is $\lambda_{2,H_{1,+}} > 0$ and $\lambda_{2,H_{1,-}}=0$. Furthermore, we know that at $M_0\kappa = -\tfrac{N_0^2(2K_0 + K_2)}{(K_0 - K_2)^2}$ we have $\lambda_{2,H_{1,+}} = 0$. Inserting $A_{1,+}$, we find that
\begin{equation*}
    \lambda_{2,H_{1,+}} = \dfrac{2(K_0-K_2)(K_0+2K_2)M_0\kappa - 2(K_0+K_2)N_0\Bigl(N_0+\sqrt{N_0^2-4M_0\kappa(K_0+2K_2)}\Bigr)}{(K_0+2K_2)^2}
\end{equation*}
and hence $\lambda_{2,H_{1,+}}$ descreases for $M_0$ sufficiently large, showing that $\lambda_{2,H_{1,+}} > 0$ for $0<M_0\kappa < -\tfrac{N_0^2(2K_0 + K_2)}{(K_0 - K_2)^2}$ and $\lambda_{2,H_{1,+}} < 0 $ for $M_0\kappa > -\tfrac{N_0^2(2K_0 + K_2)}{(K_0 - K_2)^2}$. In particular, $\lambda_{2,H_{1,+}}$ changes sign from positive to negative precisely when $M\!\!M$ crosses the diagonal, i.e. $H_{1,+}$ changes from unstable to a saddle point. Since
\begin{equation*}
    \begin{split}
        \lambda_{2,H_{1,-}} & = \dfrac{-(2K_0+K_2)N_0^2+ 2(K_0-K_2)(K_0+2K_2)M_0\kappa}{(K_0+2K_2)^2} \\
        & \quad + \dfrac{2(K_0+K_2)N_0\Bigl(N_0+\sqrt{N_0^2-4M_0\kappa(K_0+2K_2)}\Bigr)}{(K_0+2K_2)^2} <0,
    \end{split}
\end{equation*}
we find that $H_{1,-}$ is always a saddle point. We point out that due to invariance with respect to $A_2 \mapsto - A_2$, it holds that $\lambda_{1,H_{2,\pm}} = \lambda_{1,H_{1,\pm}}$ and $\lambda_{2,H_{2,\pm}} = \lambda_{2,H_{1,\pm}}$.

In order to understand the existence of further heteroclinic orbits outside of the invariant sets, note that, since $K_0<0$ and $K_2<0$, there exists a sufficiently large ball containing all fixed points which is an overflowing invariant set. By a time-reversal argument, we may conclude the existence of heteroclinic orbits connecting $H_{1,+}$ to $R$ and $H_{1,+}$ to $H_{2,-}$ in the regime $M_0\kappa < -\tfrac{K_0 N_0^2}{(K_0 - K_2)^2}$, before the mixed-modes bifurcate. In the regime $-\tfrac{K_0 N_0^2}{(K_0 - K_2)^2} < M_0\kappa < -\tfrac{N_0^2(2K_0 + K_2)}{(K_0 - K_2)^2}$, the branch connecting $H_{1,+}$ to $H_{2,-}$ remains.

Below the diagonal, the situation is more subtle since the additional saddle point $M\!\!M$ is located here. Using a time-reversal argument and the Poincaré--Bendixson theorem \cite[Thm. 6.32]{d.meiss2017}, there exist heteroclinic connections from $R$ to $M\!\!M$ and also from $H_{1,+}$ to $M\!\!M$. The details of this argument can be found in Section \ref{app:heteroclinic} of the appendix. Using this, the fact that $T$ is a stable fixed point and $M\!\!M$ is a saddle node, we also find a heteroclinic connection from $M\!\!M$ to $T$.

Finally, in the case that $M_0 \kappa > -\tfrac{N_0^2(2K_0 + K_2)}{(K_0 - K_2)^2}$, the fixed point $H_{1,+}$ is a saddle point and $M\!\!M$ lies above the diagonal and is now an unstable fixed point. Using that the unstable manifolds of $H_{1,+}$ and $H_{2,-}$ are the diagonals, we can conclude with similar arguments as above and in Appendix \ref{app:heteroclinic} that there are heteroclinic orbits from $M\!\!M$ to $T$, $H_{1,+}$ and $H_{2,-}$, respectively. Additionally, using time-reversal arguments, we find a heteroclinic orbit from $H_{1,+}$ to $R$.

We summarise the existence of heteroclinic orbits discussed above in the phase portraits in Figures \ref{fig:phase-diags-hex1} and \ref{fig:phase-diags-hex2}.

\begin{figure}[H]
    \centering
    \includegraphics[width=0.48\linewidth]{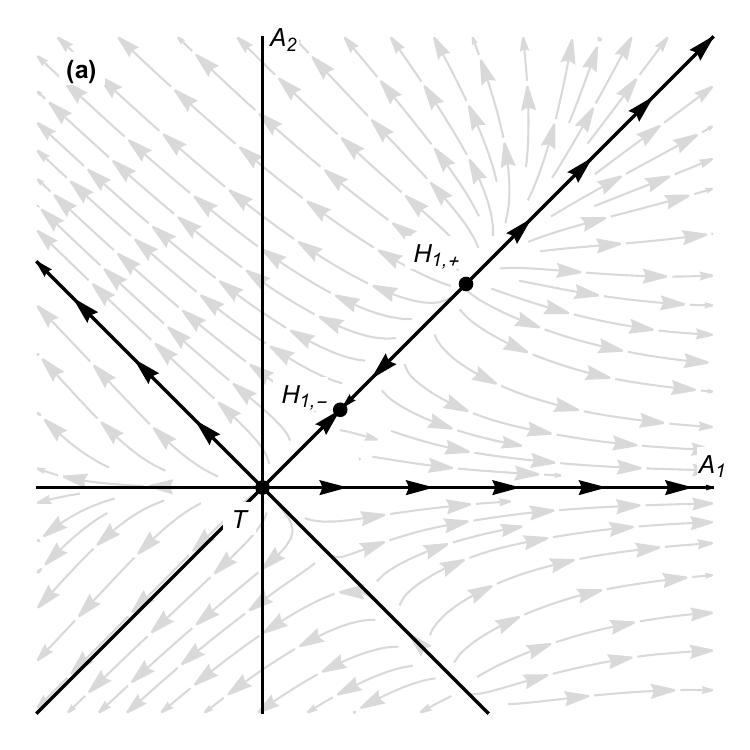}
    \hspace{0.1cm}
    \includegraphics[width=0.48\linewidth]{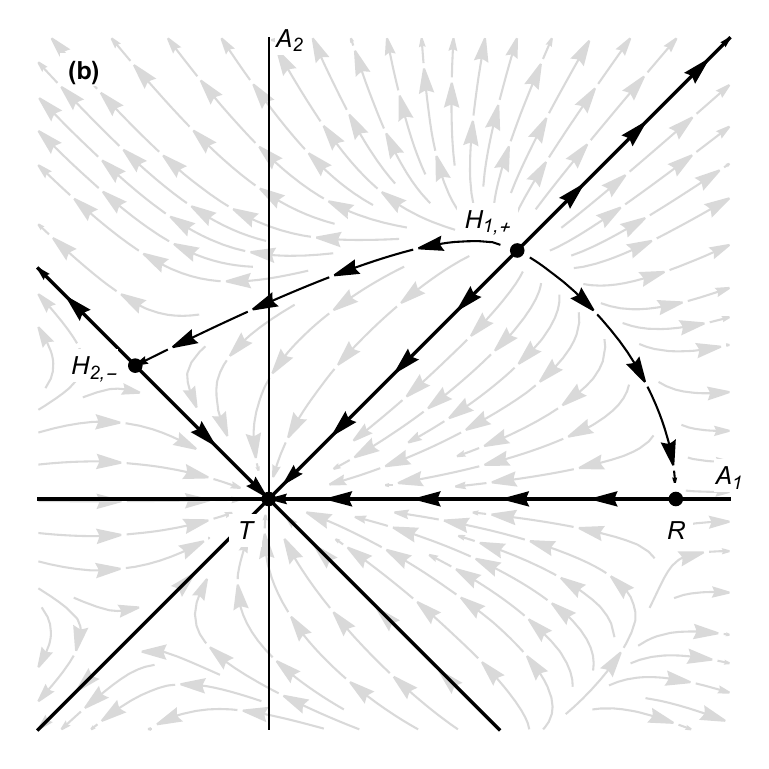}
    
    \caption{Phase planes of the dynamical system \eqref{eq:leading-order-system-fronts-hex} in the following parameter regimes: a) $M_0<0$, where only the hexagonal patterns and the trivial fixed point exist; b) $0 < M_0\kappa < -\tfrac{K_0 N_0^2}{(K_0 - K_2)^2}$, where the trivial fixed point, square patterns and hexagonal patterns exist.}
    \label{fig:phase-diags-hex1}
\end{figure}

\begin{figure}[H]
    \centering
    \includegraphics[width=0.48\linewidth]{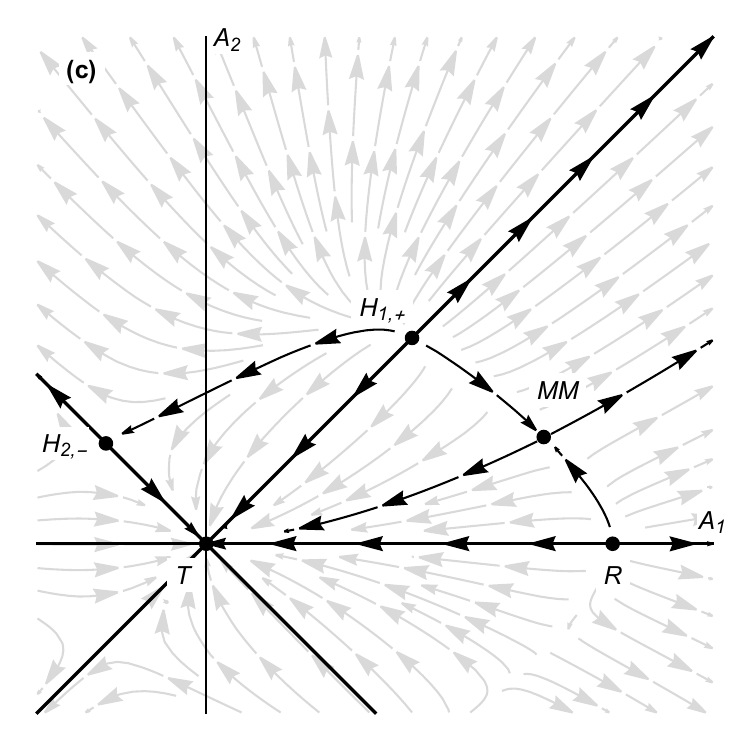}
    \hspace{0.1cm}
    \includegraphics[width=0.48\linewidth]{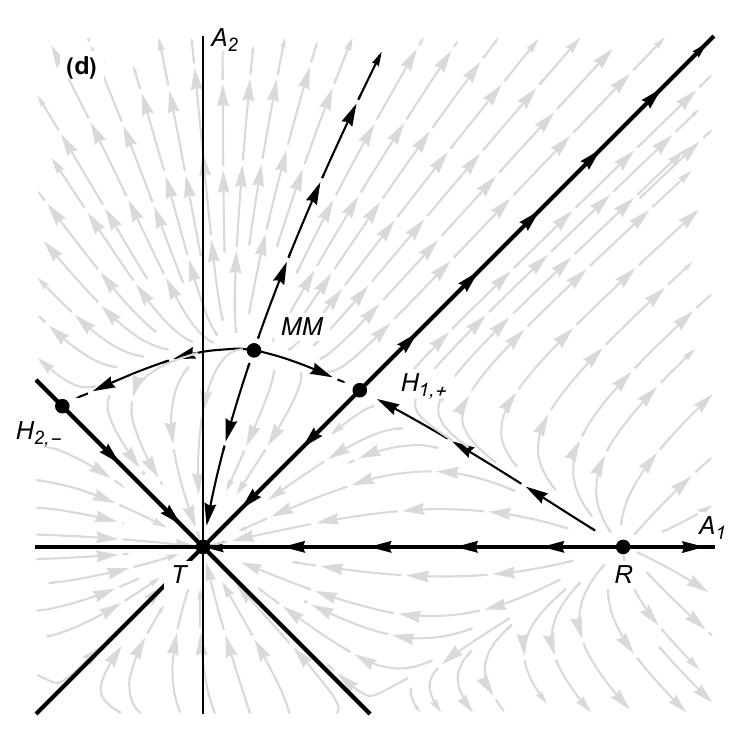}
    \caption{Phase planes of the dynamical system \eqref{eq:leading-order-system-fronts-hex} in the following parameter regimes after the mixed modes bifurcated from the roll waves:  c) $-\tfrac{K_0 N_0^2}{(K_0 - K_2)^2} < M_0\kappa < -\tfrac{N_0^2(2K_0 + K_2)}{(K_0 - K_2)^2}$, d) $M_0 \kappa > -\tfrac{N_0^2(2K_0 + K_2)}{(K_0 - K_2)^2}$. }
    \label{fig:phase-diags-hex2}
\end{figure}

\begin{remark}
    In the previous discussion, we have assumed that $K_0<0$ which in particular implied that a sufficiently large ball is an overflowing invariant set. For $K_0 > 0$, this gets more subtle, and more complicated situations might occur. Since $K_0 < 0$ is the case, where both hexagons and roll waves exist for $M_0 > 0$ we thus restrict the discussion here to this scenario.
\end{remark}

\subsection{Persistence of heteroclinic orbits and final results}

In order to conclude the construction of the fast-moving modulating front dynamics for system \eqref{eq:thin-film-equation}, it remains to prove that the heteroclinic orbits found for the reduced equations persist under the higher-order perturbations. First, we need to take the conservation law into account. Recalling the derivation of the reduced equations from the physical system, we find that
\begin{equation*}
    0 = c \partial_\Xi A_0 + \eps^2 \partial_\Xi f_{\text{sq}}(A_0,A_1,A_2)
\end{equation*}
for the square lattice and 
\begin{equation*}
    0 = c \partial_\Xi A_0 + \eps^2 \partial_\Xi f_{\text{hex}}(A_0,A_1,A_2,A_3)
\end{equation*}
for the hexagonal lattice, where $f_{\mathrm{sq}}$ and $f_{\mathrm{hex}}$ are sufficiently regular functions. Hence, after integrating once and applying the implicit function theorem, we find that $A_0$ is slaved by $A_1$ and $A_2$ in the square case and by $A_1$, $A_2$ and $A_3$ in the hexagonal case. In particular, using this relation, we find a closed system for $A_1$ and $A_2$ in the square case and for $A_1$, $A_2$ and $A_3$ in the hexagonal case.

Next, we argue that the relevant invariant subspaces remain invariant under higher-order perturbation. This follows from symmetries of the original system, which are preserved by the center manifold reduction. First, the subsets $\{A_2 = 0\}$ and $\{A_2 = A_3 = 0\}$ remain invariant using translation invariance. Second, the subset $\{A_1 = A_2\}$ and $\{A_1 = A_2 = A_3\}$ remain invariant using rotation invariance. Finally, the subset $\{A_2 = A_3\}$ in the hexagonal case remains invariant using reflection symmetry $y \mapsto -y$.

Since all relevant invariant subspaces are persistent and all fixed points are hyperbolic, we find that the phase plane diagrams as discussed in Sections \ref{sec:modulation-center-square} and \ref{sec:modulation-center-hex}, see also Figures \ref{fig:phase-diags-square1}, \ref{fig:phase-diags-square2}, \ref{fig:phase-diags-hex1} and \ref{fig:phase-diags-hex2}, persist. Before we state the main existence theorem for modulating fronts, we note that a heteroclinic orbit from, for example, $S$ to the $T$, corresponds to an invasion of the trivial state by square patterns in the full system \eqref{eq:thin-film-equation}.

\begin{theorem}[Modulating fronts on square lattice]\label{thm:modulating-fronts-square}
    Consider the thin-film system \eqref{eq:thin-film-equation} with $M = M_m^* + \eps^2 M_0$ and fix a front speed $c > 0$. Then, there exists a $\eps_0 > 0$ such that the following holds true for all $0 < \eps < \eps_0$. The thin-film system \eqref{eq:thin-film-equation} has modulating front solutions $(h,\theta)(t,\x) = (1,1) + \Vcal(x-ct, \x)$ satisfying
    \begin{equation*}
        \Vcal(\xi,x,y) = \Vcal(\xi,x+\tfrac{2\pi}{k_m^*},y) = \Vcal(\xi,x,y+\tfrac{2\pi}{k_m^*}),
    \end{equation*}
    which describe, for $M_0 < 0$,
    \begin{itemize}
        \item an invasion of the roll waves by the pure conduction state if $K_0 > 0$;
        \item an invasion of the square patterns by the pure conduction state if $K_0 + K_1 >0$;
        \item an invasion of the roll waves by the square patterns if $K_0 >0$, $K_0 + K_1 >0$ and $K_1 - K_0 > 0$;
        \item an invasion of the square patterns by the roll waves if $K_0 > 0$, $K_0 + K_1 >0$ and $K_1 - K_0 < 0$;
    \end{itemize}
    and, for $M_0 > 0$,
    \begin{itemize}
        \item an invasion of the pure conduction state by roll waves if $K_0 < 0$;
        \item an invasion of the pure conduction state by square patterns if $K_0 + K_1  < 0$;
        \item an invasion of the roll waves by the square patterns if $K_0 < 0$, $K_0 + K_1 < 0$ and $K_1 - K_0 > 0$;
        \item an invasion of the square patterns by the roll waves if $K_0 < 0$, $K_0 + K_1 < 0$ and $K_1 - K_0 < 0$.
    \end{itemize}
\end{theorem}

\begin{remark}\label{rem:modulating-pattern-selection-square}
    We point out that the parameter conditions in Theorem \ref{thm:modulating-fronts-square} are consistent with the parameter regimes for pattern selection in \cite[Eq.~(33)]{shklyaev2012} (note that there $K_0$ and $K_1$ have the opposite sign) and Remark \ref{rem:pattern-selection-square}. In particular, for $M_0 > 0$ square patterns are selected if $K_1 - K_0 < 0$, whereas roll waves are selected if $K_1 - K_0 > 0$. Additionally, Theorem \ref{thm:modulating-fronts-square} and the phase portraits Figures \ref{fig:phase-diags-square1} and \ref{fig:phase-diags-square2} suggest that the selection for $K_1 - K_0 < 0$ is either facilitated via a direct invasion of the pure conduction state by square waves or by a two-stage invasion process, where first the pure conduction state is invaded by the roll waves, which then are invaded by square patterns, see Figure \ref{fig:Modfronts-square}. Analogously, for $K_1 - K_0 > 0$.
\end{remark}

\begin{figure}[H]
    \centering
    \includegraphics[width=\linewidth]{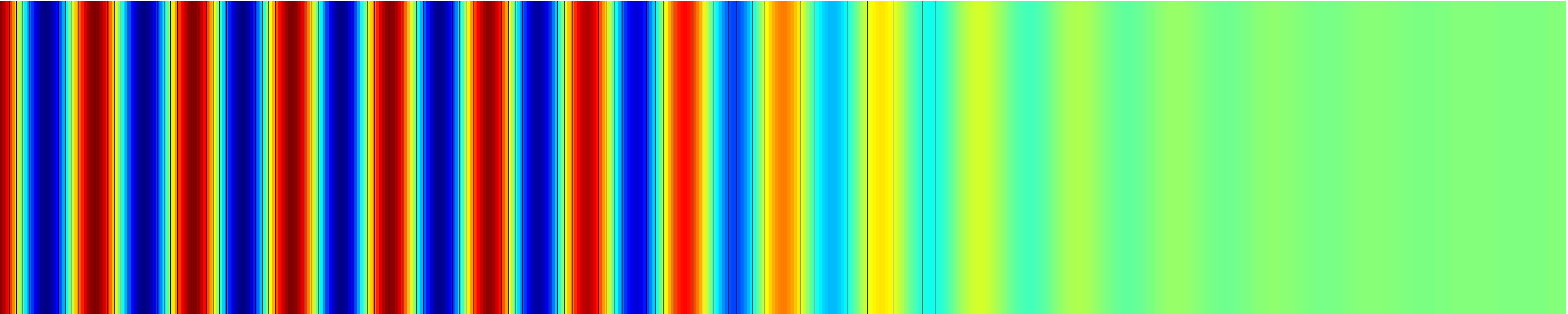}

    \vspace{0.25cm}

    \includegraphics[width=\linewidth]{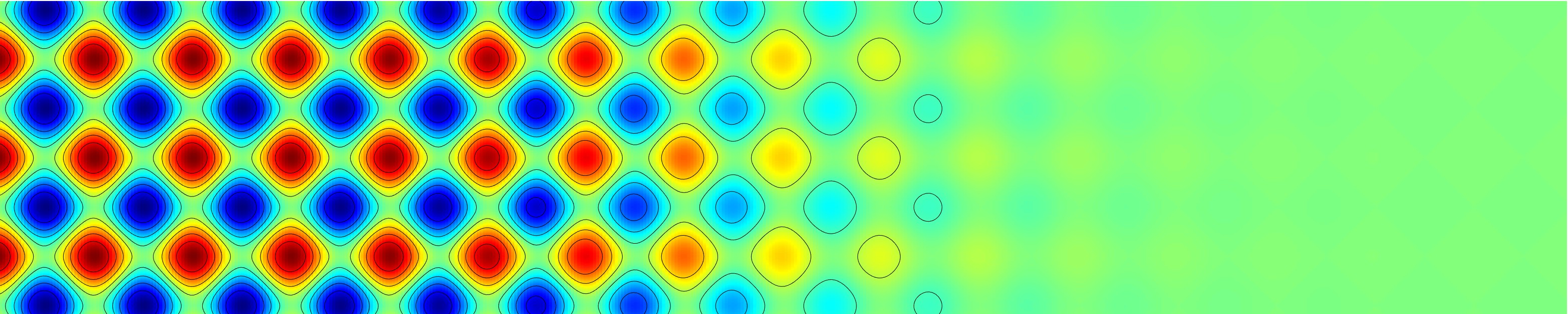}

    \vspace{0.25cm}

    \includegraphics[width=\linewidth]{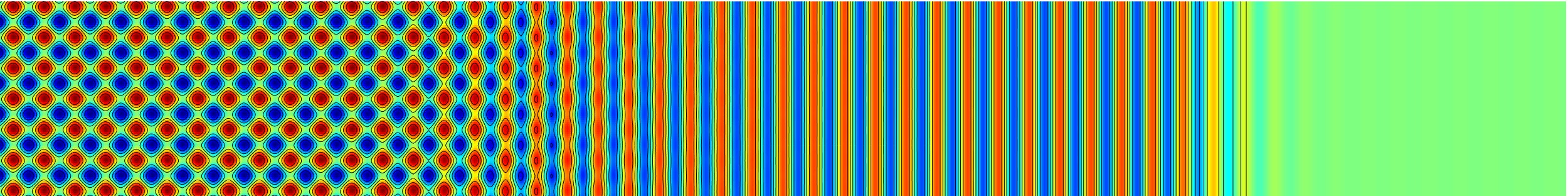} 

    \vspace{0.3cm}
    \begin{tikzpicture}
         \draw[-stealth, very thick] (-2,-0.25) -- node[below]{direction of propagation} (2,-0.25);
    \end{tikzpicture}
    \caption{Examples of moving pattern interfaces on a square lattice. The first image shows an invasion of the pure conduction state by roll waves. The second image is an invasion of the pure conduction state by a square pattern. The last image is a two-stage invasion process consisting of a primary invasion of the pure conduction state by roll waves followed by a secondary invasion of the roll waves by a square pattern.}
    \label{fig:Modfronts-square}
\end{figure}
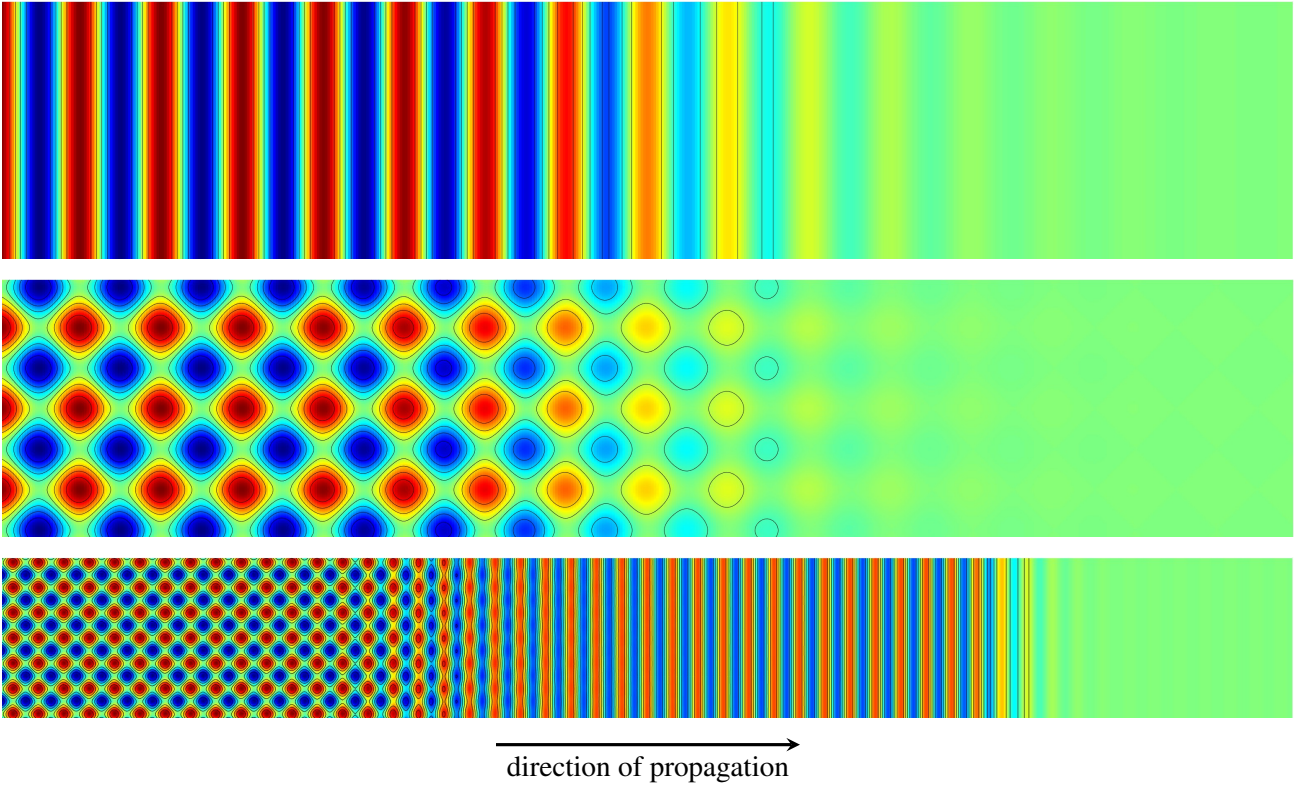

\begin{theorem}[Modulating fronts on hexagonal lattice]\label{thm:modulating-fronts-hex}
    Consider the thin-film system \eqref{eq:thin-film-equation} with $M = M_m^* + \eps^2 M_0$ and fix a front speed $c > 0$. Furthermore, assume that $g \in (10,18)$ and $\beta$ is such that $N(\beta,g) = \eps N_0$ with $N_0 > 0$ in \eqref{eq:leading-order-system-fronts-hex}, $K_0 < 0$ and $K_2 < 0$. Then, there exists a $\eps_0 > 0$ such that the following holds true for all $0 < \eps < \eps_0$. The thin-film system \eqref{eq:thin-film-equation} has modulating front solutions $(h,\theta)(t,\x) = (1,1) + \Vcal(x-ct, \x)$ satisfying
    \begin{equation*}
        \Vcal(\xi,x,y) = \Vcal(\xi,x+\tfrac{2\pi}{k_m^*},y) = \Vcal(\xi,x-\tfrac{\pi}{k_m^*},y+\tfrac{\sqrt{3}\pi}{k_m^*}) = \Vcal(\xi,x-\tfrac{\pi}{k_m^*},y-\tfrac{\sqrt{3}\pi}{k_m^*}),
    \end{equation*}
    which describe, for $M_0 < 0$,
    \begin{itemize}
        \item an invasion of up-hexagons by the pure conduction state;
        \item an invasion of up-hexagons by up-hexagons with a larger amplitude;
    \end{itemize}
    and, for $M_0 > 0$,
    \begin{itemize}
        \item an invasion of the pure conduction state by the roll waves;
        \item an invasion of the pure conduction state by up-hexagons;
        \item an invasion of the pure conduction state by down-hexagons;
        \item an invasion of roll waves by up-hexagons, if $M_0\kappa < -\tfrac{K_0 N_0^2}{(K_0 - K_2)^2}$;
        \item an invasion of down-hexagons by up-hexagons, if $M_0\kappa < -\tfrac{N_0^2(2K_0 + K_2)}{(K_0 - K_2)^2}$;
        \item an invasion of the mixed modes by either roll waves or up-hexagons, if $-\tfrac{K_0 N_0^2}{(K_0 - K_2)^2} < M_0\kappa < -\tfrac{N_0^2(2K_0 + K_2)}{(K_0 - K_2)^2}$;
        \item an invasion of both up- and down-hexagons by the mixed modes
        if $M_0\kappa > -\tfrac{N_0^2(2K_0 + K_2)}{(K_0 - K_2)^2}$;
        \item an invasion of the pure conduction state by the mixed modes if $M_0\kappa > -\tfrac{K_0 N_0^2}{(K_0 - K_2)^2}$.
    \end{itemize}
\end{theorem}

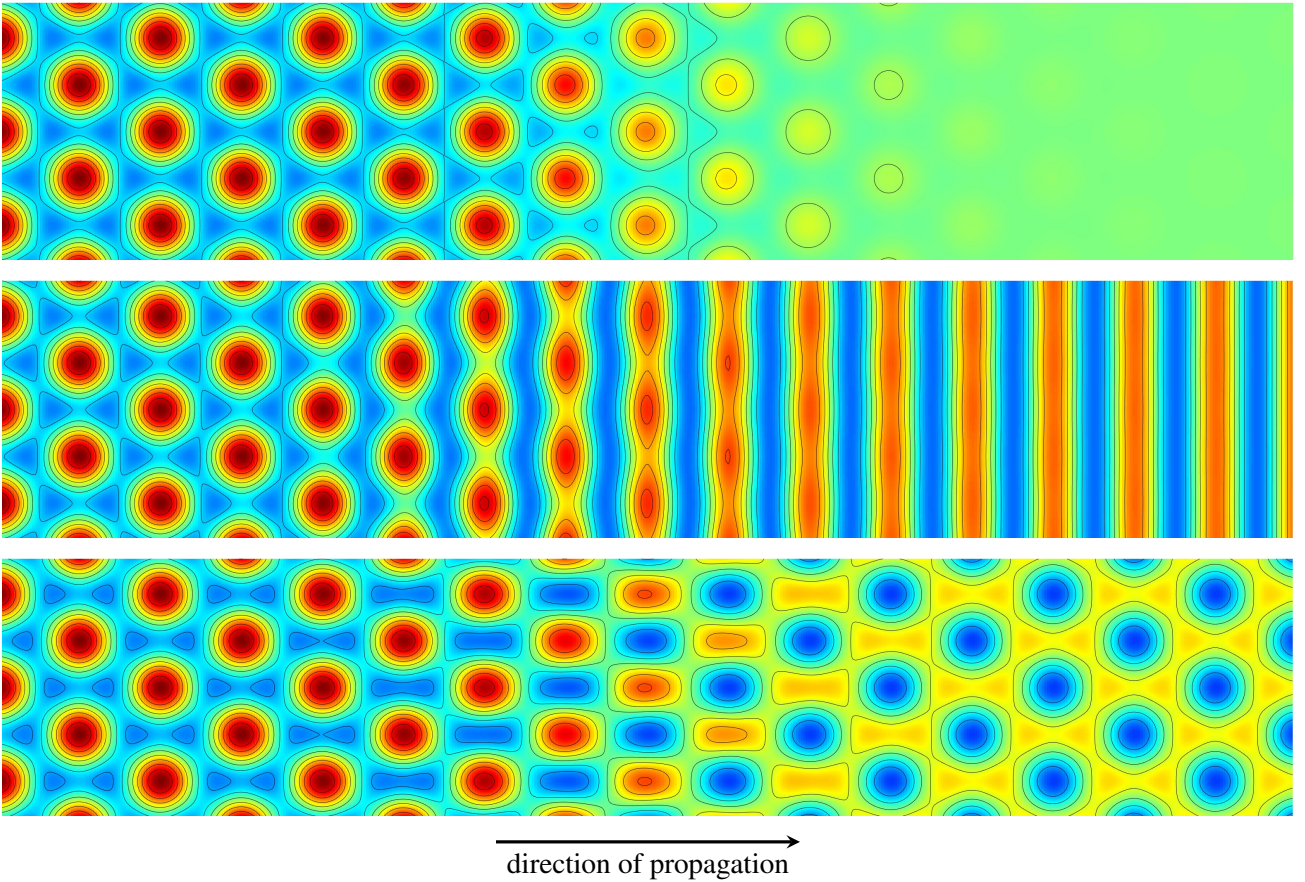
\begin{figure}[H]
    \centering
    \includegraphics[width=\linewidth]{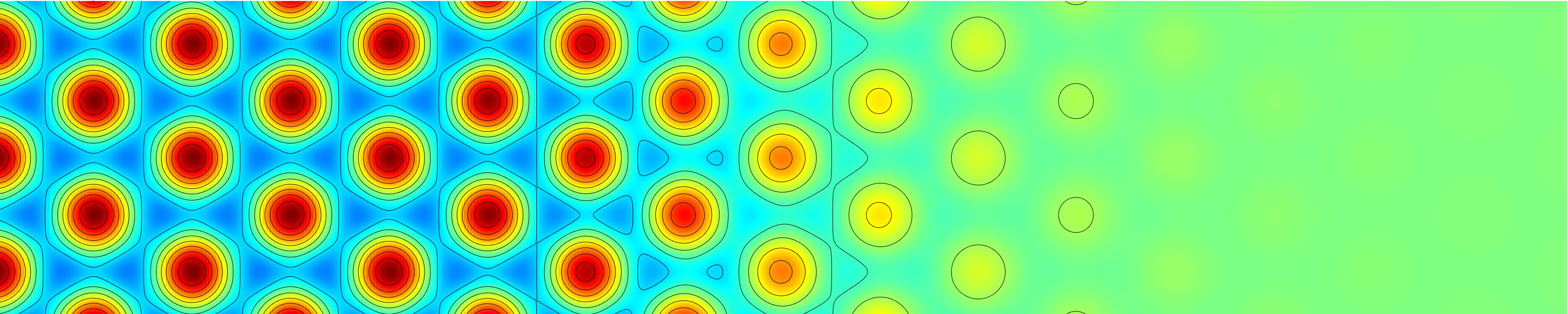}

    \vspace{0.25cm}
    
    \includegraphics[width=\linewidth]{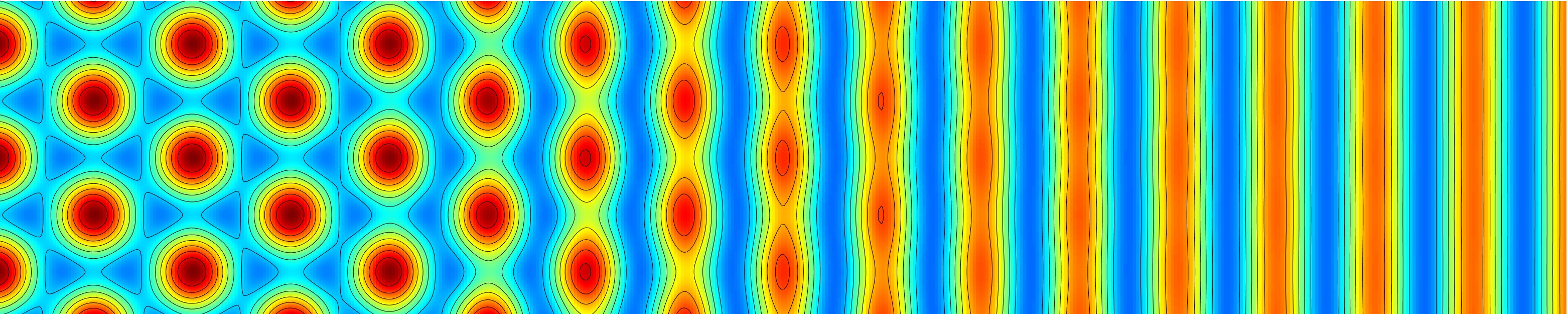}

    \vspace{0.25cm}

    \includegraphics[width=\linewidth]{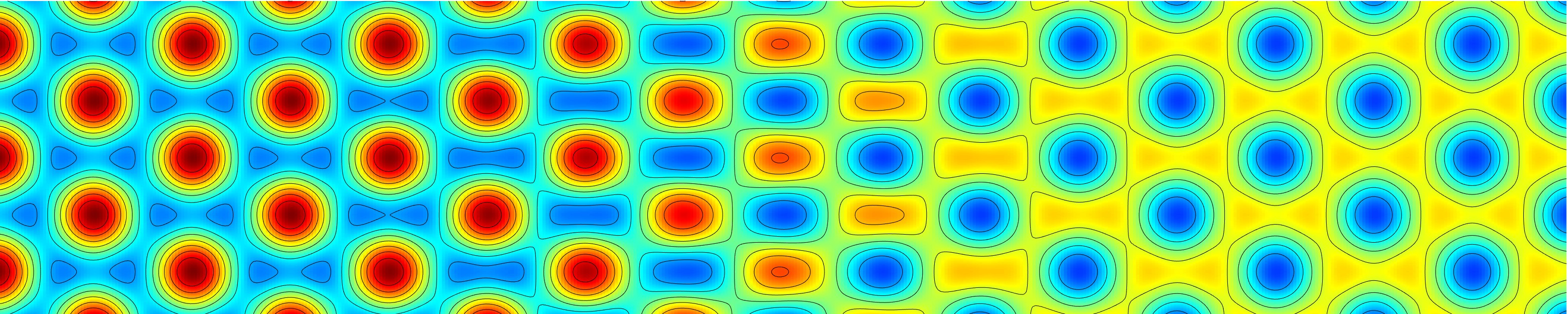}
    
    \vspace{0.3cm}
    \begin{tikzpicture}
         \draw[-stealth, very thick] (-2,-0.25) -- node[below]{direction of propagation} (2,-0.25);
    \end{tikzpicture}
    \caption{Examples of moving pattern interfaces on a hexagonal lattice. The first image shows an invasion of the pure conduction state by a hexagonal pattern. The second image is an invasion of roll waves by a hexagonal pattern. The third image is an invasion of down-hexagons by up-hexagons.}
    \label{fig:Modfronts-hex}
\end{figure}

\begin{remark}\label{rem:modulating-pattern-selection-hex}
    Similar to the previous Remark \ref{rem:modulating-pattern-selection-square}, Theorem \ref{thm:modulating-fronts-hex} and the phase portraits in Figures \ref{fig:phase-diags-hex1} and \ref{fig:phase-diags-hex2} suggest that pattern selection on a hexagonal lattice is either facilitated by a direct invasion or a two-stage invasion process. We note that such a two-stage invasion of a homogeneous ground state by a hexagonal pattern via roll waves has been numerically observed in a damped Kuramoto–Sivashinsky equation in \cite{csahok1999}.
\end{remark}

\begin{figure}[H]
    \includegraphics[width=\linewidth]{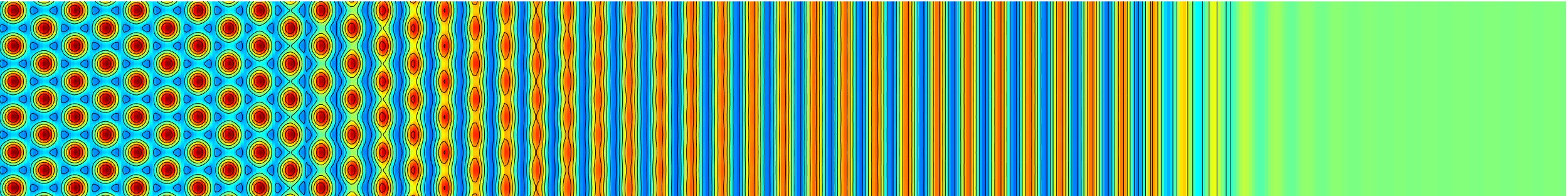}
    
    \vspace{0.3cm}
    \begin{tikzpicture}
         \draw[-stealth, very thick] (-2,-0.25) -- node[below]{direction of propagation} (2,-0.25);
    \end{tikzpicture}
    \caption{Example of a moving pattern interface on a hexagonal lattice. This two-stage invasion process consists of a primary invasion of the pure conduction state by roll waves followed by a secondary invasion of the roll waves by a hexagonal pattern.}
    \label{fig:Modfronts-hex-two}
\end{figure}
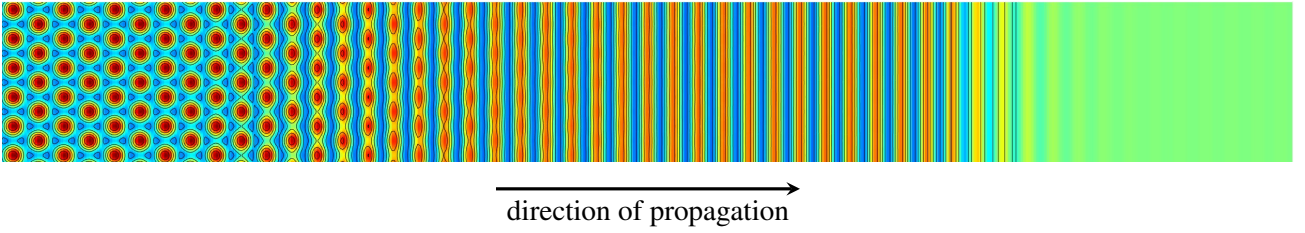

\SkipTocEntry\section*{Acknowledgements}
B.H.~was partially supported by the Swedish Research Council -- grant no.~2020-00440 -- and the Deutsche Forschungsgemeinschaft (DFG, German Research Foundation) -- Project-ID 444753754.

J.J.~acknowledges the support of Lunds Universitet, where major parts of the research were carried out.

This material is also partially based upon work supported by the Swedish Research Council under grant no.~2021-06594 while the authors were in residence at Institut Mittag-Leffler in Djursholm, Sweden, during October 2023.

\SkipTocEntry\section*{Data Availability Statement}

The explicit formulas for the coefficients, which can be found in the Supplement Material, were computed in Mathematica \cite{Mathematica}. The plots of the coefficients, the phase diagrams and the approximate solutions were generated using Python. The code used to generate the corresponding data and videos of selected modulating fronts are available under \cite{Hilder_Data_availability_for_2024}. We also make the Supplementary Material available at \cite{Hilder_Data_availability_for_2024}.

\appendix

\section{Heteroclinic orbits for mixed modes}\label{app:heteroclinic}

We now give a proof of the existence of heteroclinic orbits between roll waves and mixed modes, as well as between hexagons and the mixed modes in Section \ref{sec:modulation-center-hex}. The proof relies on standard arguments for planar dynamical systems, in particular on the Poincaré-Bendixson theorem, see e.g.~\cite[Thm.~6.32]{d.meiss2017}.

\begin{lemma}
    Let $K_0, K_2 < 0$ in \eqref{eq:leading-order-system-fronts-hex}. If $-\tfrac{K_0 N_0^2}{(K_0 - K_2)^2} < M_0 \kappa < -\tfrac{N_0^2(2K_0 + K_2)}{(K_0 - K_2)^2}$ such that the fixed point $M\!\!M$ exists, the system \eqref{eq:leading-order-system-fronts-hex} has heteroclinic orbits from $R$ to $M\!\!M$ and $H_{1,+}$ to $M\!\!M$.
\end{lemma}

\begin{proof} After time-reversal, a sufficiently large ball around $T$ is an inflowing invariant manifold. Additionally, $T$ is an unstable fixed point and $R$ and $H_{1,+}$ are stable ones. Since $M\!\!M$ is a saddle point, it has a one-dimensional unstable and a one-dimensional stable manifold. We argue that the $\omega$-limit set of the unstable manifold $\omega(W^u(M\!\!M))$, which is defined as the union of the $\omega$-limit sets of every point on the unstable manifold, contains the fixed points $R$ and $H_{1,+}$. Assume first that $\omega(W^u(M\!\!M))$ contains neither $R$ nor $H_{1,+}$. Note that $M\!\!M$ is contained in a compact inflowing invariant set, which contains only the fixed points $T$, $R$, $H_{1,+}$ and $M\!\!M$. Since $T$ is unstable, we find that $T \notin \omega(W^u(M\!\!M))$. Additionally, since \eqref{eq:leading-order-system-fronts-hex} has a strictly decreasing Lyapunov function, $W^s(M\!\!M)$ and $W^u(M\!\!M)$ cannot intersect since this would create a homoclinic orbit to $M\!\!M$. In particular, $M\!\!M \notin \omega(W^u(M\!\!M))$. Therefore, since $W^u(M\!\!M)$ is contained in a compact set and $\omega(W^u(M\!\!M))$ does not contain any fixed point, Poincaré--Bendixson, see e.g.\ \cite[Thm.~6.32]{d.meiss2017}, yields that $\omega(W^u(M\!\!M))$ contains a periodic orbit of \eqref{eq:leading-order-system-fronts-hex} -- a contradiction. 

Now, assume that $\omega(W^u(M\!\!M))$ contains either $R$ or $H_{1,+}$, but not both. Without loss of generality we assume that $R \in \omega(W^u(M\!\!M))$. Under this assumption, the closure of $W^u(M\!\!M)$ is a closed curve, and its interior is a compact, invariant set $\Ccal$. Since $W^s(M\!\!M)$ is tangential to the stable eigenspace, there is a point $p \in W^s(M\!\!M) \cap \Ccal$. In particular, its backward orbit is contained in a compact set. Since $R$ is a stable fixed point and there are no homoclinic orbits, the $\omega$-limit set of the backward orbit of $p$ does not contain any fixed point. Applying Poincaré-Bendixon, therefore leads to the existence of a periodic orbit and, thus, a contradiction. Hence $\omega(W^u(M\!\!M))$ contains both $R$ and $H_{1,+}$. Since both $R$ and $H_{1,+}$ are stable there cannot be a heteroclinic connection between these fixed points. Therefore, $\omega(W^u(M\!\!M)) = \{R, H_{1,+}\}$, see \cite[Thm.~2.4.11]{schneider2017} or \cite[Thm.~6.41]{d.meiss2017}. This proves that the claimed heteroclinic connections exist.
\end{proof}

\printbibliography

\end{document}